\documentclass{amsart}

\usepackage[usenames,dvipsnames]{xcolor} 

\usepackage[T1]{fontenc}
\usepackage{amsmath,amsthm,indentfirst}
\usepackage{amssymb}
\usepackage{amsfonts,ifpdf}

\ifpdf
\usepackage[pdftex,pdfstartview=FitH,pdfborderstyle={/S/B/W 1}]{hyperref}
\else
\usepackage[hypertex]{hyperref}
 \fi
\usepackage{amscd}
\usepackage[latin1]{inputenc}
\DeclareMathAlphabet{\mathpzc}{OT1}{pzc}{m}{it}
\DeclareFontFamily{OT1}{rsfs}{}
\DeclareFontShape{OT1}{rsfs}{n}{it}{<-> rsfs10}{}
\DeclareMathAlphabet{\mathscr}{T1}{rsfs}{n}{it}
\numberwithin{equation}{section}

\theoremstyle{plain}
\newtheorem*{maintheorem*}{Main Theorem}
\newtheorem*{thm*}{Theorem}
\newtheorem*{remark*}{Remark}
\newtheorem*{conjecture*}{Conjecture}
\newtheorem*{prop*}{Proposition}
\newtheorem*{lem*}{Basic Lemma}
\newtheorem{thm}{Theorem}[section]
\newtheorem{cor}[thm]{Corollary}
\newtheorem{lem}[thm]{Lemma}
\newtheorem{prop}[thm]{Proposition}

\newtheorem*{thma1*}{Theorem A.1}
\newtheorem*{thma2*}{Theorem A.2}
\newtheorem*{thmb*}{Theorem B}
\newtheorem*{thmd*}{Theorem D}
\newtheorem*{propc*}{Proposition C}

\theoremstyle{definition}

\newtheorem*{proofc*}{Proof of Theorem C}

\newtheorem{definition}[thm]{Definition}

\newtheorem*{def*}{Definition}

\def\bbz{\mathbb{Z}}

\def\bbf{\mathbb{F}}

\def\Mbf{\mathbb{A}}

\def\bbn{\mathbb{N}}
\def\Gbf{\mathbb{G}}

\def\bbp{\mathbf{P}}

\def\bbw{\mathbf{W}}

\def\Qcal{\mathcal{Q}}

\def\Ocal{{\mathcal O}}

\def\rcal{\mathcal{R}}

\def\Hcal{\mathcal{H}}

\def\gfrak{\mathfrak{g}}
\def\Sfrak{\mathfrak{S}}

\def\Mbf{{\mathbf M}}

\def\Gbf{{\mathbf G}}

\def\G{\Gbf}
\def\Hbf{{\mathbf H}}
\def\H{\Hbf}

\def\fele{{\mathpzc{c}}}

\def\SL{{\rm{SL}}}
\def\PGL{{\rm{PGL}}}

\def\GL{{\rm{GL}}}
\def\Lie{{\rm Lie}}

\def\h{\hspace{1mm}}

\def\la{\lambda}

\def\vare{\varepsilon}
\def\cpct{X_{{\rm{cpct}}}}

\def\zg0{Z_{G_{w}}(s)}

\def\zg{Z_G(s)}

\def\be{\begin{equation}}
\def\ee{\end{equation}}

\def\gfield{{K}}
\def\field{{k}}
\def\sfield{{l}}
\def\order{{\mathfrak o}}
\def\vorder{{\mathfrak o_v}}
\def\vintegers{{\mathcal O}}

\def\pins{{\mathcal{P}}}

\def\pinsker{\pins}
\def\rootsys{{\Phi}}
\newcommand\supp{\operatorname{supp}}
\def\lsupp{\mathcal{S}}
\def\linv{\mathcal{I}}
\def\ppor{{\;\propto\;}}

\def\unif{{\theta}}
\def\resor{{q}}
\def\pchar{{p}}
\def\entr{\mathsf{h}}
\def\entropy{\entr}

\def\Uef{{W^+_G(a_\alpha)}}

\def\Cef{{Z_{\alpha}}}
\def\nbhd{{\mathsf B}}
\def\fele{{r}}

\def\identity{{\rm id}}
\def\id{\identity}
\def\cpct{{\mathsf Q}}
\def\salpha{{s_\alpha}}
\def\Calpha{Z_\alpha}
\def\Xinv{X_{\rm inv}}
\def\NM{{\bf F}}
\def\NMM{{\bf F}}
\def\RNM{{\bf H}'}
\def\APhi{\overline{\Phi}}

\begin{document}

\title{Diagonal actions in positive characteristic}

\author{M. Einsiedler}
\address{M.E.:  Departement Mathematik, ETH Z\"urich, R\"amistrasse 101, CH-8092 Z\"urich, Switzerland}
\email{manfred.einsiedler@math.ethz.ch}
\thanks{M.E.\ acknowledges support by the SNF (Grant 152819 and 178958).}

\author{E. Lindenstrauss}
\address{E.L.: The Einstein Institute of Mathematics, Edmond J.\ Safra Campus, 
Givat Ram, The Hebrew University of Jerusalem, Jerusalem, 91904, Israel}
\email{elon@math.huji.ac.il}
\thanks{E.L.\ acknowledges support by the ERC (AdG Grant 267259).}

\author{A. Mohammadi}
\address{A.M.: Department of Mathematics, University of California, San Diego, CA 92093}
\email{ammohammadi@ucsd.edu}
\thanks{A.M.\ acknowledges support by the NSF (DMS 1724316, 1764246, and 1128155) and Alfred P.~Sloan Research Fellowship.}

\begin{abstract}
We prove positive characteristic analogues of certain measure rigidity
theorems in characteristic zero. More specifically we give a classification
result for positive entropy measures on quotients of $\operatorname{SL}_d$
and a classification of joinings for higher rank actions on simply connected
absolutely almost simple groups. 
\end{abstract}

\maketitle

\section{Introduction}\label{sec:intro}
Let $G$ be a locally compact, second countable group
and let $\Gamma$ be a lattice in $G.$ Put $X=G/\Gamma.$ 
A subset $S\subset X$ is called {\em homogeneous} if  
there exists a closed subgroup $\Sigma<G$ and some $x\in X$
such that $\Sigma x$ is closed and supports a $\Sigma$-invariant
probability measure. 
A probability measure $\mu$ on $X$ is called {\em homogeneous}
if $\supp\mu$ is homogeneous and $\mu$ is the $\Sigma$-invariant
probability measure on $\supp\mu$.

Let $A$ be a closed abelian subgroup of $G$.
An $A$-invariant probability measure $\mu$ on $G/\Gamma$ 
will be called {\em almost homogeneous} if 
\begin{equation}\label{eq;almosthomogeneous}
\mu=\int_{A/A\cap\Sigma}a_*\nu \operatorname{d}\!a
\end{equation}
where
\begin{enumerate}
\item $\Sigma\subset G$ is a closed subgroup such that $A/A\cap\Sigma$ is compact,
\item $\nu$ is a homogeneous measure stabilized by $\Sigma$,
\item $\operatorname{d}\!a$ is the Haar probability measure on the group $A/A\cap\Sigma.$
\end{enumerate}

Let $\gfield$ be a global function field, i.e.\ a finite extension of the field of rational functions in one variable over a finite field $\bbf_\pchar.$ For any place $w$ of $\gfield$ we let $\gfield_w$ denote the completion of $\gfield$ at $w,$ and let $\order_w$ be the ring of integers in 
$\gfield_w.$ As in the case of number fields, the field $\gfield$ embeds diagonally in the restricted product $\prod'_w\gfield_w.$ Given a place $v$ we put
\[
\Ocal_v=\gfield\cap\prod_{w\neq v}\order_w
\]
to be the ring of $v$-integers in $\gfield.$

For the rest of this paper we will assume that a place $v$ of $\gfield$ is fixed
and put 
\[
\field:=\gfield_v,\;\order:=\vorder, \text{ and }\;\vintegers:=\Ocal_v.
\]
Recall that we may and will identify $\field$ with
$\bbf_\resor((\unif^{-1})),$ the field of Laurent series over the finite field $\bbf_\resor,$
after this identification we have $\order=\bbf_\resor[[\unif^{-1}]],$~\cite[Ch.~1]{Weil-BasicNum}.

The most familiar case is when $K=\bbf_\resor(\theta)$, the field rational functions 
in one variable with coefficients in $\bbf_\resor.$ 
Then if we choose the valuation $v$ coming from $\theta^{-1}$
we have $\Ocal_v=\bbf_\resor[\theta]$ is the polynomial ring.

\subsection{Positive entropy classification for measures on quotients of \texorpdfstring{$\SL_d$}{SL\_d}}

Let $G=\SL(d,\field)$ and let $\Gamma<G$ be an {\em inner type} lattice in $G$. See \S\ref{sec:innerforms} for the definition and discussion of inner type lattices; as an explicit example the reader may let $\Gamma=\SL(d, \Ocal)$. Let $X:=G/\Gamma.$ 
Furthermore, we let~$A$ be the full diagonal 
subgroup of $\SL(d,\field)$.
{\it Throughout the paper we always assume $d>2.$}

Given an $A$-invariant probability measure $\mu$ we let $\entropy_\mu(a)$ denote the measure theoretic entropy of $a\in A.$  

\begin{thm}\label{thm:measure-class}
Suppose $\mu$ is an $A$-invariant ergodic probability measure on $X,$ further
assume that $\entropy_\mu(a)>0$ for some $a\in A.$ Then $\mu$ is almost homogeneous. 
\end{thm}

We note that this is a positive characteristic analogue of the result of \cite{EKL}
by A.~Katok with the two first named authors.

The conclusion of Theorem~\ref{thm:measure-class}
cannot be strenghtened to saying that~$\mu$ is homogeneous. In fact,~$K=\bbf_\resor(\unif)$
has many subfields~$K'$ (without a bound on~$[K:K']$), 
defining~$\field'$ to be the closure of~$K'$ in~$\field$,
one could take the measure~$\nu$ to be the Haar measure
on the closed orbit~$\Sigma\Gamma$ for~$\Sigma=\SL(d,\field')$,
and~$\mu$ could be as in~\eqref{eq;almosthomogeneous}
since~$A/(A\cap\Sigma)$ is compact.

\subsection{Joining classification}\label{sec:joining-statement}
Furstenberg \cite{Furstenberg-disjointness}
introduced the following notion in 1967 that has since become 
a central tool in ergodic theory. 
Suppose we are given two measure preserving systems for a group $S$ acting on Borel
probability spaces $(X_i,m_i)$ for $i=1,2$. A {\em joining} is a Borel
probability measure $\mu$ on $X_1\times X_2$ such that the push-forwards
satisfy $(\pi_i)_*\mu=m_i$ for $i=1,2$ and is invariant under 
the diagonal action on $X_1\times X_2$,
i.e.\ $s.(x_1,x_2)=(s.x_1,s.x_2)$ for all $s\in S$ and $(x_1,x_2)\in X_1\times X_2$.

We give a classification of ergodic joinings in the following setting.
Let $\Gbf_i$ be connected, simply connected, absolutely almost simple groups defined 
over $\field$ for $i=1,2$. 
Put $G_i=\Gbf_i(\field)$ and let $\Gamma_i$ be a lattice in $G_i$
and define $X_i=G_i/\Gamma_i$ for $i=1,2.$ 
Denote by $m_i$ the Haar measure on $X_i.$

Let $\la_i:{\bf G}_m^2\to\Gbf_i$ be two algebraic homomorphisms with finite kernel defined over $\field,$
and put $\mathbf A_i=\la_i(\Gbf_m^2).$ 
We define the notion of joining as above using these monomorphisms.

Let $\mathbf{A}=\{(\la_1(t),\la_2(t)):t\in\Gbf_m^2\}$ and let $A=\mathbf{A}(\field)$.

\begin{thm}\label{thm:joining}
Assume ${\rm char}(\field)\neq2,3.$
Suppose that $\Gbf_i,$ $\mathbf{A}_i,$ and $X_i$ are as above for $i=1,2$. 
Let $\mu$ be an ergodic joining of the action
of $A_i$ on $(X_i,m_i)$ for $i=1,2$.  Then $\mu$
is an algebraic joining. That is, one of the following holds
\begin{enumerate}
\item $\mu=m_1\times m_2$ is the trivial joining, or
\item $\mu$ is almost homogeneous, moreover, the group $\Sigma$ appearing in the definition
of an almost homogeneous measure satisfies the following 
\begin{itemize}  
\item $\pi_{i}(\Sigma)=G_i$ for $i=1,2,$ and
\item $\ker(\pi_i|_{\Sigma})$ is contained in the finite group $Z(G_1\times G_2)$ for $i=1,2.$
\end{itemize}  
\end{enumerate}   
\end{thm}

These results are a positive characteristic analogue of the work \cite{EL-Joining} (see also \cite{EL-Joining2} for stronger results in the zero characteristic setting) of the two first named authors.

It is also worth mentioning that even for joinings,  in general, virtual homogeneity can not be improved to homogeneity.
Indeed, let $\field/\field'$ be a Galois extension of degree 2 with the nontrivial
Galois automorphism $\tau.$ Let $\mathbf G_1=\mathbf G_2=\SL_3$ and let  
$\Gamma_1=\Gamma$ and $\Gamma_2=\tau(\Gamma)$ for a lattice $\Gamma\subset \SL(3,\field).$
Let $\lambda_1=\lambda_2$ be the monomorphism
$(t,s)\mapsto{\rm diag}\Bigl(t,s,(ts)^{-1}\Bigr).$ 
The measure~$\nu$ could be the Haar measure
on the closed orbit $\Sigma(\Gamma_1\times\Gamma_2)$
of $\Sigma=\{(g,\tau(g)): g\in \SL(3,\field)\}$ and $\mu$ could be as in~\eqref{eq;almosthomogeneous}.    

\subsection{Main difference to the zero characteristic setting}

We apply in this paper the high entropy method that was developped 
in the zero characteristic setting in a series of papers, see \cite{EK-PureApplied,EKL,EK-2,EL-Pisa},
and for Theorem~\ref{thm:measure-class} also the low entropy method, see \cite{Elon-Quantum, EKL, EL-GenLow}.
These arguments use crucially leaf-wise measures for the root subgroups (or
more generally the coarse Lyapunov subgroups), which are locally finite measures
on unipotent subgroups.

Suppose we are able using the above tools to show the leafwise measures on the coarse Lyapunov subgroups have some invariance. Then using Poincare recurrence along $A$ one can show the invariance group has arbitrarily large and arbitrarily small elements.
The key difference lies in the next step of the argument.
In the zero characteristic setting
a closed subgroup of a unipotent group containing arbitrarily small
and arbitrarily large elements has to contain a one-parameter subgroup
-- and hence the leafwise measures for the one-parameter subgroup have to be Haar which gives unipotent invariance for the measure under consideration.

In the positive characteristic world this is very far from being true.
In fact using a fairly direct adaptation of the methods used in \cite{EKL,EL-Joining} etc.\ one can find almost surely an unbounded subgroup of a unipotent group
that has positive Hausdorff dimension which again preserves the leafwise measure. 
However, as there are uncountably many such subgroups
and since these may vary from one point to another it is not clear how to continue
from this by purely dynamical methods. 

Decomposing the measure~$\mu$
according to the Pinsker~$\sigma$-algebra~$\pins_a$ (for some~$a\in A$) 
we find a subgroup of~$G$ that preserves
the conditional measure on an atom for~$\pins_a$ 
 and has a semisimple Zariski closure. To classify such subgroups we use a result of Pink \cite{Pink-Compact} (see also \cite{LarPink} for related results by Larsen and Pink).
This allows us to deduce invariance under the group of points of a semisimple subgroup
for some local subfield.  After this we use
a measure classification result \cite{MS-SL2} by Alireza Salehi-Golsefidy
and the third named author (as a replacement of Ratner's measure classification theorem
\cite{Ratner:1990p609,Ratner-measure} extended to the $S$-arithmetic setting by Ratner \cite{Ratner-measure} resp. Margulis and Tomanov \cite{MarTom}). 

We note that analogues of the measure rigidity theorems of Ratner for general unipotent flows in 
positive characteristic setting are not yet known. Some special cases 
have been investigated, specifically the above mentioned paper~\cite{MS-SL2} which is used in our proof and the earlier paper~\cite{EG}.  

{{} Finally we note that ideally one would like to have a result similar to~\cite{EL-SplitCase} in the setting at hand.
A general treatment as in~\cite{EL-SplitCase} will likely require more subtle algebraic considerations.  
}

\subsection*{Acknowledgements}
The authors would like to thank Alireza Salehi Golsefidy, Michael Larsen, {Shahar Mozes,} Gopal Prasad, and Richard Pink for helpful conversations.
The results of \cite{MS-SL2} are used in our work in an essential way, and we thank Alireza Salehi Golsefidy for agreeing to present the results in that paper in a way that would be convenient for our purposes.
{}{We would also like to thank the anonymous referees for their helpful comments.} 

{}
 

\section{Notation}\label{sec:notation}

\subsection{}\label{sec:field}
Throughout, $\gfield$ denotes a global function field. 
We let $v$ be a place in $\gfield,$ fixed once and for all.
Denote by $\vintegers$ the ring of $v$-integers in $\gfield.$
Put $\field:=\gfield_v$ the completion of $\gfield$ at $v.$ 
Then $\field$ is identified 
with $\bbf_\resor((\unif^{-1})),$ the field of Laurent series 
over the finite field $\bbf_\resor$ where $\resor$ is a power of the prime number $\pchar={\rm char}(\gfield).$
We denote by $\order$ the ring of integers in $\field.$
Then $\order=\bbf_\resor[[\unif^{-1}]]$ and the maximal ideal $\mathfrak m$ in $\order$ equals $\unif^{-1}\order.$ The norm on $k$ will be denoted by $|\cdot|_v,$ or simply by $|\cdot|;$
note that with our notation we have $|\theta|_v>1.$ 
With our normalizations $\log_q(|r|)$ is the $v$-valuation of $r\in\field.$   

Unless explicitly mentioned otherwise, a subfield $\field'\subset\field$ is 
always an infinite and closed subfield of $\field;$ hence, $\field/\field'$ is a finite 
extension.

\subsection{}\label{sec:group} 
Let $\Gbf$ be a connected, simply connected, semisimple $\field$-algebraic group.
Put $G=\Gbf(\field).$ We always assume $\Gbf$ is $\field$-isotropic.

Fix a maximal $\field$-split $\field$-torus ${\bf S}$ of $\Gbf.$
{}{We will always assume that ${\bf A}={\bf S}$ in the case of Theorem~\ref{thm:measure-class},
and assume that ${\bf A}_i$ is contained in ${\bf S}_i,$ for $i=1,2,$ in the case of Theorem~\ref{thm:joining}.}

Let ${}_k\rootsys$ denote the set of relative roots ${}_\field\rootsys({\bf S},\Gbf)$; 
this is a (possibly not reduced) root system, see~\cite[Thm.~21.6]{Borel-AlgGrBook}. 
Let ${}_k\rootsys^\pm$ denote positive and negative roots with respect to a fixed ordering on ${}_k\rootsys$.

Recall from~\cite[Remark~2.17, Prop.\ 21.9, and Thm.\ 21.20]{Borel-AlgGrBook} that
for any $\alpha\in{}_k\rootsys$ there exists a unique affine $\field$-split unipotent 
$\field$-subgroup ${\bf U}_{(\alpha)}$ which is normalized by ${\bf Z}_{\Gbf}({\bf S}),$ 
the centralizer of ${\bf S},$
and its Lie algebra is $\gfrak_{(\alpha)}:=\gfrak_\alpha+\gfrak_{2\alpha}.$
Here, as usual, for a root $\beta\in{}_k\rootsys$ we let $\gfrak_{\beta}$ be the subspace
in the Lie algebra on which ${\bf S}$ acts by the root $\beta.$

A subset $\Psi\subset{}_k\rootsys$ is said to be {\em closed} if
$\alpha\in\Psi$ and~$\frac12\alpha\in\Phi$ implies $\frac12\alpha\in\Psi$, and
$\alpha,\beta\in\Psi$
with $\alpha+\beta\in{}_k\rootsys$ implies $\alpha+\beta\in\Psi$.  
A subset $\Psi \subset {}_k\rootsys$ is said to be \emph{positively closed} if it is closed and is contained in ${}_k\rootsys^+$ for some ordering of the root system.
For any positively closed subset $\Psi\subset{}_k\rootsys$
there exists a unique affine $\field$-split unipotent 
$\field$-subgroup ${\bf U}_{\Psi}$ which is normalized by ${\bf Z}_{\Gbf}({\bf S})$ 
and its Lie algebra is the sum of $\{\gfrak_{(\alpha)}:\alpha\in\Psi\}.$
Moreover, ${\bf U}_\Psi$ is generated by $\{{\bf U}_{(\alpha)}:\alpha\in\Psi\setminus 2\Psi\},$ i.e.\ ${\bf U}_\Psi$ is $\field$-isomorphic as a $\field$-variety to $\prod_{\alpha\in\Psi\setminus 2\Psi}{\bf U}_{(\alpha)}$ where the product can be taken 
in any order,~\cite[Prop.\ 21.9 and Thm.\ 21.20]{Borel-AlgGrBook}.

If $\Psi=\{\alpha\}$ and no multiple of $\alpha$ is a root, 
we simply write ${\bf U}_\alpha$ for ${\bf U}_\Psi.$
{}{We also write $U_\Psi={\bf U}_\Psi(\field)$ for a positively closed subset $\Psi\subset{}_k\rootsys$.}

Given a subset $E\subset G$ we let $\langle E\rangle$ denote the closed (in the Hausdorff topology) group generated by $E.$

For each $\alpha\in{}_k\rootsys$ we fix a collection of one parameter subgroups, $\{u_{\alpha,i}: 1\leq i\leq d_\alpha\}$ 
generating $U_{(\alpha)}$
and define $U_{(\alpha)}[R]$ to be the compact group generated by 
$\{u_{\alpha,i}( \fele ):|\fele|_v<R, 1\leq i\leq d_\alpha\}.$
For any positively closed $\Psi\subset\Phi$ we put 
\[
U_\Psi[R]=\langle \{U_{(\alpha)}[R]:(\alpha)\subset \Psi\}\rangle
\]

Given $a\in A$ we put 
\be\label{eq:W-pm}
W_G^\pm( a)=\{g\in G:\textstyle\lim_{k\to\pm\infty}a^{-k}ga^{k}=\id\}
\ee
to be the expanding (resp.\ contracting) horospherical subgroup corresponding to $a.$ 

\subsection{}\label{sec:maximaltorusTandabsroots}
Let ${}_\field\Phi({\bf A},\Gbf)$ denote the set of roots of ${\bf A},$ i.e., the characters 
for the adjoint action of ${\bf A}$ on the Lie algebra of $\Gbf.$
We will say $\Psi\subset{}_\field\rootsys({\bf A},\Gbf)$ is positively closed if 
\begin{equation}\label{eq:vartheta to be}
 \{\alpha \in \Phi({\bf S},\Gbf):\alpha|A \in \Psi\}
\end{equation}
is positively closed in the sense of \S\ref{sec:group},
and set
\[
V_\Psi:=\prod_{\alpha|_A \in \Psi}U_{(\alpha)}
\] 
for any positively closed subset $\Psi\subset{}_\field\Phi({\bf A},\Gbf)$.
We also let ${\bf V}_{\Psi}$ denote the underlying algebraic group.
An important special case is when $\Psi=[\alpha]=\{r\alpha\in{}_\field\rootsys({\bf A},\Gbf):r>0\}$
for some $\alpha\in{}_\field\Phi({\bf A},\Gbf)$.
In this case $V_{[\alpha]}$ is called a coarse Lyapunov subgroup.

\subsection{Inner type lattices in $\SL(d,\field)$}\label{sec:innerforms}
Recall that in Theorem~\ref{thm:measure-class} we assumed $\Gamma$ is an inner type lattice in  $\SL(d,\field)$;
we recall the definition here.  
Let $D$ be a division algebra of dimension $s^2$ over $\gfield$ and let $B={\rm Mat}_r(D)$ be a central simple algebra over $\gfield$; we assume $d=rs$.

Let $\Omega$ be any field extension of $\gfield$ so that $B\otimes_{\gfield}\Omega\simeq{\rm Mat}_d(\Omega)$ -- one can always find a finite separable extension of $\gfield$ with this property. 
Define the {\em reduced norm} ${\rm Nrd}_B:B\to\Omega$ of $B$ by ${\rm Nrd}_B(g):=\det(g\otimes 1)$.
Then ${\rm Nrd}_B(g)\in\gfield$ for all $g\in B$ and ${\rm Nrd}_B(g)$ is independent of the choice of the splitting field 
$\Omega$ and the implicit isomorphism which we fixed. 
More generally
\be\label{eq:char-poly-inner}
\det(g\otimes 1-\xi\id)\in\gfield[\xi]\quad\text{for every $g\in B$},
\ee
see e.g.~\cite[\S22]{draxl_1983}. 


We now use $B$ to define a $\gfield$-group which is isomorphic to $\SL_d$ over the algebraic closure $\bar\gfield$ of $\gfield$.
Fix a $\gfield$-basis $\mathcal C$ for $D$ and consider the (left) regular representation $\rho$ of $D$ into ${\rm Mat}_{s^2}(\gfield)$, i.e. $g\in D$
is sent to the matrix corresponding to $y\mapsto gy$. 
If we express $\rho$ in the basis $\mathcal C$, we get a system $\{f_\ell(g_{ij})=0\}$ of 
{\em linear} equations in entries $g_{ij}$ with coefficients in $\gfield$ that together define the image of $\rho$.
We identify ${\rm Mat}_{rs^2}(\gfield)$ with ${\rm Mat}_{r}({\rm Mat}_{s^2}(\gfield))$ and let $B'$ 
be the subset of ${\rm Mat}_{rs^2}(\gfield)$ consisting of elements $g_{ij}^{cd}$ for $1\leq i,j\leq s^2$ and $1\leq c,d\leq r$
satisfying $\{f_\ell(g_{ij}^{cd})=0\}$ for all $1\leq c,d\leq r$. Then $\rho$ identifies $B$ and $B'$.
Moreover, in view of the above discussion on ${\rm Nrd}_B$ there exists a polynomial $h$
with coefficients in $\gfield$ so that ${\rm Nrd}_B(g)=h(\rho(g^{cd}))$ for all $g\in B$, see~\cite{Tomanov2002} and~\cite[Ch.~2]{platonov1993algebraic} for a similar discussion and construction. 

For any $\gfield$-algebra $\Upsilon$ define 
\[
\SL_{1,B}(\Upsilon):=\{g\in{\rm Mat}_{rs^2}(\Upsilon):f_\ell(g_{ij}^{cd})=0, h(g_{ij}^{cd})=1\}.
\] 
If $\Omega$ is any field extension of $\gfield$ so that $B\otimes_{\gfield}\Omega\simeq{\rm Mat}_d(\Omega)$,
then $\SL_{1,B}(\Omega)$ is isomorphic 
to $\SL(d,\Omega)$. In particular, $\SL_{1,B}(\bar\gfield)$ is isomorphic 
to $\SL(d,\bar\gfield)$. A group so defined is called an {\em inner $\gfield$-form} of $\SL_d$.

Assume now that $B$ is a a central simple algebra over 
$\gfield$ as above; further, assume that it satisfies $B\otimes_{\gfield}\field\simeq{\rm Mat}_d(\field)$.
For every place $w$ of $\gfield$ define 
\[
\SL_{1,B}(\order_w):=\SL_{1,B}(\gfield_w)\cap {\rm GL}_{rs^2}(\order_w).
\]
Recall that $\SL_{1,B}(\gfield)$ diagonally embeds in the restricted (with respect to 
$\SL_{1,B}(\order_w)$) product $\prod'_{w}\SL_{1,B}(\gfield_w)$. 
Put 
\be\label{eq:Gamma-B}
\Lambda_B=\{\gamma\in\SL_{1,B}(\gfield): \gamma\in\SL_{1,B}(\order_w)\text{ for all $w\neq v$}\}
\ee
Then $\Lambda_B$ is a lattice in $\SL(d,\field)$, see e.g.~\cite[Ch.~I, \S3]{Margulis-Book}.

We will call a subgroup $\Gamma<\SL(d,\field)$ a lattice of {\em inner type} if there exists a central simple algebra $B$ over $\gfield$ 
so that $\Gamma$ is commensurable to $\Lambda_B$.

\section{Preliminary results}\label{sec:prelim-lem}
\subsection{Algebraic structure of compact subgroups of semisimple groups}\label{sec:hilbert-LarPin1}

Given a variety $\Mbf$ which is defined over
$k$ there are two topologies on $\Mbf(k),$ the set of $k$-points of $\Mbf.$ Namely, the Zariski topology and the topology arising from the local field $k.$ We will refer to the latter as the Hausdorff topology.

The following theorems are very special cases of the work of Pink, \cite{Pink-Compact},
and will play an important role in our study.
Roughly speaking, they assert that compact and Zariski dense subgroups of semsimple groups 
have an algebraic description.

\begin{thma1*}[Cf.~\cite{Pink-Compact}, Theorems 0.2 and 7.2]
Suppose $\Qcal\subset\SL(2,\field)$ is a compact and Zariski dense subgroup.
Further, assume that 
\be\label{eq:gen-unip}
\Qcal=\langle \{g\in\Qcal:\text{$g$ is a unipotent element}\}\rangle,
\ee
Let
$
\field''$ be the closed field of quotients generated by $\{{\rm tr}(\rho(g)):g\in\mathcal Q\}$,
where $\rho$ is the unique irreducible subquotient of the adjoint representation of $\PGL_2,$
and set 
\be\label{eq:def-field-def}
{\field':=\begin{cases}\field''&\text{ if ${\rm char}(\field)\neq 2$,}\\
\{c:c^2\in\field''\}&\text{ if ${\rm char}(\field)= 2$.}\end{cases} }.
\ee
Then, there is a $\field$-isomorphism (unique up to unique isomorphism) 
\[
\varphi:{\SL_2}\times_{\field'}\field\to\SL_2
\] 
so that $\Qcal$ is an open subgroup of $\varphi(\SL(2,\field')).$ 
\end{thma1*}

\begin{proof}
Denote by $\bar{\Qcal}$ the image of $\Qcal$ under the natural map from $\SL_2$ to $\PGL_2.$
Then $\bar\Qcal$ is Zariski dense in $\PGL_2.$

By~\cite[Thm.~0.2]{Pink-Compact}, there exist  
\begin{itemize}
\item a subfield $\field'\subset\field,$
\item an absolutely simple adjoint group ${\bf L}$ defined over $\field',$ and
\item a $\field$-isogeny $\phi: {\bf L}\times_{\field'}\field\to\PGL_2$  whose derivative vanishes nowhere, \end{itemize}
where $\field'$ is unique, and ${\bf L}$ and $\phi$ are unique up to unique isomorphism, so that the following hold.
\begin{itemize}
\item $\bar\Qcal\subset\varphi({\bf L}(\field')),$ see~\cite[Thm.~3.6]{Pink-Compact},
\item let $\widetilde{{\bf L}}$ denote the simply connected cover of ${\bf L}$
and let $\widetilde{\phi}$ be the induced isogeny from $\widetilde{{\bf L}}\times_{\field'}\field$ to $\SL_2$.  
Then any compact subgroup 
$
\Qcal'\subset\widetilde{\phi}\Bigl(\widetilde{{\bf L}}(\field')\Bigr)
$ 
which is Zariski dense and normalized by $[\bar\Qcal,\bar\Qcal]$ is an open subgroup of $\widetilde{\phi}\Bigl(\widetilde{{\bf L}}(\field')\Bigr),$ see~\cite[Thm.~7.2]{Pink-Compact}.
\end{itemize}

The fact that $\field'$ can be taken as in~\eqref{eq:def-field-def}
follows from the proof of~\cite[Prop.~0.6(a)]{Pink-Compact}, see in particular~\cite[Prop.\ 3.14]{Pink-Compact} -- {}{in particular, since we are dealing with groups of type $A_1$, we only need the exceptional definition of $\field'$ in characteristic 2.}

Moreover,~\cite[Prop.\ 1.6]{Pink-Compact} implies that there are no non-standard isogenies for groups of type $A_1.$
Hence, by~\cite[Thm.~1.7(b)]{Pink-Compact}, the isogeny $\phi$ above is an isomorphism.

We now prove the other claims. First let us 
recall from~\cite[Thm.~2]{Gille} 
that since $\SL_2$ is simply connected, 
for every unipotent element $u\in\SL(2,\field)$ there exists a 
parabolic $\field$-subgroup, ${\bf P},$ of $\SL_2$ so that $u\in R_u({\bf P}(\field)).$
Hence,~\eqref{eq:gen-unip} implies that
\be\label{eq:gen-unip1}
\mathcal Q=\langle\mathcal Q\cap R_u(P):\text{$P$ is a parabolic subgroup of $\SL(2,\field)$}\rangle.
\ee

Let $P$ be a parabolic subgroup so that $\mathcal Q\cap R_u(P)\neq\{1\}.$
Let $a$ be a diagonalizable matrix in $\operatorname{PSL}(2,\field)\subset\PGL(2,\field)$ 
whose conjugation action contracts $R_u(P).$ Then $a$ contracts ${\phi}(h)$ for any $h\in{{\bf L}}(\field')$ 
where ${\phi}(h)\in R_u(P).$ 
Put $a'={\phi}^{-1}(a)$, the above implies that $h$ can be contracted to identity using conjugation by $a'.$ In particular, $h$ is a unipotent element.

In view of~\eqref{eq:gen-unip} and the above discussion $\widetilde{{\bf L}}(\field')$ contains nontrivial unipotent elements. 
Thus we get from~\cite[Cor.~3.8]{BoTi-Unip}, see also~\cite{Gille}, that $\widetilde{{\bf L}}$ is $\field'$-isotropic.
Since $\widetilde{{\bf L}}$ is simply connected, 
and $\phi$ is an isomorphism we get $\widetilde{{\bf L}}=\SL_2.$

Finally, using~\cite[Ch.\ I, Thm.\ 2.3.1]{Margulis-Book}, we have 
\[
\Qcal\cap R_u(P)\subset\widetilde\phi\Bigl(\widetilde{\bf L}(\field')\Bigr)
\] 
for any parabolic subgroup $P$ of $\SL(2,\field).$ 
Hence $\Qcal\subset\widetilde\phi\Bigl(\widetilde{\bf L}(\field')\Bigr)$ by~\eqref{eq:gen-unip1}.
This finishes the proof in case (a).
\end{proof}

For the second theorem we need some more terminology.
By a linear algebraic group ${\bf G}$ over $\field\oplus\field$
we mean ${\bf G}_1\coprod {\bf G}_2$ where each 
${\bf G}_i$ is a linear algebraic group over $\field.$
The adjoint representation of $\Gbf$ on ${\rm Lie}({\bf G})={\rm Lie}(\Gbf_1)\oplus{\rm Lie}(\Gbf_2)$
is the direct sum of the adjoint representations of ${\bf G}_i$ on ${\rm Lie}(\Gbf_i),$
and the group of $\field\oplus\field$-points of $\Gbf$
is ${\bf G}(\field\oplus\field)=\Gbf_1(\field)\times\Gbf_2(\field).$

Suppose $\G={\bf G}_1\coprod {\bf G}_2$ is a fiberwise absolutely almost simple, connected, simply connected $\field\oplus\field$-group. 
Let $\rho=(\rho_1,\rho_2)$ where $\rho_i$ is the unique irreducible subquotient of the adjoint representation of $\G_i^{\rm ad},$ see~\cite[\S1]{Pink-Compact}. 
The trace ${\rm tr}(\rho(g))$
for an element $g=(g_1,g_2)$ in $\G(\field\oplus\field)$ is defined by
\[
{\rm tr}(\rho(g))=({\rm tr}(\rho_1(g_1)),{\rm tr}(\rho_2(g_2)))\in \field\oplus\field.
\] 

Given a subfield $\field'\subset\field$ and a continuous embedding $\tau:\field'\to\field$
of fields, we put 
\be\label{eq:delta-tau}
\Delta_\tau(\field'):=\{(c,\tau(c)):c\in\field'\}.
\ee
As in~\cite[pp.~16--17]{Pink-Compact}, by a semisimple subring $\field''\subset\field\oplus\field$ we mean one of the following 
\begin{enumerate}
\item[($\field''$-1)] $\field''=\field_1 \oplus\field_2$ where $\field_i\subset\field$ is a closed subfield for $i=1,2,$ or
\item[($\field''$-2)]  $\field''=\Delta_\tau(\field')$ for a subfield $\field'\subset\field$ and a continuous embedding $\tau:\field'\to\field$.
\end{enumerate}

If $\field''=\Delta_\tau(\field')$
and $\Hbf$ is a $\field'$-group, 
we write, by abuse of notation, also $\Hbf$ for the corresponding $\tau(\field')$-group
as well as the $\Delta_\tau(\field')$-group obtained from $\Hbf$.
The base change of $\Hbf$ from $\Delta_\tau(\field)$ to $\field\oplus\field$ is then defined by
\[
\Hbf\times_{\Delta_\tau(\field')}(\field\oplus\field)=\bigg(\Hbf\times_{\field'}\field\bigg)\coprod\bigg(\Hbf\times_{\tau(\field')}\field\bigg).
\]

\begin{thma2*}[Cf.~\cite{Pink-Compact}, Theorems 0.2 and 7.2]
Assume that ${\rm char}(\field)\neq2,3,$ and let $\G_i$, $i=1,2$ 
be absolutely almost simple, connected, simply connected $\field$-groups. 
Let $\Qcal\subset\G_1(\field)\times\G_2(\field)$ be a compact subgroup so that $\pi_i(\Qcal)$ is Zariski dense in $\G_i$ for $i=1,2.$ 
Further, assume that 
\be\label{eq:gen-unip-2}
\Qcal=\langle \{g\in\Qcal:\text{$g$ is a unipotent element}\}\rangle,
\ee
Let $\field''\subset\field\oplus\field$ be defined as follows.
\be\label{eq:def-kro}
\field'':=\text{the closed ring of quotients generated by $\{{\rm tr}(\rho(g)):g\in\mathcal Q\}$}.
\ee
Then one of the following holds.
\begin{enumerate}
\item There are
\begin{enumerate}
\item[(i)] closed subfields $\field_i\subset\field$ so that $\field''=\field_1\oplus\field_2$, 
\item[(ii)] $\field_{i}$-groups ${\bf H}_{i},$ and 
\item[(iii)] $\field$-isomorphism $\varphi_{i}:{\bf H}_{i}\times_{\field_{i}}\field\to\G_i,$ 
\end{enumerate}
so that $\Qcal$ contains an open subgroup of the form 
\[
\Qcal_{1}\times\Qcal_{2}\subset \varphi_{1}\Bigl({\bf H}_{1}(\field_{1})\Bigr)\times \varphi_{2}\Bigl({\bf H}_{2}(\field_{2})\Bigr).
\] 
\item There are 
\begin{enumerate}
\item[(i$'$)]  a closed subfield $\field'\subset \field$ and a continuous embedding $\tau:\field'\to\field$ 
so that $\field''=\Delta_\tau(\field'),$  
\item[(ii$'$)] a $\field'$-group ${\bf H},$ and 
\item[(iii$'$)] a $\field\oplus\field$-isomorphism $\varphi:{\bf H}\times_{\field''}(\field\oplus\field)\to\G_1\coprod\G_2,$ 
\end{enumerate}
so that $\Qcal$ is an open subgroup of $\varphi\bigl({\bf H}(\field'')\bigr).$
\end{enumerate}
Moreover, $\field''$ is unique, and ${\bf H}$ and $\varphi$ are unique up to unique isomorphisms.
\end{thma2*}

\begin{proof}
Similar to Theorem~A.1 these assertions are special cases of results in~\cite{Pink-Compact} as we now explicate. 
Let $\G_i^{\rm ad}$ denote the adjoint form of $\G_i$ for $i=1,2$.
Denote by $\bar{\Qcal}$ the image of $\Qcal$ under the natural map
{from $\G_1\coprod\G_2$ to $\G_1^{\rm ad}\coprod\G_2^{\rm ad}$}. 
Then $\pi_i(\bar\Qcal)$ is Zariski dense in $\G_i^{\rm ad}$ for $i=1,2$.
 
By~\cite[Thm.~0.2]{Pink-Compact} we have the following. There exist  
\begin{itemize}
\item a semisimple subring $\field''\subset\field\oplus\field,$
\item a fiberwise absolutely simple adjoint group ${\bf L}$ defined over $\field'',$ and
\item a $\field\oplus\field$-isogeny $\phi: {\bf L}\times_{\field''}(\field\oplus\field)\to\G_1^{\rm ad}\sqcup\G_2^{\rm ad}$  whose derivative vanishes nowhere, 
\end{itemize}
where $\field''$ is unique, and ${\bf L}$ and $\phi$ are unique up to unique isomorphism, so that the following hold.
\begin{itemize}
\item $\bar\Qcal\subset\phi({\bf L}(\field'')),$ see~\cite[Thm.~3.6]{Pink-Compact},
\item let $\widetilde{{\bf L}}$ denote the simply connected cover of ${\bf L}$
and let $\widetilde{\phi}$ be the induced isogeny from 
$\widetilde{{\bf L}}\times_{\field''}(\field\oplus\field)$ to $\G_1\sqcup\G_2$.  
Then any compact subgroup 
$
\Qcal'\subset\widetilde{\phi}\Bigl(\widetilde{{\bf L}}(\field'')\Bigr)
$ 
which is fiberwise Zarsiki dense and normalized by $[\bar\Qcal,\bar\Qcal]$ is an open subgroup of $\widetilde{\phi}\Bigl(\widetilde{{\bf L}}(\field'')\Bigr),$ see~\cite[Thm.~7.2]{Pink-Compact}.
\end{itemize}

Recall our assumption that ${\rm char}(\field)\neq2,3$.
Therefore, $\G_1$ and $\G_2$ have no non-standard isogenies, see~\cite[Prop.\ 1.6]{Pink-Compact}. 
This also implies that $\field''$ can be taken as in~\eqref{eq:def-kro}, see~\cite[Prop.\ 3.13 and Prop.\ 3.14]{Pink-Compact}.
Moreover, by~\cite[Thm.~1.7(b)]{Pink-Compact}, the isogeny $\phi$ above is an isomorphism.

The above discussion thus implies that if $\field''=\Delta_\tau(\field')$, see ($\field''$-2), then (i$'$), (ii$'$), and (iii$'$) hold. 
Similarly, if $\field''=\field_1\oplus\field_2$, see ($\field''$-1), then (i), (ii), and (iii) hold, in view of the above discussion, and the 
description of algebraic groups and their isogenies over $\field_1\oplus\field_2$ and $\field\oplus\field$.

Finally recall from~\eqref{eq:gen-unip-2} that $\Qcal$ is generated by unipotent elements, therefore,
$\Qcal\subset\widetilde{\phi}\Bigl(\widetilde{{\bf L}}(\field'')\Bigr),$ see~\cite[Ch.\ I, Thm.\ 2.3.1]{Margulis-Book}.
This finishes the proof of case (b).
\end{proof}

We will also need the following lemma.
Let $\mathbf U^+$ (resp.\ $\mathbf U^-$) denote the group of upper (resp.\ lower) triangular unipotent 
matrices in $\SL_2$. 
Also let $\mathbf T$ denote the group of diagonal matrices in $\SL_2.$ 
Put $U^\pm:=\mathbf U^\pm(\field)$ and $T:=\mathbf T(\field)$.

\begin{lem}\label{lem:E-generated}
Let the notation be as in Theorem~A.1. 
Put $E=\varphi(\SL(2,\field'))$, then
\begin{enumerate}
\item $E=\langle E\cap U^+,E\cap U^-\rangle.$
\item $E\cap T$ is unbounded.
\end{enumerate}
\end{lem}

\begin{proof}
We showed in the course of the proof of Theorem~A.1 that there are 
nontrivial unipotent elements $h^\pm\in \SL(2,\field')$ so that ${\varphi}(h^\pm)\in U^\pm,$ respectively.
Since $\SL_2$ is simply connected, it follows from~\cite[Thm.~2]{Gille} 
that there are $\field'$-parabolic subgroups ${\bf P}^\pm$
of $\SL_2$ so that $h^\pm\in R_u({\bf P}^\pm).$
The groups $R_u(\mathbf P^\pm)$ are one dimensional  $k'$-split unipotent subgroups,
 hence $\varphi(R_u(\mathbf P^\pm)(\field'))\subset \varphi(\SL_2)$
is an infinite group. Note that $\varphi(\SL_2)=\SL_2$ in Theorem~A.1. 
Let ${\bf U}'_\pm$ denote the Zariski closure of
$\varphi(R_u(\mathbf P^\pm)(\field'))$. Then ${\bf U}'_\pm$ is a nontrivial connected unipotent subgroup of $\varphi(\SL_2)$ which intersects ${\bf U}^\pm\cap \varphi(\SL_2)$ nontrivially. Therefore,
${\bf U}'_\pm={\bf U}^\pm\cap \varphi(\SL_2)$ which implies
\be\label{eq:Palpha-Ualpha}
{\varphi}(R_u({\bf P}^\pm)(\field'))\subset U^\pm\cap E.
\ee

Using the fact that $\SL_2$ is simply connected one more time,
we note that $\SL(2,\field')$ is generated by $R_u({\bf P}^\pm)( \field'),$~\cite[Ch.~1, Thm.~2.3.1]{Margulis-Book}.
This and~\eqref{eq:Palpha-Ualpha} imply (1) in the lemma.

We now show (2) in the lemma. Let $\mathbf S=\mathbf P^+\cap\mathbf P^-.$
Then ${\bf S}$ is a one dimensional $\field'$-split $\field'$-torus; put $S=\mathbf S(\field').$
Now 
\[
T':=\varphi(S)\subset TU^+\cap TU^-=T
\] 
satisfies the claim in (2).  
\end{proof}

\subsection{Measures invariant under semisimple groups}\label{sec:ratner-sl2}
We will state in this subsection the measure classification 
result by Salehi-Golsefidy and the third named author~\cite{MS-SL2}
for probability measures that are invariant under non-compact semisimple groups
in the positive characteristic setting.

For this
we need some notation and definitions, these generalize the notions defined in~\eqref{eq:W-pm} to a general 
connected group.
Let $\field$ be a local field.
Suppose $\mathbf{M}$ is a connected $k$-algebraic group, 
and let $\lambda:{\bf G}_m\rightarrow{{\bf M}}$ be a noncentral homomorphism defined over $k.$ 
Define $-\lambda(\cdot)=\lambda(\cdot)^{-1}.$ 

{}{Recall that a morphism from ${\bf G}_m$ to $\bf M$ is said to have 
a limit at $0$ when it can be extended to a morphism from ${\bf A}^1$ to $\bf M$. 
As in~\cite[\S13.4]{Sp-AlgGrBook} and~\cite[Ch.~2 and App.~C]{CGP-PseudoRed}, we let $\bbp_{{\bf M}}(\lambda)$ 
denote the smooth closed subgroup of ${{\bf M}}$ defined over $\field$ so that 
\[
\bbp_{{\bf M}}(\lambda)(R)=\{r\in{\bf M}(R): \lambda r\lambda^{-1}\text{from ${\bf G}_m$ to ${\bf M}$ has a limit at 0}\}
\]
for any algebra $R/\field$.


Let $\bbw_{\bf M}^+(\lambda)$ be the closed normal subgroup of $\bbp_{{\bf M}}(\lambda)$ so that
\[
\bbw_{\bf M}^+(\lambda)(R)=\{r\in{\bf M}(R): \lambda r\lambda^{-1}\text{from ${\bf G}_m$ to ${\bf M}$ has a limit at 0}\}
\]
for any algebra $R/\field$.  Similarly, define $\bbw_{{\bf M}}^+(-\lambda)$ which we will denote by $\bbw_{{\bf M}}^-(\lambda).$

The centralizer of the image of $\lambda$ is denoted by ${\bf Z}_{{\bf M}}(\lambda).$ 
 
The subgroups $\bbw_{\bf M}^+(\lambda)$, ${\bf Z}_{{\bf M}}(\lambda)$, $\bbw_{{\bf M}}^-(\lambda)$ are {\em smooth} closed subgroups, see ~\cite[Ch.~2 and App.~C]{CGP-PseudoRed}. 
}

The multiplicative group ${\bf G}_m$ acts on $\Lie({{\bf M}})$ via $\lambda,$ and the weights are integers. 
The Lie algebras of ${\bf Z}_{{\bf M}}(\lambda)$ and $\bbw_{{\bf M}}^\pm(\lambda)$ 
may be identified with the weight subspaces of this action corresponding to the zero, positive and negative weights. 
It is shown in~\cite[Ch.~2 and App.~C]{CGP-PseudoRed} 
that $\bbp_{{\bf M}}(\lambda),$ ${\bf Z}_{{\bf M}}(\lambda)$ and $\bbw_{{\bf M}}^{\pm}(\lambda)$ 
are $k$-subgroups of ${{\bf M}}.$ Moreover, $\bbw_{{\bf M}}^+(\lambda)$ is a normal subgroup of $\bbp_{{\bf M}}(\lambda)$ and the product map 
\[
\mbox{${\bf Z}_{{\bf M}}(\lambda)\times \bbw_{{\bf M}}^+(\lambda)\rightarrow\bbp_{{\bf M}}(\lambda)$ 
is a $k$-isomorphism of varieties.} 
\]

A pseudo-parabolic $\field$-subgroup of ${\bf M}$ is a group of the form
$\bbp_{\bf M}(\la)R_{u,\field}({\bf M})$ for some $\la$ as above 
where $R_{u,\field}({\bf M})$ denotes the maximal
connected normal unipotent $\field$-subgroup of ${\bf M},$~\cite[Def.~2.2.1]{CGP-PseudoRed}.

We also recall from~\cite[Prop.~2.1.8(3)]{CGP-PseudoRed} that the product map
\begin{equation}\label{opposite-horo}
\bbw^-_{{\bf M}}(\lambda)\times{\bf Z}_{{\bf M}}(\lambda)\times\bbw_{{\bf M}}^+(\lambda)\rightarrow{{\bf M}}\;\text{ is an open immersion of $k$-schemes.}
\end{equation} 
It is worth mentioning that these results are generalization to arbitrary groups of analogous and 
well known statements for reductive groups. 

Let $M={{\bf M}}(\field),$ and put 
\[
\mbox{$W_M^\pm(\la)=\bbw^\pm_{{\bf M}}(\lambda)(k),$ and $Z_M(\la)={\bf Z}_{{\bf M}}(\lambda)(k).$ }
\]
From~\eqref{opposite-horo} we conclude that $W_M^-(\la)Z_M(\la)W_M^+(\la)$ 
is a Zariski open dense subset of $M,$ which contains a neighborhood of 
identity with respect to the Hausdorff topology.

For any $\la$ as above define 
\be\label{eq:M-+-la}
M^+(\la):=\langle W_M^+(\la),W_M^-(\la)\rangle.
\ee

\begin{lem}\label{lem:w-pm-normal}\leavevmode
\begin{enumerate}
\item For any $\la$ as above, $M^+(\la)$ is a normal and unimodular subgroup of $M$.
\item There are only countably many subgroups {of the form} $M^+(\la)$ in $M.$
\end{enumerate}
\end{lem}

Combining results in~\cite[App.\ C]{CGP-PseudoRed} together with part~(1) in the lemma one can actually conclude that there are only finitely many such subgroups. We shall only make use of the weaker statement above.

\begin{proof}
Part (1) is proved in~\cite[Lemma 2.1]{MS-SL2}.
We now prove (2). 
First note if $\la_1,\la_2:\mathbf G_m\to\mathbf M$
are two homomorphisms so that $\la_1=g\la_2g^{-1}$ for some $g\in M,$
then $M^+(\la_1)=gM^+(\la_2)g^{-1}.$ Therefore, by part (1) we have  
\be\label{eq:la-la-conj}
M^+(\la_1)=M^+(\la_2)\text{ whenever $\la_1=g\la_2g^{-1}$ for some $g\in M.$}
\ee

Let now $\mathbf S$ be a maximal $\field$-split $\field$-torus
in $\mathbf M.$ By~\cite[Thm.~C.2.3]{CGP-PseudoRed} there is some
$g\in M$ so that $g\lambda g^{-1}:\mathbf G_m\to\mathbf S.$
The claim now follows from this,~\eqref{eq:la-la-conj}, and the fact that the
finitely generated abelian group $X_{*}(\mathbf S)={\rm Hom}({\bf G}_m,\mathbf S)$ is countable.
\end{proof}

Given any subfield $\sfield\subset\field$ so that $\field/\sfield$ is a finite extension
we let $\mathcal R_{\field/\sfield}$ denote the Weil's restriction of scalars, see~\cite[\S{A.5}]{CGP-PseudoRed}. 

In the following let~$\G$ be a connected k-group and let $\Gamma\subset G$
	be a discrete subgroup in $G=\G(k)$.
Furthermore, let $\field'\subset\field$ be a closed subfield and let ${\bf H}$ be 
an absolutely almost simple, $\field'$-isotropic, $\field'$-group. Assume that
$\varphi:{\bf H}\times_{\field'}\field\to{\bf G}$ is a nontrivial $\field$-homomorphism, and
put $E=\varphi\Bigl(\mathbf H(\field')\Bigr).$  We use in an essential way
the following measure classification result by Alireza Salehi-Golsefidy
and the third named author. 

\begin{thmb*}[\cite{MS-SL2}, Theorem 6.9 and Corollary 6.10]\label{thm:ratner-sl2}
Let $\nu$ be a probability measure on $G/\Gamma$ which is $E$-invariant
and ergodic. Then, there exist
\begin{enumerate}
\item some $\sfield=(\field')^q\subset\field $ where $q=p^n,$ $p={\rm char}(\field)$ and $n$ is a nonnegative integer, 
\item a connected $\sfield$-subgroup ${\bf M}$ of $\rcal_{\field/\sfield}({\bf G})$ so that ${\bf M}(\sfield)\cap \Gamma$ is Zariski dense in ${\bf M},$ and
\item an element $g_0\in G,$ 
\end{enumerate} 
such that $\nu$ is the $g_0Lg_0^{-1}$-invariant probability Haar measure on the closed orbit $g_0L\Gamma/\Gamma$ 
with 
\[
L=\overline{M^+(\la)({\bf M}(\sfield)\cap\Gamma)},
\]
where  
\begin{itemize}
\item the closure is with respect to the Hausdorff topology, and 
\item $\la:{\bf G}_m\to{\bf M}$ is a noncentral $\sfield$-homomorphism, $M^+(\la)$ is defined in~\eqref{eq:M-+-la}, and $E\subset g_0M^+(\la)g_0^{-1}.$
\end{itemize} 
\end{thmb*}

\subsection{A version of Borel density theorem}\label{sec:measures-varieties}
Let $\field'\subset\field$ be an infinite closed subfield.  
We recall from~\cite[Prop.~1.4]{Shalom-BorelDensity} that the {\em discompact radical} of a $\field'$-group
is the maximal $\field'$-subgroup which does not have any nontrivial compact
$\field'$-algebraic quotients.
It is shown in~\cite[Prop.~1.4]{Shalom-BorelDensity} that 
this subgroup exists and the quotient of the $k'$-points
of the original group by the~$\field'$-points of the discompact radical is compact. 
{Let $\bf A$ be a $\field$-split torus.}
Let $A^{{\rm sp}}_{\field'}\subset \mathcal R_{\field/\field'}({\bf A})(\field')=A$ 
denote the $k'$-points of the maximal $k'$-split subtorus of $\mathcal R_{\field/\field'}({\bf A}).$

Suppose ${\bf V}$ is a variety defined over $\field'$ 
and assume that ${\mathcal R}_{\field/\field'}({\bf A})$ acts on ${\bf V}$ via $\field'$-morphisms.
In particular, $A={\mathcal R}_{\field/\field'}({\bf A})(k')$ acts on $V={\bf V}(\field')$ via $k'$-morphisms.

\begin{lem}[Cf.~\cite{Shalom-BorelDensity}, Theorem 1.1]\label{lem:zd-measure1}
Let $(X,\eta)$ be an $A$-invariant ergodic probability space.  
Let $f:X\to V$ be an $A$-equivariant Borel map. 
Then there exists some $v_0\in{\rm Fix}_{A^{{\rm sp}}_{\field'}}(V),$
so that $f_*\eta$ is the $A$-invariant measure on the compact orbit $Av_0.$
In particular, $f(x)\in Av_0$ for $\eta$-a.e.\ $x.$    
\end{lem} 

\begin{proof}
This follows from~\cite[Thm.\ 1.1]{Shalom-BorelDensity} in view of the fact that 
$A^{{\rm sp}}_{\field'}$ is the {\em discompact} radical of ${\mathcal R}_{\field/\field'}({\bf A})$
as defined in~\cite{Shalom-BorelDensity}, see also~\cite[Thm.~3.6]{Shalom-BorelDensity}. 
\end{proof}


\subsection{Pinsker \texorpdfstring{$\sigma$}{sigma}-algebra and unstable leaves}\label{sec:pinsker} 
Throughout this section we assume $\G$ is a $\field$-isotropic semisimple $\field$-group
and let ${\bf A}$ be a $\field$-split $\field$-torus in $\Gbf.$ 
Put $G=\Gbf(\field)$ and $A=\mathbf A(\field).$ Let $\Gamma$ be a discrete subgroup of $G$
and put $X=G/\Gamma.$ 

Let $a\in A$ be a nontrivial element. 
Recall that for an $a$-invariant measure $\mu$ we define the Pinsker $\sigma$-algebra as 
\[
\pins_a:=\bigl\{B\in\mathcal B: \entropy_\mu(a,\{B,X\setminus B\})=0\bigr\}.
\] 
It is the largest $\sigma$-algebra with respect to which $\mu$ has zero entropy, see~\cite{Walters} for further discussion.

Let us recall the following important and well known proposition; we outline 
the proof for the sake of completeness.

\begin{prop}\label{prop:pins-meas-hull}
The Pinsker $\sigma$-algebra, $\pins_a,$ is equivalent to the $\sigma$-algebra of Borel 
sets foliated by $W_G^+( a)$ leaves. 
\end{prop}

Note that the Pinsker~$\sigma$-algebra for~$a$ equals the Pinsker~$\sigma$-algebra for~$a^{-1}$,
which shows that the proposition also applies similarly for $W_G^-( a)$.

\begin{proof}
Suppose $\mathcal C$ is any $\sigma$-algebra whose elements are foliated by
$W^+_G( a)$ leaves. Let ${\rm p}:(X,\mu)\to (Y,{\rm p}_*\mu)$ be the corresponding factor map.
Using the Abramov-Rohklin conditional entropy formula 
and the relationship between entropy and leafwise measures, see~\cite{EL-Pisa}, we get the following
\[
\entropy(a,(Y,{\rm p}_*\mu))=0.
\] 
The definition of the Pinsker $\sigma$-algebra 
then implies that $\mathcal C\subset\pinsker_a.$ 

For the converse we recall from~\cite[Sec.~9]{MarTom}, see also~\cite{EL-Pisa}, that there is a finite entropy generator, i.e. a countable partition $\xi$ of finite entropy such that~$\bigvee_{n=-\infty}^{\infty}a^{-n}\xi$
is equivalent to the full Borel $\sigma$-algebra, and  
so that in addition the past is subordinate with respect to $W^+_G( a).$ 
That is to say that on the complement of a null set every atom of~$\bigvee_{n=-\infty}^0 a^{-n}\xi$
is an open subset of a~$W^+_G(a)$-orbit.
Hence, after removing a null set,
any set measurable with respect to the tail 
$\bigcap_{k\in\mathbb{N}}\bigvee_{n=-\infty}^{-k}a^{-n}\xi$ is a union of~$W_G^+(a)$-orbits. 
Since $\pinsker_a$  is equivalent to the tail of $\xi$ modulo $\mu$, the claim follows.
\end{proof}   

The following will be used in the course of the proof of Theorem~\ref{thm:joining}.
\begin{lem}\label{lem:pinsker-comp-joinning}
Let $X_i=G_i/\Gamma$ be as in Theorem~\ref{thm:joining}. 
In particular, $G_i=\Gbf_i(\field)$ where $\Gbf_i$ is a connected, simply connected,
absolutely almost simple group defined over $\field$ for $i=1,2.$ 
Let $a=(a_1,a_2)\in A$ such that~$a$ generates an unbounded group, 
and suppose $\mu$ is an ergodic joining 
of the $A_i$-action on $(X_i,m_i),$ for $i=1,2.$ Let $\mu=\int_{X_1\times X_2}\mu_x^{\pins_a}\operatorname{d}\!\mu(x)$ where
$\mu_x^{\pins_a}$ denotes the conditional measure for $\mu$-a.e.\ $x$
with respect to the Pinsker $\sigma$-algebra $\pins_a$.    
Then there exists a subset $X'\subset X_1\times X_2$ with $\mu(X')=1$
so that 
\[
\text{$\pi_{i*}(\mu_x^{\pins_a})=m_i$ for all $x\in X'$ and $i=1,2.$}
\]
\end{lem}

\begin{proof}
Let ${\rm P}$ denote the pinsker factor of $X$ and let $\Upsilon:X\to{\rm P}$ 
be the corresponding factor map. This is a zero entropy factor of $X.$

Put $Z=X_1\times X_2\times {\rm P},$ and let 
\[
\nu=\int \mu_x^{\pins_a}\times\delta_{\Upsilon(x)} \operatorname{d}\!\Upsilon_*\mu(x).
\]
Let ${\rm p}_i:Z\to X_i\times{\rm P}$ be the natural projection.
Then ${\rm p}_{i*}\nu$ is a measure on $X_i\times{\rm P}$ which projects to $m_i$ and $\Upsilon_*\mu$
for $i=1,2.$ Now $(X_i,m_i)$ is a system with completely positive entropy. This follows, e.g., from Proposition~\ref{prop:pins-meas-hull} and the ergodicity of the action of $W^\pm(a_i)$; note that the latter holds since $\G_i$ is connected, simply connected, and absolutely almost simple,~\cite[Ch.~1, Thm. 2.3.1, Ch.~2, Thm.\ 2.7]{Margulis-Book}. However, $({\rm P},\Upsilon_*\mu)$ is a zero entropy system, therefore, by disjointness theorem of Furstenberg~\cite{Furstenberg-disjointness}, see also~\cite[Thm.~18.16]{Glas-Joining} we obtain 
\be\label{eq:K-Zero}
{\rm p}_{i*}\nu=m_i\times\Upsilon_*\mu.
\ee

Let us now decompose ${\rm p}_{i*}\nu$ as
\[
{\rm p}_{i*}\nu=\int \bigl({\rm p}_{i*}\nu\bigr)_{(x_i,p)}^{X_i\times\mathcal B_{{\rm P}}} \operatorname{d}\!{\rm p}_{i*}\nu.
\]
Then~\eqref{eq:K-Zero} implies that for ${\rm p}_{i*}\nu$-almost every $(x_i,p)$ we have
\[
\bigl({\rm p}_{i*}\nu\bigr)_{(x_i,p)}^{X_i\times\mathcal B_{{\rm P}}}=m_i\times \delta_{p}.
\]
This in view of the definition of $\nu$ implies the claim.
\end{proof}


\subsection{Leafwise measures}\label{sec:leafwise}
We refer to~\cite{EL-Pisa} for a comprehensive treatment of leafwise measures.

Recall that $\G$ is a $\field$-isotropic semisimple $\field$-group
and let ${\bf A}$ be a $\field$-split $\field$-torus in $\Gbf.$
Let ${\bf S}$ be a maximal $\field$-split $\field$-torus of $\Gbf$
which contains ${\bf A}$.
Let 
${}_\field\rootsys({\bf S},\Gbf)$ be the relative root system of $\Gbf$,
and let ${}_\field\Phi({\bf A},\Gbf)$ denote the set of roots of ${\bf A}$ as in~\S\ref{sec:notation}.

{}{
\begin{def*}
Let $U$ be an $A$-normalized unipotent $k$-subgroup of $G$ contained in some $W_G^-(a)$.
The leafwise measure $\mu_x^U$ 
along $U$ is defined for $\mu$-a.e.\ $x\in X$. For all such $x$ we put
\[
\lsupp_x^U=\supp(\mu_x^U)\quad\text{and}\quad 
\linv_x^U=\{v\in U: v\mu_x^U=\mu_x^U\}.
\]
\end{def*}
}

The leafwise measures are canonically defined up to proportionality, 
and we write $\propto$ to denote proportionality. 
The main case we shall be interested in is when $V_\Psi:=U_{\vartheta(\Psi)}$ is the associated unipotent subgroup of a 
{positively} closed set $\Psi\subset{}_\field\rootsys({\bf A},\Gbf)$, 
in which case we will use $\mu_x^\Psi,\lsupp_x^\Psi,\linv_x^\Psi$ to denote $\mu_x^{V_\Psi},\lsupp_x^{V_\Psi},\linv_x^{V_\Psi}$ respectively.

\begin{lem}\label{lem:prop constant}
Under the above assumtions, a.s.\ $\linv_x^U=\{v\in U: v\mu_x^U\propto \mu_x^U\}$.
\end{lem}

\begin{proof}
This is true in general, but is particularly easy in the positive characteristic case:
Suppose $u \in U$ is such that $u\mu_x^U\propto\mu_x^U$.
Then $u\mu_x^U=\kappa\mu_x^U$ for some $\kappa>0$.
Since $U$ is unipotent,
$u$ is torsion of exponent $p^n$ for some $n$, hence $\kappa^{p^n}=1$,  which implies (since $\kappa>0$) that $\kappa=1$.
\end{proof}

We recall some properties of leafwise measures which will be used throughout. Our formulation is taken from~\cite{EL-GenLow}, see~\cite{Elon-Quantum} as well as~\cite{EL-Pisa} and references there.

\begin{lem}\label{lem:properties}
Let $U$ be an $A$-normalized unipotent $k$-subgroup of $G$ contained in some $W_G^-(a)$.
There is a conull subset $X'\subset X$ so that
\begin{enumerate}
\item For all $x\in X'$ the map $x\mapsto\mu_x^U$ {}{from $X$ to the space of Radon measures on $U$ normalized so that $\mu_x^U([1])=1$} is a measurable map. In particular,
$\mu_x^U$ is defined for all $x\in X'.$
\item For every $x\in X'$ and every $u\in U$ so that $ux\in X',$ we have $\mu_{x}^U\propto(\mu_{ux}^U)u$ where $(\mu_{ux}^U)u$ denotes the push forward of $\mu_{ux}^U$ under the map $v\mapsto vu.$
\item  For every $x\in X'$ we have $\mu_x^U(U[1])=1$ and 
$\mu_x^U(U[\epsilon])>0$ for all $\epsilon>0.$
\item Suppose $\mu$ is $a$-invariant under some $a\in A.$ Then for $\mu$-a-e.\ $x\in X,$ 
we have $\mu_{a x}^U\propto (a\mu_x^Ua^{-1}).$
\end{enumerate}
\end{lem}

\begin{lem}[Cf.~\cite{EL-GenLow}, \S6]\label{lem:ZC-1Ball}
Let $a\in A$ be so that the Zariski closure of $\langle a\rangle$, ${\bf A'}$ say, is $k$-isomorphic to $\G_m$
and that ${\bf A}'(\field)/\langle a\rangle$ is compact.  
Suppose $\mu$ is $a$-invariant and let $U$ be an $A$-normalized unipotent $k$-subgroup of $G$ contained in  $W_G^-(a)$.
Let $Q$ be any compact open subgroup of $U$. Then for $\mu$-a.e.\ $x$, the Zariski closure of 
$\linv_x^U\cap Q$ is normalized by $a$ and contains $\linv_x^U$.
\end{lem}

\begin{proof}
Let $\mathcal E$ denote a countably generated $\sigma$-algebra that is equivalent 
to the $\sigma$-algebra of $a$-invariant sets. 
Then~$(\mu_x^{\mathcal E})_y^{U}=\mu_{y}^U$
for~$\mu_x^{\mathcal E}$-a.e.~$y$ and~$\mu$-a.e.~$x$, see e.g.~\cite{EL-Pisa}.
Therefore, we may assume that $\mu$ is $a$-ergodic. 

Let $\mathfrak U_0$ denote a fixed compact open subgroup of $U$. 
For any $n\in\bbz$, define
\[
\mathfrak U_n=a^n\mathfrak U_0a^{-n}.
\]
Then $\mathfrak U_n\subset Q$ for large enough $n$, hence, it suffices to prove the lemma for $Q=\mathfrak U_n$.

Let $X'\subset X$ be a conull set where Lemma~\ref{lem:properties} holds.
For any $x\in X'$ and any $n\in\bbz$ define
\[
{\bf F}_{x,n}=\text{The Zariski closure of $\mathfrak U_n\cap\linv_x^U$}.
\]
Then ${\bf F}_{x,n}$ is a $\field$-group, see e.g.~\cite[Lemma 11.2.4(ii)]{Sp-AlgGrBook}.

Note also that ${\bf F}_{x,n}\subset{\bf F}_{x,m}$ whenever $n\geq m$.
Therefore, there exists some $n_0=n_0(x)$ so that $\dim {\bf F}_{x,n}=\dim{\bf F}_{x,n_0}$
for all $n\geq n_0$, where $\dim$ is the dimension as a $\field$-group.
Since the number of connected components of ${\bf F}_{x,n_0}$ is finite, 
there exists $n_1=n_1(x)$ so that ${\bf F}_{x,n}={\bf F}_{x,n_1}$ for all $n\geq n_1.$
Put $\mathbf F_x:={\bf F}_{x,n_1}.$ 

The definition of ${\bf F}_{x,n}$, in view of Lemma~\ref{lem:properties}(4), implies that
\[
{\bf F}_{ax,n+1}=a{\bf F}_{x,n}a^{-1}.
\]
Therefore, we have 
\be\label{eq:a-equi-F-2}
\mathbf F_{ax}=a\mathbf F_xa^{-1}.
\ee 

Let $k[\G]$ denote the ring of regular functions of $\G$.
For every $x\in X'$, let $J_x\subset k[\G]$ be the ideal of regular functions 
vanishing on ${\bf F}_x.$ Let $m(x)$ be the minimum integer so that
$J_x$ is generated by polynomials of degree at most $m(x)$. 
In view of~\eqref{eq:a-equi-F-2}, we have $m(x)=m(ax)$. Since $\mu$ is $a$-ergodic, we have
$x\mapsto m(x)$ is essentially constant. Replacing $X'$ by a conull subset, if necessary, we assume that $m(x)=m$ for all $x\in X'$.  

Let $\Upsilon=\{h\in k[\G]: \deg(h)\leq m\}$.
Using a similar argument as above, we may assume that $\dim(J_x\cap\Upsilon)=\ell$ for all $x\in X'$.

Let $f:X\to {\rm Grass}(\ell)$, the Grassmannian of $\ell$-dimensional subspaces of $\Upsilon$, be the map
defined by $f(x)=J_x\cap\Upsilon$ for all $x\in X'$. Then, $f$ is an $a$-equivariant, Borel map. Therefore, $\nu=f_*\mu$
is a probability measure on ${\rm Grass}(\ell)$ which is invariant and ergodic for a $\field$-algebraic action of $a$ on ${\rm Grass}(\ell)$. 
Hence,
\[
\bar\nu=\int_{{\bf A}'(k)/\langle a\rangle}b_*\nu\operatorname{d}\!b
\]   
is an ${\bf A}'(k)$-invariant, ergodic probability measure on ${\rm Grass}(\ell)$ equipped with an algebraic action of ${\bf A}'(k)$.
By~\cite[Thm.~3.6]{Shalom-BorelDensity}, $\bar\nu$ is the delta mass at an ${\bf A}'(k)$-fixed point 
which implies $\nu=\bar\nu$ is the delta mass at an ${\bf A}'(k)$-fixed. 

Therefore, $f$ is essentially constant. Using the definition of $f$, we get that 
$a{\bf F}_xa^{-1}={\bf F}_x$ for $\mu$-a.e.\ $x$.
This,~\eqref{eq:a-equi-F-2}, and the ergodicity of $\mu$ imply that ${\bf F}_x={\bf F}$ for $\mu$-a.e.\ $x$.

Let now $C\subset X'$ be a compact subset with $\mu(C)>1-\epsilon$ so that 
\begin{itemize}
\item $n_1(x)\leq N_1$ for all $x\in C$,
\item ${\bf F}_x={\bf F}$ for all $x\in C$.
\end{itemize}  

By pointwise ergodic theorem for almost every $x\in X$ there is a sequence $m_i\to\infty$
so that $a^{m_i}x\in C$ for all $i$. Let now $x$ be such a point, and let $u\in\linv_x^U$.
By Lemma~\ref{lem:properties}(4) we have 
\[
a^{m_i}ua^{-m_i}\in\mathfrak U_{N_1}\cap\linv_{a^{m_i}x}^U\subset{\bf F}(\field).
\]
for all large enough $i$.
Since ${\bf F}(\field)$ is normalized by $a$, we get that $u\in{\bf F}(\field)$. 
\end{proof}

From this point on we assume that $\mu$ is $A$-invariant.
We recall the product structure for leafwise measures, see~\cite{EK-2}. Our formulation is taken from~\cite[Prop.~8.5 and Cor.~8.8]{EL-Pisa}. 

\begin{lem}\label{lem:prod-struc}
Fix some $a\in A.$ Let $H=T\ltimes U$ where $U<W_G^-(a)$ and $T<Z_G(a).$ 
Then there exists a conull subset $X'\subset X$ with the following properties.
\begin{enumerate}
\item For every $x\in X'$ and $h\in H$ such that $hx\in X'$
we have $\mu_{x}^T\propto (\mu_{hx}^T)t$ where $h=ut=tu'$ for $t\in T$
and $u,u'\in U.$
\item For every $x\in X'$ we have $\mu_x^H\propto\iota_*(\mu_x^T\times \mu_x^U)$
where $\iota(t,u)=tu$ is the product map.
\item Assume further that $T$ centralizes $U.$ Then for all $x\in X'$ and $t\in T$ so that $tx\in X'$ we have $\mu_x^U\propto\mu_{tx}^{U}.$
\end{enumerate}
\end{lem}
By induction, as in \cite[\S8]{EL-Pisa}, this lemma implies a product structure for the conditional measures $\mu ^ \Psi _ x$.

\begin{prop}[Cf.~{\cite[Thm. 8.4]{EK-2}}] \label{prop:general product structure} Let $\Psi \subset {}_\field\Phi({\bf A},\Gbf)$ be a positively closed subset of Lyapunov exponents. Let $[\alpha _ 1], [\alpha _ 2], \dots, [\alpha _ k]$ be any ordering of the course Lyapunov weights contained in $\Psi$. Then for $\mu$-a.e. $x \in X$,
\begin{equation*}
\mu _ x ^ \Psi \propto \iota _ {*}\bigl (\mu _ x ^{[\alpha _ 1]} \times \dots \times \mu _ x ^ {[\alpha _ k]}\bigr)
.\end{equation*}
\end{prop}

For the proof cf.\ e.g.\ \cite{EK-2} or \cite[\S8]{EL-Pisa}. 

\begin{lem}\label{lem:commutator}
Suppose $\mu$ is an $A$-invariant ergodic probability measure.
Let $\Psi \subset {}_\field\Phi({\bf A},\Gbf)$ be a positively closed subset, and assume that $\alpha,\beta\in\Psi$ 
are linearly independent roots. Let $\Psi' \subset \Psi$ be those elements of $\Psi$ that can be expressed as a linear combination of $\alpha$ and $\beta$ with strictly positive coefficients. Then $\Psi'$ is also closed and 
for $\mu$-a.e.~$x$  we have 
\begin{align*}
\biggl[\lsupp_x^{[\alpha]},\lsupp_x^{[\beta]}\biggr]\subset\linv_x^{\Psi}\qquad\text{and}\qquad
\biggl[\lsupp_x^{[\alpha]},\lsupp_x^{[\beta]}\biggr]\subset\linv_x^{\Psi'}
\end{align*}
\end{lem}

\begin{proof}
By e.g.~\cite[\S2.5]{BoTi-RedGr} both $\Psi '$ and $\Psi ' \cup \left\{ \alpha, \beta \right\}$ are 
{positively} closed subset of ${}_\field\Phi({\bf A},\Gbf)$. Let $[\gamma _ 1]$, \dots, $[\gamma _ \ell]$ be an enumeration of all course Lyapunovs in $\Psi \setminus (\Psi ' \cup \left\{ \alpha, \beta \right\})$.

Then by Proposition~\ref{prop:general product structure}
\begin{align}\label{eq:prod-psi}
\mu_x^\Psi&\ppor\iota_*(\mu_x^{[\alpha]}\times\mu_x^{[\beta]}\times\mu_x^{\Psi'}\times\mu_x^{[\gamma_1]}\times\dots\times\mu_x^{[\gamma_\ell]})\\
\notag&\ppor\iota_*(\mu_x^{[\beta]}\times\mu_x^{[\alpha]}\times\mu_x^{\Psi'}\times\mu_x^{[\gamma_1]}\times\dots\times\mu_x^{[\gamma_\ell]})
\end{align}
where $\iota$ is the product map.

Let now $f \in C_c(V^\Psi)$, then~\eqref{eq:prod-psi} and Fubini's theorem implies that
\begin{align*}
\int f(g)d \mu_x^{W}&= \kappa \int f(v_\alpha v_\beta v_{\Psi'}v_{\gamma_1}\dots v_{\gamma_\ell})\operatorname{d}\!\mu_x^{V_{[\alpha]}}\operatorname{d}\!\mu_x^{V_{[\beta]}}\operatorname{d}\!\mu_x^{\Psi'}
\operatorname{d}\!\mu_x^{V_{[\gamma_1]}}\dots\operatorname{d}\!\mu_x^{V_{[\gamma_\ell]}}\\
&=\kappa' 
\int f(v_\beta v_\alpha v_{\Psi'}v_{\gamma_1}\dots v_{\gamma_\ell})\operatorname{d}\!\mu_x^{V_{[\alpha]}}\operatorname{d}\!\mu_x^{V_{[\beta]}}\operatorname{d}\!\mu_x^{\Psi'}
\operatorname{d}\!\mu_x^{V_{[\gamma_1]}}\dots\operatorname{d}\!\mu_x^{V_{[\gamma_\ell]}}
\\
&=\kappa'\int f(v_\alpha v_\beta [v_\beta,v_\alpha] v_{\Psi'}v_{\gamma_1}\dots v_{\gamma_\ell})\operatorname{d}\!\mu_x^{V_{[\alpha]}}\operatorname{d}\!\mu_x^{V_{[\beta]}}\operatorname{d}\!\mu_x^{\Psi'}
\operatorname{d}\!\mu_x^{V_{[\gamma_1]}}\dots\operatorname{d}\!\mu_x^{V_{[\gamma_\ell]}}
\end{align*}
for $\kappa,\kappa'$ independent of $f$.

From this we get for $\mu_x^{[\alpha]}$-a.e.\ $v_\alpha \in V_{[\alpha]}$ and $\mu_x^{[\beta]}$-a.e.\ $v_\beta \in V_{[\beta]}$,
\[
\mu_x^{\Psi'}\propto[v_\beta,v_\alpha]\mu_x^{\Psi'}
\]
hence applying Lemma~\ref{lem:prop constant} we deduce $[v_\beta,v_\alpha]\mu_x^{\Psi'}=\mu_x^{\Psi'}$.
Applying Proposition~\ref{prop:general product structure} again we concluded that 
also $[v_\beta,v_\alpha]\mu_x^{\Psi}=\mu_x^{\Psi}$.

Since $\linv_x^{\Psi}$ is a (Hausdorff) closed subgroup of $V^\Psi$ it follows that
\be\label{eq:high-main}
\Bigl[\lsupp_x^{[\alpha]},\lsupp_x^{[\beta]}\Bigr]\subset\linv_x^{\Psi}\text{ almost surely.}
\ee
\end{proof}

\begin{lem}[Cf.~\cite{EK-2}, \S8]\label{lem:Inv-prod-struc}
Let $\mu$ be an $A$-invariant probability measure on $X$.
There is a conull subset $X'\subset X$ with the following property.
Let $\Psi\subset{}_\field\rootsys({\bf A},\Gbf)$ be a {positively} 
closed subset such that $V_\Psi \subset W_G^-(a)$ for some $a$. Then for all $x\in X'$,
if $v=\prod v_\alpha\in\mathcal I_x^{\Psi},$ with $v_\alpha \in V_{[\alpha] }$ for all $[\alpha]\subset\Psi$, then $v_\alpha\in\mathcal I_x^{[\alpha]}$  for all $[\alpha]$.
\end{lem}

\begin{proof}
We say a root
$\alpha \in \Psi$ is {\em exposed} (cf. \cite{EL-Pisa}) if there exists an element $b \in A$ so that
$\alpha(b)=1$ and $|\beta(b)|<1$ for all $\beta \in \Psi \setminus[\alpha]$. If $\Psi$ is as above then clearly it has at least one exposed Lyapunov weight $\alpha$ and that $\Psi ' = \Psi \setminus [\alpha]$ is also positively closed. 
Moreover, for any $v _ \alpha \in V _ {[\alpha]}$ and $v ' \in V _ {\Psi '}$ it holds that $[v_\alpha,v'] \in V _ {\Psi '}$.

Suppose $v _ \alpha v' \in \mathcal I_x^{\Psi}$ with $v _ \alpha \in V _ {[\alpha]}$ and $v ' \in V _ {\Psi '}$. Then
\begin{align*}
\int f(g)d \mu_x^{\Psi}&=\kappa \int f(g_\alpha g')\operatorname{d}\!\mu_x^{V_{[\alpha]}}\operatorname{d}\!\mu_x^{\Psi'}\\
&=\int f(v_\alpha v' g)d \mu_x^{\Psi}\\
&=\kappa\int f(v_\alpha v' g_\alpha g')\operatorname{d}\!\mu_x^{V_{[\alpha]}}\operatorname{d}\!\mu_x^{\Psi'}\\
&=\kappa\int f(v_\alpha g_\alpha v' [v', g_\alpha]  g')\operatorname{d}\!\mu_x^{V_{[\alpha]}}\operatorname{d}\!\mu_x^{\Psi'}
\end{align*}
for some $\kappa$ independent of $f$.

It follows by uniqueness of decomposition that for $\mu_x^{V_{[\alpha]}}$-a.e.\ $g_\alpha$, 
\[
v' [v',g_\alpha]\mu_x^{\Psi'}\propto \mu_x^{\Psi'}
\]
hence by Lemma~\ref{lem:prop constant} we have that $v' [v',g_\alpha] \in \linv_x^{\Psi'}$. It follows that $v_\alpha \mu_x^{V_{[\alpha]}} = \mu_x^{V_{[\alpha]}}$ and $v_\alpha \in \linv_x^{[\alpha]}$. Moreover, as for a.e.~$x$ the identity is in the support of $\mu_x^{V_{[\alpha]}}$ 
by Lemma~\ref{lem:properties}.(3) we have that $v' \in \linv_x^{\Psi'}$. The lemma now easily follows by induction on the cardinality of $\Psi$.
\end{proof}

For any~$W_G^\pm(a)$ we fix some increasing sequence of compact open subgroups~$K_n$
with $W_G^\pm(a)=\bigcup_n K_n$ and some decreasing sequence of compact open
subgroups $O_n\subset K_1$ 
with~$\{e\}=\bigcap_n O_n$. Then any closed subgroup~$\linv<W_G^\pm(a)$
is determined by the finite subgroups~$\linv\cap K_n/O_n<K_n/O_n$,
which allows us to speak of measurability of a subgroup depending
on~$x\in X$.

\begin{lem}\label{lem:pinsker-measurable-stable}
Let $a\in A$. Then $\linv_x^{W_G^\pm(a)}$ is $\pins_a$-measurable.
\end{lem}

\begin{proof}
We prove this for $W^-_G(a)$, the proof in the other case is similar.
There is a full measure set $X'\subset X$ so that whenever $x,wx\in X'${}{, for some $w\in W^-_G(a)$, then} we have
\[
\mu_x^{W^-_G(a)}\propto\mu_{wx}^{W^-_G(a)}w.
\] 
This implies $\linv_x^{W^-_G(a)}=\linv_{wx}^{W^-_G(a)}.$ The lemma
now follows from Proposition~\ref{prop:pins-meas-hull}. 
\end{proof}

\begin{lem}\label{lem:pinsker-measurable} 
Let $\alpha \in {}_\field\rootsys({\bf A},\Gbf)$ be such that $V_{[\alpha]} < W^-_G(a)$. 
Then the subgroup $\linv_x^{[\alpha]}$ is $\pins_{a}$-measurable.
\end{lem} 

\begin{proof}
In view of Proposition~\ref{prop:pins-meas-hull}, it suffices to show that 
$x\mapsto\linv_x^{[\alpha]}$ is constant along $W_G^-(a)$-leaves almost surely, which is an immediate corollary of Lemma~\ref{lem:pinsker-measurable-stable} and Lemma~\ref{lem:Inv-prod-struc}.
\end{proof}


\section{High entropy part of Theorem~\ref{thm:measure-class}}\label{sec:high}
We now start the proof of Theorem~\ref{thm:measure-class}.
{}{Recall that $A$ is the full diagonal subgroup of $G=\SL(d,\field)$.}
Throughout \S\ref{sec:high}--\S\ref{sec:mc-proof}, $\mu$ denotes an ergodic $A$-invariant measure
on $G/\Gamma.$ 

For any $\alpha\in\Phi$ there exists a $\field$-embedding  
$\varphi_\alpha:\SL_2\to\SL_d$ so that $U_\alpha=\varphi_\alpha({}{U^+})$ and $U_{-\alpha}=\varphi_\alpha(U^-)$
where $U^\pm$ denote the upper and lower triangular unipotent subgroups of $\SL_2.$
We let $H_\alpha{}{:}={\rm Im}(\varphi_\alpha).$ 

Let $T$ denote the diagonal subgroup of $\SL_2.$ 
Let $t_\alpha=\begin{pmatrix}\theta&0\\0&\theta^{-1}\end{pmatrix}\in T$ 
be an element 
so that $\alpha(\varphi_\alpha(t_\alpha))=\theta^2$ 
and $\beta(\varphi_\alpha(t_\alpha))=\theta^{\vare}$ with $\vare\in\{-1,0,1\}$ 
for all $\beta\in\rootsys\setminus\{\pm\alpha\}$ {}{where $\theta$ is as in \S\ref{sec:field}.}
Put 
\[
a_\alpha:=\varphi_\alpha(t_\alpha).
\]
Then $U_\alpha\subset W^+_G(a_\alpha).$

Given a root $\alpha\in\rootsys$ we define 
\[
\rootsys^+_\alpha:=\{\beta\in\rootsys:U_\beta\subset W^+(a_\alpha)\},
\]  
and put $\rootsys^-_\alpha=-\rootsys^+_\alpha.$

\begin{lem}\label{lem:root-system-A}
Let $\alpha\in\rootsys$ and let $\beta\in\rootsys^-_\alpha\setminus\{-\alpha\}.$ The following hold.
\begin{enumerate}
\item $\beta+\alpha\in\Phi_\alpha^{+}.$
\item if $\beta+n\alpha\in\Phi$ for some integer $n\geq 1$, then $n=1.$
\item $\alpha\in\rootsys^-_\beta.$ 
\end{enumerate}  
\end{lem} 

\begin{proof}
Assertions (1) and (3) are general facts, which follow from the definitions and hold for any roots system.
Part (2) is a special feature of root systems of type $A$ which is the case we are concerned with here.
\end{proof}

A well-known theorem by Ledrappier and Young~\cite{LedYoung-Ent} relates the entropy,
the dimension of conditional measures along invariant foliations, and Lyapunov
exponents, for a general $C^2$ map on a compact manifold, and in~\cite[\S9]{MarTom}
an adaptation of the general results to flows on locally homogeneous
spaces is provided.
The following is taken from~\cite[Lemma~6.2]{EK-PureApplied}, see also~\cite[Prop.\ 3.1]{EKL}  and \cite{EL-Pisa}. 
For any root $\alpha\in\Phi$ there exists $s_{\alpha}(\mu)\in[0,1]$ 
so that for any $a\in A$ with $|\alpha(a)|\geq1$ we have 
\[
\entropy_\mu(a,U_{\alpha})=s_{\alpha}(\mu)\log|\alpha(a)|
\] 
where $\entropy_\mu(a,U_{\alpha})$
denotes the entropy contribution of $U_\alpha$.
Indeed $s_\alpha(\mu)$ are defined as the local dimension of the leafwise measure along $\alpha$
as we now recall. Define
\[
D_\mu(a_\alpha,U_\alpha)(x)=\lim_{|n|\to\infty}\frac{\log\bigl(\mu_x^{U_\alpha}(a_\alpha^n U_\alpha[1] a_\alpha^{-n})\bigr)}{n},
\]
the limit exists by~\cite[Lemma 9.1]{EK-2}, 
and define $\entropy_\mu(a_\alpha,U_\alpha)=\int D_\mu(a_\alpha,U_\alpha)\operatorname{d}\!\mu,$ 
the entropy contribution of $U_\alpha.$
Since $D_\mu(a_\alpha,U_\alpha)(x)$ is $A$-invariant and $\mu$ is $A$-ergodic, we have 
\[
\mbox{$\entropy_\mu(a_\alpha,U_\alpha)=D_\mu(a,U)(x)$ for $\mu$-a.e.\ $x.$}
\]
Therefore, $s_\alpha(\mu)=\frac{1}{2}D_\mu(a,U)(x)$ for $\mu$-a.e.\ $x.$

Moreover, the following properties hold.  
\begin{enumerate}
\item[($\salpha$-1)] $s_{\alpha}(\mu)=0$ if and only if $\mu_x^{\alpha}$ is delta mass at the identity,
\item[($\salpha$-2)] $s_{\alpha}(\mu)=1$ if and only $\mu_x^{\alpha}$ is the Haar measure on $U_{\alpha},$
\item[($\salpha$-3)] for any $a\in A$ we have 
\[
\entropy_\mu(a)=\sum s_{\alpha}(\mu)\log^+|\alpha(a)|
\]
where $\log^+(\ell)=\max\{0,\log\ell\}.$ 
\end{enumerate}

\medskip

The following is the main result of this section.

\begin{prop}[Cf.~\cite{EL-SplitCase}, Theorem 5.1]\label{prop:high}
Let $\alpha\in\Phi$ be so that $\mu_x^{\alpha}$ is nontrivial 
for $\mu$-a.e.\ $x.$ Then at least one of the following holds. 
\begin{enumerate}
\item $\mu_x^{\beta}=\delta_\id$ for all $\beta\in\rootsys^-_\alpha\setminus\{-\alpha\}$ and $\mu$-a.e.\ $x.$ 
\item $\linv_x^{\pm\alpha}$ are nondiscrete subgroups of~$U_{\pm\alpha}$ for $\mu$-a.e.\ $x.$
\end{enumerate}
\end{prop}

\begin{proof}
	Recall that for~$\SL(d)$ the roots~$\alpha$ can be identified with 
	ordered tuples	of indices~$(i,j)\in\{1,\ldots,d\}$ satisfying~$i\neq j$. 
  We use the local dimensions~$s_\alpha=s_{(i,j)}$ to define a relation
  on~$\{1,\ldots,d\}$. In fact we write~$i\precsim j$ if $i=j$ or $s_{(i,j)}>0$,
  and~$i\sim j$ if $i\precsim j\precsim i$. Lemma~\ref{lem:commutator}
  implies that $\precsim$ is transitiv, i.e.~if~$i\precsim j\precsim k$ then
  also $i\precsim k$ for~$i,j,k\in\{1,\ldots, d\}$. 
  
  It follows that~$\sim$ is an equivalence relation on~$\{1,\ldots,d\}$
  and that~$\precsim$ descends to a partial order on the quotient by~$\sim$.
  Let us write $[i]$ for the equivalence classes with respect to~$\sim$. 
  To simplify matters we may assume (by applying a suitable element 
  of the Weyl group) that for every~$i$
  the equivalence class~$[i]=\{m,m+1,m+2,\ldots,n\}$ consist of consecutive
  indices for some~$m\leq i$ and~$n\geq i$. 
  Moreover, we may assume that~$i\precsim j$ for two indices
  implies either~$i \sim j$ or $i\leq j$. 

  We now prove that~$i\precsim j$ implies~$i\sim j$. Otherwise, we claim
  that we can choose
  a diagonal matrix~$a$ with two different eigenvalues 
  (equal to powers of~$\theta$, {}{see~\S\ref{sec:field}}) such that the leafwise measures of 
  the stable horospherical subgroup $W_G^-(a)$ are nontrivial
  and the leafwise measures of the unstable horospherical subgroup $W_G^+(a)$
  are trivial almost surely. 
  More precisely assuming~$[i]=\{m,m+1,m+2,\ldots,n\}$ (so that
  by the indirect assumption~$j>n$)
  we define~$a$ to be the diagonal matrix with 
  the first~$m$ eigenvalues equal to~$\theta^{(d-m)}$ and the last~$d-m$
  eigenvalues equal to~$\theta^{-m}$. 
  By assumption $s_{(i,j)}>0$, which implies $h_\mu(a)>0$
  by ($s_\alpha$-3), the choice of $a$, and since $i\leq n<j$.
  However, for all $k\leq n<\ell$ we have $s_{\ell,k}=0$ (by
  our ordering of the indices) and hence $h_\mu(a^{-1})=0$
  also by ($s_\alpha$-3). This contradiction
  proves the claim that~$i\precsim j$ implies~$i\sim j$.

  Given a root~$\alpha=(i,j)$ with~$s_\alpha>0$ there are now two options.
  Either~$[i]=\{i,j\}$ or the cardinality of~$[i]$ is at least three. 
  In the first case we have~$s_{(i,\ell)}=s_{(j,\ell)}=s_{(\ell,i)}=s_{(\ell,j)}=0$ 
  for all~$\ell\notin\{i,j\}$ and translating this to the language of 
  roots we obtain (1).
  In the second case let~$\ell\in[i]\setminus\{i,j\}$ and apply 
  Lemma~\ref{lem:commutator} for the roots $(i,\ell),(\ell,j)$ to see 
  that~$\linv^{(i,j)}_x$ (and similarly also~$\linv^{(j,i)}_x$) is a nondiscrete group almost surely.  
\end{proof}

\section{Low entropy part of Theorem~\ref{thm:measure-class}}\label{sec:low}
We use the notation introduced in \S\ref{sec:high}.
In view of Proposition~\ref{prop:high} the following is the standing assumption for the rest of this section. 
There is a root $\alpha\in\rootsys$ so that 
\be\label{eq:assumption-low}
\mbox{$s_\alpha=s_{-\alpha}>0$ and $s_\beta=0$}
\ee
for any~$\beta\in\Phi_\alpha^\pm\setminus\{\alpha,-\alpha\}$.

Let us put 
\[
\Calpha:={}{Z_G(U_\alpha)\cap Z_G(U_{-\alpha})=Z_G(H_\alpha)}.
\]

We have the following.

\begin{lem}[Cf.~~\cite{EKL}, Lemma 4.4(1)]\label{lem:Uef-leaf}
There is a null set $N$ so that for all $x\in X\setminus N$ we have 
\[
\Uef x\cap (X\setminus N)\subset U_{\alpha}x.
\]
In particular, for all $x\in X\setminus N$ if $u\in\Uef$ is so that $ux\in \Uef x\cap (X\setminus N)$ and 
$\mu_x^{\alpha}=\mu_{ux}^{\alpha},$ then $u\in\linv_x^{\alpha}.$ 
\end{lem}

\begin{proof}
In view of Lemma~\ref{lem:prod-struc} there is a null set $N_1$ so that 
for all $x\in X\setminus N_1$ we have that $\mu_x^{\Uef}$
is a product of the leafwise measures $\mu_x^{\beta}$ for all $U_\beta\subset \Uef.$
By~\eqref{eq:assumption-low} 
it follows that 
\be\label{eq:U-alpha-1}
\supp\Bigl(\mu_x^{\Uef}\Bigr)=\supp(\mu_x^{\alpha})\text{ for all $x\in X\setminus N_1$}
\ee

Recall also that there is a null set $N_2$ so that if $x,ux\in X\setminus N_2$ for some $u\in \Uef,$ then 
\be\label{eq:U-alpha-2}
\mu_x^{\Uef}\ppor\mu_{ux}^{\Uef}u.
\ee
Let $x\in X\setminus (N_1\cup N_2),$ then by~\eqref{eq:U-alpha-1} we have 
$\supp\Bigl(\mu_x^{\Uef}\Bigr)\subset U_\alpha.$ Therefore by~\eqref{eq:U-alpha-2}
we get $u\in U_\alpha,$ this finishes the proof of the first claim if we require that $N\supsetneq N_1\cup N_2.$

To see the last assertion, let $N_3\subset X$ be a null subset 
so that $\mu_{ux}^{\alpha}u\ppor\mu_x^{\alpha}$ for all $x\not\in N_3$.
Set $N=N_1\cup N_2\cup N_3$. Let $x\in X\setminus N$ and let $u$ be as in the statement. 
In view of the first part in the lemma, we have $u\in U_{\alpha}.$
Our assumption and the fact that $U_\alpha$ is a commutative group
give 
\[
u\mu_x^{\alpha}=\mu_{ux}^{\alpha}u\ppor\mu_x^{\alpha}.
\]

Now one argues as in the proof Lemma~\ref{lem:commutator}
and gets $u\in\linv_x^{\alpha}.$
\end{proof}

We also recall the following definition from~\cite{EL-GenLow}.
 
\begin{definition}\label{def:aligned}
Let $H,Z\subset G$ be closed subgroups of $G.$
We say the leafwise measures $\mu_x^H$ are {\it locally $Z$-aligned
modulo $\mu$} if for every $\vare>0$ and neighborhood $\nbhd_\id^Z\subset Z$
of the identity, there exists a compact set $\cpct$ with $\mu(\cpct)>1-\vare$ 
and some $\delta>0$ so that for every $x\in \cpct$ we have
\[
\{y\in \cpct: \mu_x^{H}=\mu_y^{H}\}\cap \nbhd_x(\delta)\subset \nbhd_\id^Z x. 
\]
\end{definition}

The following is a direct corollary of the main result of~\cite{EL-GenLow}, proved there explicitly also
for the positive characteristic case. 

\begin{thm}[Cf.~\cite{EL-GenLow}, Theorem 1.4]\label{thm:gen-low}
Under the  assumption~\eqref{eq:assumption-low} {one of} the following holds.
\begin{enumerate}
\item[(LE-1)] $\mu_x^{\alpha}$ is locally $\Cef$-aligned modulo $\mu$.
\item[(LE-2)] There exists an $a_\alpha$-invariant subset $\Xinv(\alpha)\subset X$ with $\mu(\Xinv(\alpha))>0$
so that for all $x\in\Xinv(\alpha)$ there is an unbounded sequence 
$\{u_{x,m}\}\subset \Uef$ such that $\mu_x^{\alpha}=\mu_{u_{x,m}x}^{\alpha}.$
\end{enumerate}
\end{thm}


\section{Proof of Theorem~\ref{thm:measure-class}}\label{sec:mc-proof}
Recall the notation in \S\ref{sec:field},
in particular, $\field=\gfield_v$ where $\gfield$ is a global function field and 
$v$ is  a place of $\gfield$ {}{and we work with the maximal torus $\bf A$.}
Throughout, $\Gamma\subset\SL(d,\field)$ is a lattice of inner type, see~\S\ref{sec:innerforms}.

Put $\GL(n,\order)_m=\ker\Bigl(\GL(n,\order)\to\GL(n,\order/\unif^{-m}\order)\Bigr).$

\begin{lem}[Cf.~\cite{EKL}, Lemma 5.3]\label{lem:diagonalizable}
For any positive integer $n$ there exists some $m=m(n)\geq 1$ 
with the following property. 
Let $a={\rm diag}(a_1,\ldots,a_n)$ with
\[
 |v(a_i)-v(a_j)|>m\text{ for all }i\neq j. 
 \]
Then $ga$ is diagonalizable over $\field,$
for all $g\in\GL(n,\order)_m$.  
Moreover, if $a'_1,\ldots,a'_n$ are the eigenvalues of $ga,$ then 
it is possible to order them so that $v(a_i)=v(a'_i)$ for all $i.$
\end{lem}

\begin{proof}
Let $\tilde\field_n$ be the composite 
of all field extensions of $\field$ of degree at most $n!.$ 
Then, the characteristic polynomial of any element in $\GL(n,\field)$ splits over $\tilde\field_n.$
Moreover, $\tilde\field_n$ is a local field, i.e.\ $\tilde\field_n/\field$ is a finite extension.
We let $v$ denote the unique extension of $v$ to $\tilde\field_n.$

We begin with the following observation. 
There is some $m_n\geq 1$ so that every $g\in\GL(n,\order)_{m_n}$ can be decomposed as
$g=g^-g^0g^+$ with $g^\pm\in W^\pm\cap\GL(n,\order)_{1}$ and $g^0\in A\cap\GL(n,\order)_{1},$ 
where $W^+$ (resp.\ $W^-$) is the group of upper (resp.\ lower) triangular unipotent matrices.
Indeed, in view of~\eqref{opposite-horo}, the product map is a diffeomorphism 
from 
\[
(W^-\cap\GL(n,\order)_{1})\times (A\cap\GL(n,\order)_{1})\times (W^+\cap\GL(n,\order)_{1})
\] 
onto its image. Therefore the claim follows from the inverse function theorem.

We show the lemma holds with $m=m_n.$
First note that after conjugating by a permutation matrix we can assume $v(a_1)>\ldots>v(a_n).$
Let $g\in\GL(n,\order)_m$ and let $b_1,\ldots,b_n$ be the eigenvalues of $ga$ listed with multiplicity 
and ordered so that $v(b_1)\geq\cdots\geq v(b_n).$   
Note that $b_i\in\tilde\field_n$ for all $1\leq i\leq n.$

Let $\|\;\|$ be the max norm on the $i$-th exterior power $\wedge^i\tilde\field_n^n,$  
with respect to the standard basis $\{e_{j_1}\wedge\cdots\wedge e_{j_i}\}.$
Denote by $\|\;\|$ the operator norm of the action of $\GL(n,\tilde\field_n)$ 
on $\wedge^i\tilde\field_n^n$ for $1\leq i\leq n.$ 

Choosing a basis of $\tilde\field_n^n$ consisting of the generalized eigenvectors
for $ga$ we get 
\be\label{eq:norm-eigen}
\lim_{\ell}\|\wedge^i(ga)^\ell\|^{1/\ell}=|b_1\cdots b_i|\; \text{ for all $i.$}
\ee

We now claim
\be\label{eq:wave-front}
\|\wedge^i(ga)^\ell\|=\|\wedge^i a^\ell\|=|a_1\cdots a_i|^\ell\; \text{ for all $\ell.$}
\ee
The second equality in the claim is immediate. 
To see the first equality 
note that if $g_1,g_2\in \GL(n,\order)_{m},$ then
\[
g_1 ag_2^-g_2^0g_2^+a=g_1(ag_2^-a^{-1})a^2g_2^0(a^{-1}g_2^+a).
\]
Moreover, since $g^\pm\in\GL(n,\order)_{1}$ and $v(a_i)-v(a_{i+1})>m$ for all $i,$
we have $ag_2^-a^{-1}$ and $g_2^0a^{-1}g_2^+a$ belong to $\GL(n,\order)_{m}.$ 

Using this we get
\[
(ga)^\ell=g_\ell a^\ell a'_\ell g'_\ell
\] 
where $g_\ell,g_\ell'\in\GL(n,\order)_{m}$ and $a'_\ell\in\GL(n,\order)_{1}$ for all $\ell.$
This implies~\eqref{eq:wave-front}.

Now~\eqref{eq:norm-eigen} and~\eqref{eq:wave-front} imply that 
$v(a_i)=v(b_i)$ for all $1\leq i\leq n,$ in particular,  $v(b_i)\neq v(b_j)$ whenever $i\neq j$. 
This implies $b_i$'s are distinct and hence $ga$ is a semisimple element.
We now show $b_i\in\field$ for all $i.$
Recall that $b_1,\ldots,b_n$ are roots of the characteristic polynomial of 
$ga$ which is polynomial with coefficients in $\field.$ 
For every $1\leq i\leq n$ let ${\rm Gal}(b_i)=\{b_j: b_j$ is a Galois conjugate of $b_i\}.$ 
Then $\{b_1,\ldots, b_n\}$ is a disjoint union of $\sqcup_{j=1}^r{\rm Gal}(b_{i_j})$ for
some $\{i_1,\ldots,i_r\}\subset\{1,\ldots,n\}.$ 
Since $v(b_i)\neq v(b_j)$ whenever $i\neq j$ 
and Galois automorphisms preserve the valuation, 
we get that ${\rm Gal}(b_i)=\{b_i\}$ for all $i.$ 
This establishes the final claim in the lemma. 
\end{proof}

\begin{prop}\label{prop:exceptional-return}
Recall that $\Gamma$ is an inner type lattice. Then $\mu_x^{\alpha}$ is not 
locally $\Cef$-aligned modulo $\mu.$ In particular, under the assumption~\eqref{eq:assumption-low}, we have that 
{\rm (LE-2)} in Theorem~\ref{thm:gen-low} holds.
\end{prop}

\begin{proof}
We recall the argument from the proof of Theorem 5.1 in~\cite{EKL}. 
Let $m$ be large enough so that the conclusion of Lemma~\ref{lem:diagonalizable} 
holds with $n=d-2.$ Without loss of generality we may assume 
$\alpha({\rm diag}(a_1,\ldots,a_d))=a_1a_2^{-1}.$ Define 
\[
\tilde{\mathfrak B}=\left\{\begin{pmatrix}\fele & 0 & 0\\ 0 & \fele & 0\\ 0 & 0 & C\end{pmatrix}: \fele\in 1+\unif^{-2}\order, C\in\GL(d-2,\order)_m\right\}{\subset \GL(d,\order)}.
\]
Put $\mathfrak B:=\tilde{\mathfrak B}\cap Z_\alpha;$ we note that $\mathfrak B$ is a compact open subgroup of~$Z_\alpha$. 

Let $a={\rm diag}(a_2, a_2, a_3,\ldots,a_d)\in A\cap Z_\alpha$ 
with $v(a_2)\neq 0$, and $|v(a_i)-v(a_j)|>m$ for all $i>j\geq 2$. In particular, we have $\alpha(a)=1$. 

Suppose (LE-1) holds. Then by Poincar\'{e} recurrence 
for $\mu$-a.e.\ $g\Gamma\in G/\Gamma$ there exist a sequence $\ell_i\to\infty$ so that 
\[
a^{\ell_i}g\Gamma\in \mathfrak Bg\Gamma\;\text{ for all $i.$}
\]
Hence for all $i$ there exist some $\gamma_i\in\Gamma$ and some $h_i\in\mathfrak B$ so that
$
h_ia^{\ell_i}=g\gamma_ig^{-1}.
$
Now Lemma~\ref{lem:diagonalizable} implies the following. If $\ell_i$ is large enough
and we write
\be\label{eq:g-gamma-g-1}
g\gamma_ig^{-1}=h_ia^{\ell_i}=\begin{pmatrix}\fele_i & 0 & 0\\ 0 & \fele_i & 0\\ 0 & 0 & D_i\end{pmatrix},
\ee
then $D_i$ is diagonalizable whose eigenvalues have the same valuation as $a_j^{\ell_i}$ for all $3\leq j\leq d.$  
Dropping the few first terms, if necessary, we assume that~\eqref{eq:g-gamma-g-1} holds for all $i$.

Since $\Gamma$ is an inner type lattice, there exists a central simple algebra $B$ over $\gfield$ so that $\Gamma$
is commensurable with $\Lambda_B$, see~\S\ref{sec:innerforms}.
There exists some $i$ (which we fix) and infinitely many $j$'s so that $\hat\gamma_j:=\gamma_j\gamma_i^{-1}\in\Lambda_B$. 
We have
\[
g\hat\gamma_jg^{-1}=h_ja^{\ell_j-\ell_i}h_i^{-1};
\]
hence if $\ell_j-\ell_i$ is large enough, we get from $h_j,h_i^{-1}\in\mathfrak B\subset \GL(n,\order)_{1}$ that
\[
g\hat\gamma_jg^{-1}=\begin{pmatrix}\fele & 0 & 0\\ 0 & \fele & 0\\ 0 & 0 & D\end{pmatrix}
\]
where $D$ is diagonalizable whose eigenvalues have the same valuation as $a_j^{\ell_j-\ell_i}$ for all $3\leq j\leq d.$
Indeed after conjugation by $h_i^{-1}$
we may apply Lemma~\ref{lem:diagonalizable}.

Altogether, (LE-1) in Theorem~\ref{thm:gen-low} 
implies that
there exists an element $\gamma\in\Lambda_B$ with the following properties 
\begin{itemize}
\item $\gamma$ is a semisimple element, 
\item no eigenvalue of $\gamma$ is a root of unity,
\item all of the eigenvalues of $\gamma$ are simple except exactly one eigenvalue which has multiplicity 2.
\end{itemize}

We now claim that none of the eigenvalues of $\gamma$ lies in $\gfield.$
To see this, assume that $\gamma$ has an eigenvalue $\sigma\in\gfield$. 
Recall from the definition of $\Lambda_B$ in~\S\ref{sec:innerforms} 
that $\Lambda_B$ is bounded in $\SL_{1,B}(\gfield_w)$ for all $w\neq v$.  
In particular, $w(\sigma)=0$ -- else the group generated by $\gamma$ in $\SL_{1,B}(\gfield_w)$ would be unbounded. This in view of the product formula implies that $v(\sigma)=0$. Hence $\sigma$ is a root of unity which is a contradiction.

Since $\gamma\in\Lambda_B$, by~\eqref{eq:char-poly-inner} we have that the coefficients of 
the characteristic polynomial of $\gamma$ are in $\gfield$.
This and the fact that $\gamma$ is semisimple imply that there exists a finite separable extension $\tilde \gfield$ of $\gfield$ 
which contains the eigenvalues of $\gamma,$ see~\cite[4.1(c)]{Borel-AlgGrBook}.
Using the above claim thus we get that
the eigenvalue with multiplicity 2 is not in $\gfield$ and is separable over $\gfield.$ 
Since any Galois conjugate of this eigenvalue is also 
an eigenvalue of $\gamma$ with the same multiplicity,  
we get a contradiction with the fact that $\gamma$ has only one non-simple eigenvalue.  
\end{proof}


\subsection{Pinsker components have non-trivial invariance}\label{sec:pinsk-inv}
We begin with the following corollary of the results in~\S\ref{sec:high} and~\S\ref{sec:low}.

\begin{cor}\label{cor:linv-nondisc}
Under the assumptions of Theorem~\ref{thm:measure-class} we have the following.
There exists some $\alpha\in\rootsys$ and a $\mu$-conull 
subset $\Xinv(\alpha)\subset X$ 
so that $\linv_x^{\pm\alpha}$ are nondiscrete for all $x\in\Xinv(\alpha)$.
\end{cor}

\begin{proof}
	Since $h_\mu(a)>0$ for some $a\in A$ there exists some $\alpha\in\Phi$ with $s_\alpha>0$.
In view of Proposition~\ref{prop:high} 
the claim in the corollary holds true almost surely 
unless $\alpha$ satisfies~\eqref{eq:assumption-low}.
 
However, in this case Theorem~\ref{thm:gen-low} and Proposition~\ref{prop:exceptional-return}
imply that (LE-2) must hold true. 
Put $X'=\{x\in X: \linv_x^{\pm\alpha}\text{ is nontrivial}\}.$ 
By (LE-2) and Lemma~\ref{lem:Uef-leaf} we get that $X'$ has positive measure.
Moreover, $X'$ is $A$-invariant in view of Lemma~\ref{lem:properties}(4).
Since $\mu$ is $A$-ergodic we get that~$\mu(X')=1$. 
Now choose~$\ell\in\mathbb{Z}$ such that~$X'_\ell=\{x\in X':\linv_x^{\pm\alpha}\cap U_{\pm\alpha}[\ell]$ is nontrivial$\}$ satisfies~$\mu(X'_\ell)>0$.
Applying ergodicity and the pointwise ergodic theorem
we see that a.e.~$x\in X$ satisfies that there exists some~$a\in A$ and
infinitely many~$n\geq 0$ and infinitely many~$n \leq 0$ such that~$a_\alpha^nax\in X'_\ell$. Using Lemma~\ref{lem:properties}.(4) this implies the corollary.
\end{proof}

Throughout the rest of this section, we fix some root $\alpha$ 
so that the conclusion of Corollary~\ref{cor:linv-nondisc} holds true, and put $\Xinv:=\Xinv(\alpha).$

For any root $\beta$ 
let $A_\beta$ denote the one parameter diagonal subgroup which is the {}{group of $\field$-points of the }Zariski closure of the group
generated by $a_\beta.$
For the sake of notational convenience we will denote $A_\beta=\{\check\beta(t):t\in k^\times\}$
where $a_\beta=\check\beta(\unif).$


Recall that $V_{[\alpha]}$ is contained in $W^+_G(a_\alpha)$.
For the rest of this section denote the Pinsker $\sigma$-algebra $\pins_{a_\alpha}$ 
for $a_\alpha$
simply by $\pins.$ We further take a decomposition
\be\label{eq:pinsker-cond}
\mu=\int_X\mu_x^{\pins}\operatorname{d}\!\mu(x),
\ee
where $\mu_x^{\pins}$ denotes the $\pins$ conditional measure for 
$\mu$ almost every $x\in X.$ 

Since $\mu$ is $A$-invariant and $A$ commutes with $a_\alpha,$ the $\sigma$-algebras $\pins$ 
is $A$-invariant. Hence we get 
\be\label{A-Aalpha-pins}
a\mu_x^\pins=\mu_{ax}^\pins\;\text{ for $\mu$-a.e.\ $x\in X$}
\ee
Recall the definition of $H_\alpha=\varphi_\alpha(\SL(2,\field))$ from the beginning of~\S\ref{sec:high}. 
For every $x\in X$ we put 
\be\label{eq:def-Hcal}
\Hcal_x:=\{g\in H_\alpha: g\mu_x^\pins=\mu_x^\pins\}.
\ee 
It follows from~\eqref{A-Aalpha-pins} that
\be\label{eq:Hcal-A-equiv}
\Hcal_{ax}=a\Hcal_xa^{-1}
\ee
for all $a\in A$ and $\mu$-a.e.\ $x.$

\begin{cor}\label{cor:pinsker-cond-inv1}
${\langle \linv_x^{\alpha},\linv_x^{-\alpha}\rangle}$ is Zariski dense in $H_\alpha$ 
as a $\field$-group for $\mu$-a.e.\ $x\in\Xinv$. Moreover,
$
{\langle \linv_x^{\alpha},\linv_x^{-\alpha}\rangle}\subset\Hcal_x.
$ 
\end{cor}

\begin{proof}
The first claim follows from Corollary~\ref{cor:linv-nondisc}.
To see the second claim, note that by Lemma~\ref{lem:pinsker-measurable}, we know that $\linv_x^{\pm\alpha}$ 
is measurable with respect to $\pins$. 
Equivalently, the groups
$\linv_x^{\pm\alpha}$ are (almost surely) constant on the atoms
of a countably generated $\sigma$-algebra~$\pins'$ that is equivalent to~$\pins$.
We now decompose $\mu$ as in \eqref{eq:pinsker-cond} into
conditional measures for the $\sigma$-algebra $\pins'$
and take the leafwise measures of $\mu_x^{\pins'}$ for the subgroup~$U_{\alpha}$.
However, Proposition~\ref{prop:pins-meas-hull} implies that we may assume
the atoms with respect to~$\pins'$ are unions of~$U_\alpha$-orbits.
This implies in turn for
the leafwise measure that~$(\mu_x^{\pins'})_y^{U_\alpha}=\mu_{y}^\alpha$
for~$\mu_x^{\pins'}$-a.e.~$y$ and~$\mu$-a.e.~$x$ (see~\cite[Prop.~5.20]{Ergodic-Number-EW}
and \cite[Prop.~7.22]{EL-Pisa} for a similar argument). 
Fixing one such~$x$ we obtain that~$(\mu_x^{\pins'})_y^{U_\alpha}$ is almost surely
invariant under $\linv_y^{\alpha}=\linv_x^{\alpha}$. However, this implies
by the relationship between the measure and its leafwise measures
that~$\mu_x^{\pins'}$ is invariant under $\linv_x^{\alpha}$. 
Since $\mu_x^{\pins'}=\mu_x^{\pins}$ almost surely we may apply the same argument for
$\linv_x^{-\alpha}$. 
Therefore, $\linv_x^{\pm\alpha}\subset\Hcal_x$
for $\mu$-a.e.\ $x.$ 
\end{proof}

\subsection{Algebraic structure of \texorpdfstring{$\Hcal_x$}{H\_x}}\label{sec:hilbert-LarPin2}
Recall from the beginning of~\S\ref{sec:high} that $H_\alpha=\varphi_\alpha(\SL_2(\field)).$
Put $U_{\pm\alpha}(\vorder)=\varphi_\alpha({\bf U}^\pm(\vorder))$ where ${\bf U}^+$ (resp.\ ${\bf U}^-$) denotes the group of upper (resp.\ lower) triangular unipotent matrices in $\SL_2.$ 

Note that $H_{\alpha}=\langle U_{\alpha},U_{-\alpha}\rangle.$ 
By Corollary~\ref{cor:pinsker-cond-inv1} for $\mu$-a.e.\ $x$
we have ${\langle \linv_x^{\alpha},\linv_x^{-\alpha}\rangle}\subset\Hcal_x.$ 
Define
\be\label{eq:def-Qx}
\Qcal_x:=\langle\Hcal_x\cap U_\alpha(\vorder),\Hcal_x\cap U_{-\alpha}(\vorder)\rangle.
\ee

Put
\be\label{eq:X-pins}
X_\pins:=\{x\in X: \Qcal_x\text{ is Zariski dense in }H_{\alpha}\text{ and $\Qcal_x\cap U_{\pm\alpha}$ are infinite}\}.
\ee

Corollary~\ref{cor:pinsker-cond-inv1} and the above definitions 
imply that $X_\pins\cap\Xinv$ is conull in $\Xinv.$
In particular, Corollary~\ref{cor:linv-nondisc} implies that $\mu(X_\pins)=1.$

Note that for all $x\in X_\pins,$ the group $\Qcal_x$ satisfies the conditions
of Theorem~A.1 in Section~\ref{sec:hilbert-LarPin1}. For any $x\in X_\pins$ define
\[
\field'_x:=\text{the field generated by $\{{\rm tr}(\rho(g)):g\in\mathcal Q_x\}$}
\]
and put
\be\label{eq:def-kx}
{\field_x:=\begin{cases}\field'_x&\text{ if ${\rm char}(\field)\neq 2$},\\ \{c:c^2\in\field'_x\}&\text{ if ${\rm char}(\field)= 2$}.\end{cases} }
\ee
Theorem~A.1 then implies that there exist 
\begin{enumerate}
\item[(C-1)]  a unique (up to a unique isomorphism) 
$\field$-isogeny $\varphi_x: {\SL_2}\times_{\field_x}\field\to \SL_2$
whose derivative vanishes nowhere, and
\item[(C-2)]  some non-negative integer $m_x$
\end{enumerate}  
so that 
\be\label{eq:Qx-structure}
\varphi_x(\SL(2,\order_x)_{m_x})\subset\Qcal_x\subset{\varphi_x}(\SL(2,\field_x))
\ee 
where $\order_x$ is the ring of integers in $\field_x$ and
\[
\SL(2,\order_x)_m:=\ker\Bigl(\SL(2,\order_x)\to\SL(2,\order_x/\varpi_x^m\order_x)\Bigr).
\] 
with $\varpi_x$ a uniformizer in $\order_x.$

Let us put
\be\label{eq:def-Ex}
E_x:={\varphi_x}(\SL(2,\field_x)).
\ee

We will use without further remark the following, which is a consequence of the implicit function theorem.
The group generated by ${\bf U}^\pm(\varpi_x^m\order_x)$ is an open subgroup of $\SL(2,\order_x)_m,$
e.g.\ a direct computation yields this group contains $\SL(2,\order_x)_{2m}.$   

\begin{lem}\label{lem:kx-Qx-Borel}\leavevmode
{}{Consider the Borel $\sigma$-algebra arising from the Chabauty topology on closed subgroups of $(\field,+)$ and $\SL(d,\field)$.}
\begin{enumerate}
\item The map $x\mapsto\field_x$ is a Borel map on $X_\pins.$
\item The equation \eqref{eq:def-Ex} defines a Borel map, $x\mapsto E_x,$ on $X_\pins$.
\end{enumerate}
\end{lem}

\begin{proof}
The map $x\mapsto\Qcal_x$ is a Borel map from a conull subset of $X$ 
into the set of closed subgroups of $H_{\alpha}(\vorder).$
This and~\eqref{eq:def-kx} imply that $x\mapsto\field_x$ is a Borel map on the conull set $X_\pins,$
as we claimed in (1).

By part (1) the map $x\mapsto\field_x$ is a Borel map.
Also recall from Lemma~\ref{lem:E-generated}(1) that $E_x=\langle E_x\cap U_{\alpha}, E_x\cap U_{-\alpha}\rangle.$
Therefore, part (2) follows if we show the map $x\mapsto E_x\cap U_{\pm\alpha}$ is a Borel map.
Note, however, that if we realize $U_{\pm\alpha}=\{u_r:r\in\field\}$ as a $\field_x$-vectors space, then 
$E_x\cap U_{\pm\alpha}=\{u_r:r\in\field_x\}$ is a one dimensional $\field_x$-subspace of $U_{\pm\alpha}$, respectively.
Hence, 
\[
E_x\cap U_{\pm\alpha}=\{u_{rr'}: r\in\field_x, u_{r'}\in\Qcal_x\cap U_{\pm\alpha}),
\]  
which implies the claim.
\end{proof}

 \begin{lem}\label{lem:kx-constant} \leavevmode
\begin{enumerate}
\item The map $x\mapsto\field_x$ is essentially constant. 
\item The map $x\mapsto E_x$ is an $A$-equivariant Borel map on a conull subset of $X.$
\end{enumerate}  
\end{lem}

\begin{proof}
We claim $\field_{x}\subset\field_{ax}$ for all $a\in A.$ 
First let us note that by symmetry, this also implies that $\field_{ax}\subset\field_x.$
Therefore, it implies that the map $x\mapsto\field_x$ is $A$-invariant; since $\mu$ is $A$-ergodic
we get part (1).

We now show the claim. Let $m_x$ be as in (C-2). 
Recall from~\eqref{eq:Hcal-A-equiv} that there is a full
measure set $X'\subset X$ so that for all $x\in X'$ and all $a\in A$ 
we have $\Hcal_{ax}=a\Hcal_xa^{-1}.$
Now for any $a$ there exists some $m_{x,a}\geq m_x$
so that if $m\geq m_{x,a},$ then
\be\label{eq:deep-levels}
a\varphi_x(\SL(2,\order_x)_{m})a^{-1}\subset \mathcal Q_{ax}.
\ee

Define $\sfield_x(m)$ to be the field generated by 
$
\{{\rm tr}(\rho(g)):g\in\varphi_x(\SL(2,\order_x)_{m})\}. 
$ 
Then 
\be\label{eq:trace-field}
\text{$\sfield_x(m)=\field_x$ for all $m\geq m_x.$}
\ee 
Indeed this is true for the field generated by $\{{\rm tr}(\rho(g)):g\in\SL(2,\order_x)_{m}\}.$
Since $\varphi_x$ has nowhere vanishing derivative 
and there are no nonstandard isogenies for type $A_1,$~\cite[Prop.\ 1.6]{Pink-Compact}, 
we get $\rho_1=\rho_2\circ \varphi_x$
where $\rho_1$ and $\rho_2$ are the adjoint representation on the source and the target of $\varphi_x$.
This implies~\eqref{eq:trace-field}.

It follows from~\eqref{eq:deep-levels} and~\eqref{eq:trace-field} that $\field_x\subset \field_{ax},$
as we claimed.  

Let us now prove part (2). By part (1) there is an $A$-invariant conull set $X'$ 
and a subfield $\field'$ so that $\field_x=\field'$ for all $x\in X'.$
Let $\order'$ denote the ring of integers in $\field'.$

We note that the same proof as in the proof of Lemma~\ref{lem:kx-Qx-Borel}(2)
implies that $E_x\cap U_{\pm\alpha}$ is the Zariski closure 
of $C\cap U_{\pm}$ in $\mathcal R_{\field/\field'}(\SL_d)$
for any nontrivial open subgroup $C$ of $\Qcal_x.$ 

Let now $a\in A$ and $x\in X'.$ Then by~\eqref{eq:deep-levels} we have
\[
a\varphi_x(\SL(2,\order')_{m})a^{-1}\subset\Qcal_{ax}
\] 
for all $m\geq m_{x,a}.$
Since $a\Hcal_xa^{-1}=\Hcal_{ax}$ and $\varphi_x(\SL(2,\order')_{m})$ is open in $\Qcal_x$
by~\eqref{eq:Qx-structure}, we thus get that $a\varphi_x(\SL(2,\order')_{m})a^{-1}$
is open in $\Qcal_{ax}$ for all $m\geq m_{x,a}.$

Since $U_{\pm\alpha}$ are 
normalized by $A$, for all $a\in A$ and all $m\geq m_{x,a}$ we have
\[
a\varphi_x(\SL(2,\order')_{m})a^{-1}\cap U_{\pm\alpha}=a\bigl(\varphi_x(\SL(2,\order')_{m})\cap U_{\pm\alpha}\bigr)a^{-1}.
\]
Taking Zariski closure in $\mathcal R_{\field/\field'}(\SL_d)$ we get that 
\[
a\bigl(E_x\cap U_{\pm\alpha}\bigr)a^{-1}=E_{ax}\cap U_{\pm\alpha}.
\]
This and Lemma~\ref{lem:E-generated}(1) imply the claim.
\end{proof}

\begin{prop}\label{lem:inv-sl2}\label{prop:inv-sl2}
For $\mu$-a.e. $x\in X_\pins$ we have
$
E_x\subset\Hcal_x.
$
\end{prop}

\begin{proof}
Let $x\in X_\pins$ and put $A'_{x}:=E_x\cap A.$ 
In view of Lemma~\ref{lem:kx-constant}(2) we have 
\be\label{eq:A'-const}
A'_{ax}=E_{ax}\cap A=aE_xa^{-1}\cap A=a(E_x\cap A)a^{-1}=A'_x
\ee
for $\mu$-a.e.\ $x$ and all $a\in A.$ Since $\mu$ is $A$-ergodic we get that
$x\mapsto A'_x$ is essentially constant. Let us denote by $A'$ this essential value.

Then by Lemma~\ref{lem:E-generated}(2)
we have $A'$ is an unbounded subgroup of $A_{\alpha}=H_{\alpha}\cap A.$ 
The group $A_\alpha$ is a one dimensional $\field$-split $\field$-torus, 
therefore, $A_\alpha/A'$ is compact. 
For any~$s\in\field$ we let~$\check\alpha(s)\in A_\alpha$ be the cocharacter associated to~$\alpha$
and evaluated at~$s$, i.e.\ $\check\alpha(s)$ is the diagonal matrix
with eigenvalues~$s,s^{-1}$ and $1$ with multiplicity~$d-2$ so that~$\alpha(\check\alpha(s))=s^2$.
This implies that there exist some $\ell>0$ and some $\fele\in\vorder^\times$,
so that if we put $s:=\unif^{\ell}\fele,$ then $\check\alpha(s)\in E_x.$ 
In particular, $\check\alpha(s)$ normalizes both $E_x\cap U_\alpha$ and $E_x\cap U_{-\alpha}.$

For every $\vare>0$ there is subset ${X_\pins}(\vare)\subset{X_\pins}$ with 
$\mu({X_\pins}(\vare))>1-\vare$ so that the map
\[
x\mapsto\mu_x^{\pins}
\]
is continuous on ${X_\pins}(\vare).$
Now by Poincar\'{e} recurrence, 
for $\mu$-a.e.\ $x\in{X_\pins}(\vare)$ there is a sequence $n_{x,i}\to\infty$ so that 
$\check\alpha(s^{n_{x,i}})\in{X_\pins}(\vare)$ for all $i$ and
$\check\alpha(s^{n_{x,i}})x\to x.$  
Then 
\[
\lim_{i\to\infty}\Hcal_{\check\alpha(s^{n_{x,i}})x}\subset\Hcal_x.
\]

Recall from~\eqref{eq:Qx-structure} that $\Qcal_x\cap U_\alpha$
contains an open compact subgroup of $E_x\cap U_\alpha.$
Therefore, using~\eqref{eq:Hcal-A-equiv} we get that
\[
E_x\cap U_{\alpha}\subset\lim_{i\to\infty}\check\alpha(s^{n_{x,i}}) (\Qcal_x\cap U_{\alpha})\check\alpha(s^{-n_{x,i}})\subset\lim_i\Hcal_{\check\alpha(s^{n_{x,i}})x}\subset
\Hcal_x,
\]
for $\mu$-a.e.\ $x\in{X_\pins}(\vare).$ 
Choosing a sequence $\vare_n\to0$ we get that $E_x\cap U_{\alpha}\subset\Hcal_x$ for
$\mu$-a.e.\ $x\in X_\pins.$

Similarly, we get $E_x\cap U_{-\alpha}\subset\Hcal_x$ for $\mu$-a.e.\ $x\in{X_\pins}$.
Recall from Lemma~\ref{lem:E-generated}(1) that $E_x$ is generated by 
$E_x\cap U_{\pm\alpha}.$ Therefore $E_x\subset\Hcal_x$ for $\mu$-a.e.\ $x\in{X_\pins}.$
\end{proof}

\subsection{Applying the measure classification for semisimple groups}

We now apply the measure classification theorem due to Alireza Salehi-Golsefidy 
and the third named author (Theorem B from Section \ref{sec:ratner-sl2}). 

\begin{lem}\label{applyingAA}
	Let $\mu$ be as in Theorem \ref{thm:measure-class}. Then
	there exist a closed infinite subfield $\sfield<\field$ and a {}{smooth} 
	algebraic $\sfield$-subgroup ${\bf M}<\mathcal R_{\field/\sfield}(\SL_d)$ such that ${\bf M}(\sfield)\cap\Gamma$ is Zariski dense in $\bf M$ over $\sfield$, and 
	a noncentral cocharacter $\lambda:\mathbf{G}_m\rightarrow\bf M$ over $\sfield$ so that the topological group
	\[
	  L=\overline{M^+(\lambda)({\bf M}(\sfield)\cap\Gamma)}
	\]
	satisfies that $L/(L\cap\Gamma)$ has finite volume.
	Moreover, for $\mu$-a.e.\ $x,$ the $E_x$-ergodic component of $\mu_x^{\pins}$ equals $h\nu_L$ 
	for some~$h\in \SL(d,\field)$ so that $x=h\Gamma$ and $\nu_L$
	is the homogeneous measure on $L/(L\cap\Gamma)$.
	\end{lem}

\begin{proof}
Let $\field'$ denote the essential value of the map $x\mapsto\field_x,$ see Lemma~\ref{lem:kx-constant}(1).
In view of Proposition~\ref{prop:inv-sl2} 
for $\mu$-a.e.\ $x$ the measure $\mu_x^{\pins}$ is invariant under 
$E_x.$

Since the $\sigma$-algebra $\pins$ is $A$-invariant we have 
$a\mu_x^\pins=\mu_{ax}^\pins$ for all $a\in A$ and $\mu$-a.e.\ $x.$
Moreover, by Lemma~\ref{lem:kx-constant}(2)
we have $E_{ax}=aE_xa^{-1}$ for $\mu$-a.e.\ $x.$
Therefore, if we let 
\begin{equation}\label{eq:pinsmain}
	\mu_{x}^\pins=\int \nu_z\operatorname{d}\!\mu_x^\pins(z)
\end{equation}
be the ergodic decomposition of $\mu_{x}^\pins$ with respect to $E_x$ 
(where for $\mu_x^{\pins}$-a.e.\ $z$ we let $\nu_z$ denote the $E_x$-ergodic 
components of $\mu_x^{\pins}$)
then  
\be\label{eq:E-A-equiv}
\mu_{ax}^\pins=\int a_*\nu_z\operatorname{d}\!\mu_x^\pins(z)
\ee
is the ergodic decomposition of $\mu_{ax}^\pins$ with respect to $E_{ax}.$

Applying Theorem~B in Section \ref{sec:ratner-sl2} 
we conclude that for $\mu_x^{\pins}$-a.e.\ $z$ the measure $\nu_z$ is described 
as follows. 

There exist
\begin{enumerate}
\item[(B-1)] $\sfield_z=(\field')^{q_z}\subset\field $ where $q_z=p^{n_z},$ $p={\rm char}(\field)$ and $n_z\geq 1$,
\item[(B-2)] a connected $\sfield_z$-subgroup ${\bf M}_z$ of $\rcal_{\field/\sfield_z}(\SL_d)$ so that ${\bf M}_z(\sfield_z)\cap \Gamma$ is Zariski dense in ${\bf M}_z,$ and
\item[(B-3)] an element $g_z\in G,$ 
\end{enumerate} 
such that $\nu_z$ is the $g_zL_zg_z^{-1}$-invariant probability Haar measure 
on the closed orbit $g_zL_z\Gamma/\Gamma$ with 
\[
L_z=\overline{M_z^+(\la_z)({\bf M}_z(\sfield_z)\cap\Gamma)} 
\]
where  
\begin{itemize}
\item the closure is with respect to the Hausdorff topology, and 
\item $\la_z:{\bf G}_m\to{\bf M}_z$ is a noncentral $\sfield_z$-homomorphism, $M_z^+(\la_z)$ is defined in~\eqref{eq:M-+-la}, and $E_x\subset M_z^{+}(\la_z).$  
\end{itemize} 
{}{Note that ${\bf M}_z$ in (B-2) is $\sfield_z$-smooth -- indeed ${\bf M}_z(\sfield_z)$ is Zariski dense in ${\bf M}_z$, see~\cite[Lemma 11.2.4(ii)]{Sp-AlgGrBook}.}

For any $z$ where $\nu_z$ is described as above, 
let $(\sfield_z,[{\bf M}_z],[M^+_{z}(\la_z)])$ be the corresponding triple
where $[\bullet]$ denotes the $\Gamma$ conjugacy class. 
This is well defined and we will refer to it as the {\it triple associated to} $z.$ 
Given a triple $(\sfield,[{\bf M}],[M^+(\la)])$ put
\[
\Sfrak(\sfield,[{\bf M}],[M^+(\la)])= \{z\in X: (\sfield,[{\bf M}],[M^+(\la)])\h\mbox{is associated to}\h z\}.
\]
Note that there are only countably many such triples. Indeed there are only countably many
closed subfields $\sfield\subset\field'$ as in Theorem~B(1),  
also there are only countably many ${\bf M}$ as in Theorem~B(2).
For any such $\sfield$ and ${\bf M}$  
there are only countably many choices of $M^+(\la)$ by Lemma~\ref{lem:w-pm-normal}(2). 
Therefore, there exists a triple $(\sfield,[{\bf M}],[M^+(\la)])$ such that 
\[
\mu(\Sfrak(\sfield,[{\bf M}],[M^+(\la)]))>0.
\] 

Note, however, that in view of~\eqref{eq:E-A-equiv}
we have $\Sfrak(\sfield,[{\bf M}],[M^+(\la)])$ is $A$-invariant. 
This, together with the fact that $\mu$ is $A$-ergodic, implies that 
\[
\mu(\Sfrak(\sfield,[{\bf M}],[M^+(\la)]))=1.
\] 
This finishes the proof of the lemma.
\end{proof}

We let $\sfield,$ ${\bf M}$ and
$L:=\overline{M^+(\la)({\bf M}(\sfield)\cap\Gamma)}$
be as in Lemma \ref{applyingAA}.
Define 
\be\label{eq:def-N0}
{\bf N}:= \text{the Zariski closure of ${N_{G'}\bigl({\bf M}(\sfield)\bigr)\cap\Gamma}$ in 
${\bf G}'$}
\ee 
where ${\bf G}':=\mathcal R_{\field/\sfield}(\SL_d)$ and $G':={\bf G}'(\sfield)=\SL(d,\field).$
Therefore, ${\bf N}$ {}{is a smooth group} defined over $\sfield,$ see e.g.~\cite[Lemma 11.2.4(ii)]{Sp-AlgGrBook}.
In view of (B-2) above we have 
\be\label{eq:N0-M0}
{\bf M}\subset{\bf N}^{\circ}\text{ and }{\bf N}\subset N_{{\bf G}'}({\bf M}).
\ee  
where ${\bf N}^{\circ}$ denotes the connected component of the identity in ${\bf N}.$

\begin{lem}\label{applyingAA2}
	We let $A_\sfield^{\rm sp}$ be the group of $\sfield$-points
	of the maximal $\sfield$-split torus subgroup of $\mathcal R_{\field/\sfield}A$. Then
	there exists some $g_0\in\SL(d,\field)$ so that $A_\sfield^{\rm sp}\subset g_0{\bf N}(\sfield)g_0^{-1}$
	and $\overline{Ag_0\Gamma/\Gamma}={\rm supp}(\mu).$ 
\end{lem}

\begin{proof}
Recall that $L\Gamma$ is a closed subset of $G$
and for $\mu$-a.e.\ $x$ and $\mu_x^\pins$-a.e.\ $z$ we have  
\be\label{eq:supp-nu}
{\rm supp}(\nu_z)=gL\Gamma/\Gamma
\ee
for some $g\in G.$ We note that the element $g$ is not well defined, however, the set $gL\Gamma$
is well defined. This, in view of (B-2), determines the set $g{\bf M}(\sfield)\Gamma$
as the smallest set of the form ${\bf R}(\sfield)\Gamma$ where $\bf R$ is an $\sfield$-subvariety
so that $\nu_z({\bf R}(\sfield)\Gamma/\Gamma)>0,$ see~\cite[Thm.~6.9]{MS-SL2} also the original~\cite[Prop.~3.2]{MarTom}.
Let now $g,g'\in G$ be such that $g{\bf M}(\sfield)\Gamma=g'{\bf M}(\sfield)\Gamma.$
Then  
\[
{\bf M}(\sfield)\subset\bigcup_{\gamma} g^{-1}g'{\bf M}(\sfield)\gamma.
\]
Hence, by Baire category theorem, there is some $\gamma_0$ so that
${\bf M}(\sfield)\cap g^{-1}g'{\bf M}(\sfield)\gamma_0$ is open in ${\bf M}(\sfield).$
Since ${\bf M}$ is Zariski connected, any open (in Hausdorff topology) subset of ${\bf M}(\sfield)$ is Zariski 
dense in ${\bf M},$~\cite[Ch.~1, Prop.\ 2.5.3]{Margulis-Book}. 
This and equality of the dimensions imply that ${\bf M}(\sfield)= g^{-1}g'{\bf M}(\sfield)\gamma_0.$
Therefore, $g^{-1}g'm_0\gamma_0=1$ for some $m_0\in{\bf M}(\sfield)$ and we get
\[
{\bf M}(\sfield)=g^{-1}g'{\bf M}(\sfield)\gamma_0=\gamma_0^{-1}{\bf M}(\sfield)\gamma_0.
\]
That is $\gamma_0\in N_G({\bf M}(\sfield))\cap\Gamma$ and 
\[
g^{-1}g'=\gamma_0^{-1}m_0^{-1}\in (N_G({\bf M}(\sfield))\cap\Gamma){\bf M}(\sfield).
\]
Hence, by~\eqref{eq:def-N0} and~\eqref{eq:N0-M0} we have
\be\label{eq:Borel-map}
g^{-1}g'\in(N_G({\bf M}(\sfield))\cap\Gamma){\bf M}(\sfield)\subset {\bf N}(\sfield).
\ee

Let $N={\bf N}(\sfield)$ and $G'={\bf G}'(\sfield)=\SL(d,\field).$
Then, by~\eqref{eq:Borel-map}, we get a Borel measurable map $f$ 
from $\Sfrak(\sfield,{\bf M},M^+(\la))$ to 
$G'/N=\SL(d,\field)/N$ defined by $f(x)=g_xN.$ 

The above discussion, in view of~\eqref{eq:E-A-equiv}, 
implies that $f$ is an $A$-equivariant Borel map, where the action of $A$ on $\SL(d,\field)/N$
is induced from the natural action of ${\mathcal R}_{\field/\sfield}({\bf A})$ on ${\bf G}'/{\bf N}.$  

Now by Lemma~\ref{lem:zd-measure1}, there exists some 
\[
g_0N\in{\rm Fix}_{A^{{\rm sp}}_{\sfield}}(\SL(d,\field)/N)
\] 
so that $f_*\mu$ is the $A$-invariant measure on the compact orbit $Ag_0N.$ 
Using the Birkhoff ergodic theorem for the action of $A$ on $X$ and the compactness of the orbit $Ag_0N,$ we can choose the representative $g_0\in\SL(d,k)$ so that $\overline{Ag_0\Gamma/\Gamma}={\rm supp}(\mu).$  
Let us recall that ${\rm Fix}_{A^{{\rm sp}}_{\sfield}}(\SL(d,\field)/N)=\{gN: g^{-1}A^{{\rm sp}}_{\sfield}g\subset N\}.$ In particular, $g_0$ satisfies 
\be\label{eq:A-N0}
g_0^{-1}A^{{\rm sp}}_{\sfield}g_0\subset N,
\ee
as we claimed.
\end{proof}

\subsection{The algebraic \texorpdfstring{$\gfield$-groups $\NM$ and $\bf H$}{K-groups F and H}}

While the groups $\bf M<\bf N$ are still somewhat mysterious at this stage, we can describe their $\field$-Zariski closure quite precisely.

Recall that $\Gamma\subset\SL(d,\field)$ is a lattice of inner type.
Hence, there exists a central simple algebra $B$ over $\gfield$ so that $\Gamma$
is commensurable with $\Lambda_B$, see~\S\ref{sec:innerforms}.
We define the shorthand $\Gamma_B:=\Gamma\cap\Lambda_B$.

\begin{lem}\label{lem:reductive}
With the notations as in Theorem~\ref{thm:measure-class}, let $\NM$ be a connected, noncommutative algebraic subgroup of $\SL_d$ so that $\NM (\field) \cap \Gamma$ is Zariski dense in $\NM$ and $A'= A   \cap g _ 0  \NM (\field) g _ 0^{-1}$ is cocompact in $A$ 
for some $g_0\in\SL(d,\field)$. 
Then $\NM$ is defined over $\gfield$ and we have the following.
\begin{enumerate}
\item $g_0^{-1}{\bf A}g_0 \subset {\NM}$,
\item ${\NM}$ has no $\tilde\gfield$-rational character for any purely
inseparable algebraic field extension
$\tilde\gfield$ of $\gfield,$
\item ${\NM}$ is a reductive $\gfield$-group,
\item ${\NM}(\field)\cap \Gamma$ is a lattice in ${\NM}(\field),$ and
\item the commutator group $[{\NM},{\NM}]$ is nontrivial, simply connected and
almost $\gfield$-simple.
\item Moreover, $[\NM,\NM](\field)\cong\prod_{i=1}^n\SL(d_0,\field)$ with~$d=nd_0$. In fact, apart from the order of the indices,  
the group~$g_0 [\NM,\NM]g_0^{-1}$ equals the subgroup consisting of~$n$ block matrices along the diagonal.
\end{enumerate}
\end{lem}

\begin{proof}  
Since $\Gamma_B$ is finite index in $\Gamma$
and $\NM$ is connected, we have $\NM(\field)\cap\Gamma_B$ is Zariski dense in $\NM$. 
This and the fact that $\Gamma_B\subset\Lambda_B$, imply that $\NM$ is defined over $\gfield,$~\cite[Lemma 11.2.4(ii)]{Sp-AlgGrBook}. 

Since $A/A'$ is compact and $\bf A$ is Zariski connected and $k$-split, 
we have that $A'$ is Zariski dense in ${\bf A}.$
Since also $g_0^{-1}A'g_0 \subset{\NM}(\field)$, we obtain $g_0^{-1}{\bf A}g_0 \subset{\NM}$ as $\field$-groups.

Let $\tilde\gfield$  
be a finite purely inseparable extension of $\gfield$.
For every place $w$ of $\field$ there exists a unique extension $\tilde w$ of $w$ to $\tilde K$.
{Recall that $k=K_v$ for a fixed place $v$ of $K$.}
Let 
\[
\tilde\Lambda_B=\{\gamma\in\SL_{1,B}(\tilde\gfield): \gamma\in\SL_{1,B}(\order_{\tilde w})\text{ for all $\tilde w\neq \tilde v$}\},
\]
see~\eqref{eq:Gamma-B}.

Let $\tilde{\mathcal O}$ be the ring of $\tilde v$-integers in $\tilde K$.
Suppose $\chi$ is an arbitrary $\tilde\gfield$-rational character of ${\NM}.$
Then there exists some $D\in\tilde\vintegers$, depending on $\chi$, so that
\[
\mathfrak B:=\chi(\Gamma_B\cap{{\NM}})\subset\chi(\tilde\Lambda_B\cap{{\NM}}(\tilde\gfield))\subset\tfrac{1}{D}\tilde\vintegers.
\]
In particular, there exists some $\ell_0\in\bbz$ so that for any place $\tilde w\neq\tilde v$ 
in $\tilde\gfield$ and any $r\in\mathfrak B$ we have $\tilde w(r)\geq \ell_0$.
Note, further, that $\mathfrak B$ is a multiplicative group, hence, 
$\tilde w(\mathfrak B)$ is a subgroup of $(\bbz,+)$. 
In consequence, we have $\tilde w(r)=0$ for any place $\tilde w\neq\tilde v$ in $\tilde\gfield$ and any $r\in\mathfrak B$. 
By the product formula we also get that
$\tilde v(r)=0$ for all $r\in\mathfrak B$. Therefore, $\mathfrak B$ is a finite group consisting of roots of unity. 
This implies that there is a finite index subgroup $\Gamma'\subset\Gamma_B\cap{{\NM}}$
so that $\chi(\Gamma')=1$. Since ${{\NM}}$ is connected and $\Gamma_B\cap{{\NM}}$
is Zariski dense in ${{\NM}}$, the group $\Gamma'$ is also Zariski dense in ${{\NM}}.$ 
This implies that $\chi$ is trivial on ${{\NM}}$ as claimed in (2).

We note that part (2) and~\cite[Thm.\ 1.3.6]{Conrad-Finiteness} imply part (4) directly. Below we give an argument using (3) which avoids the full force of~\cite[Thm.\ 1.3.6]{Conrad-Finiteness}. In particular, the classification of pseudo reductive groups in~\cite{CGP-PseudoRed} which is used to resolve the main difficulties in~\cite{Conrad-Finiteness} is not used in our proof of~(4). 

We now prove part (3). 
Let $\tilde\gfield$ be a finite purely inseparable extension of $\gfield$ so that $R_u({{\NM}})$
is defined and splits over $\tilde\gfield,$ see~\cite[18.4 and 15.5]{Borel-AlgGrBook}. 
 Restricting the adjoint representation of~${\NM}$ to the Lie algebra of~$R_u({{\NM}})$
and taking the determinant we obtain a~$\tilde\gfield$-character. {}{
{}{We claim that if $R_u({}{\NM})$ is nontrivial, this character is also nontrivial.  In view of this claim, (3) follows from~(2).

We now show the claim. Recall that~$R_u({}{\NM})$ is a $\tilde\gfield$-split unipotent subgroup.
Since $\SL_d$ is simply connected, we get from~\cite{Gille}, see also~\cite{BoTi-Unip}, 
that there exists a $\tilde\gfield$-parabolic subgroup $\bf P$ of 
$\SL_d$ so that $R_u({}{\NM})\subset R_u({\bf P})$ and $N_{\SL_d}(R_u({}{\NM}))\subset {\bf P}$. The claim now follows; 
indeed $g_0^{-1}{\bf A}g_0 \subset{{\NM}}\subset N_{\SL_d}(R_u({}{\NM}))$ and~$g_0^{-1}{\bf A}g_0 $} is a maximal torus 
which is $\field$-split and hence also $\tilde\gfield_{\tilde v}$-split.} 

Part (4) follows from (2), (3), and~\cite{Harder-Tamagawa}. Note that the absence of
a unipotent radical (defined over~$\field$ or not) makes the necessary arguments
in our case much simpler.

For the rest of the argument we fix a maximal $\gfield$-torus {}{$\bf T$} in
${{\NM}}$ which is $\field$-split, see~\cite[Cor.~A.2.6]{CGP-PseudoRed}. 
Note that by~\cite[Thm.~C.2.3]{CGP-PseudoRed},
there is some $g\in{{\NM}}(\field)$ so that 
\[
g{\bf T}g^{-1}=g_0^{-1}{\bf A}g_0 .
\]

We now establish part~(5).  
First note that ${{\NM}}$ is not commutative so $[\NM,\NM]$ is nontrivial. 
Let~$K'$ be a separable field extension of~$K$ such that~$\bf T$
splits over~$K'$. Therefore, there exists some~$g_1\in \SL_d(K')$ 
so that~$g_1{\bf T}g_1^{-1}$ is 
the full diagonal subgroup of~$\SL_d$.
Moreover, let ${\bf T}_0\subset{\bf T}$ be the central torus of~$\NM$.
Then  
\[
g_1[{{\NM}},{{\NM}}]g_1^{-1}
 \subset g_1[{\bf Z}_{\SL_d}({\bf T}_0),{\bf Z}_{\SL_d}({\bf T}_0)]g_1^{-1}=\prod_i\SL_{d_i}
\]
for some integers~$d_1,d_2,\ldots$ (that depend on the subgroup~$g_1{\bf T}_0g_1^{-1}$).
Since ${\bf T}\subset {{\NM}}$ has absolute rank $d-1$, the rank of $[{{\NM}},{{\NM}}]$ equals $d-1-\dim({\bf T}_0).$
Moreover, the torus ${\bf T}_0$ is central in ${\bf Z}_{\SL_d}({\bf T}_0)$, hence we have 
\[
d-1-\dim({\bf T}_0)\geq {\rm rank}([{\bf Z}_{\SL_d}({\bf T}_0),{\bf Z}_{\SL_d}({\bf T}_0)])=\sum_i (d_i-1).
\]
Together with the above inclusion, we thus get that $d-1-\dim({\bf T}_0)=\sum_i (d_i-1)$.
Since $[{{\NM}},{{\NM}}]$ is semisimple 
 and $\prod_i\SL_{d_i}$ has no proper semisimple
subgroup of the same rank, we get $g_1[{{\NM}},{{\NM}}]g_1^{-1}=\prod_i\SL_{d_i}.$
Let~$W_1,\ldots$ be the 
various irreducible subspaces for the action of~$[{{\NM}},{{\NM}}]$ on
the $d$-dimensional vector space 
that are defined over~$K'$ and correspond to the various blocks 
of~$g_1[{{\NM}},{{\NM}}]g_1^{-1}$. As~${\NM}$ is nonabelian at least
one of the subspace, say~$W_1$, has dimension $\geq 2$.
Let~${\bf W}$ be the sum of~$W_1$ and all its Galois images. 
Then $\bf W$ is invariant under~${\NM}$ and is defined over~$K$ -- recall that $K'/K$ is separable.
Since~${\NM}$ has no~$K$-rational characters, we see that~$\bf W$ 
has full dimension. {}{Otherwise the determinant of the restriction 
of~${\NM}$ to~$\bf W$ gives a $K$-character which is  nontrivial  since~$\bf T$ is a maximal torus -- 
indeed over $\gfield'$ we can conjugate $\bf T$ to $\bf A$ the diagonal subgroup, now any subspace of the standard $d$-dimensional representation of $\SL_d$ that is invariant under $\bf A$ and whose weights sum to zero is trivial.} 
This implies that~$[{{\NM}},{{\NM}}]$ is semisimple and almost $K$-simple. In particular, we obtain $d_i=d_j$ for all $i,j$
which gives part~(6).
\end{proof} 

Define 
\be\label{eq:def-NM}
\hat\NMM:=\text{the Zariski closure of ${\bf N}(\sfield)\cap\Gamma$ in $\SL_d$.}
\ee

In particular, $\hat\NMM$ is {}{a smooth group} defined over $\field,$ see~\cite[Lemma 11.2.4(ii)]{Sp-AlgGrBook}.
Put $\NM=\hat\NMM^\circ$, the connected component of the identity in $\hat\NM$. 

Since $[\Gamma:\Gamma_B]<\infty$, we have that $\NM$ coincides with the connected component
of the identity in $\hat\NM_B:=$ the Zariski closure of ${\bf N}(\sfield)\cap\Gamma_B$ in $\SL_d$.  
Now $\hat\NMM_B$ is {}{a smooth group} defined over $\gfield$, therefore,
$\NMM$ is also a smooth group defined over $\gfield$ and hence over $\field$.

\begin{lem}\label{lem:NM-is-NM}\leavevmode
\begin{enumerate}
\item ${\bf N}(\sfield)\subset\hat\NMM(\field)$ and hence ${\bf N}(\sfield)$ is Zariski dense in $\hat\NMM,$ 
\item $\NMM$ satisfies the conditions in Lemma~\ref{lem:reductive}.
\end{enumerate}
\end{lem}

\begin{proof}
For part (1) we note first that the definition~\eqref{eq:def-NM} implies that 
\[
{\bf N}(\sfield)\cap\Gamma\subset\hat\NMM(\field)=\mathcal R_{\field/\sfield}(\hat\NMM)(\sfield)\subset{\bf G}'(\sfield).
\]
Therefore, by~\eqref{eq:def-N0} we have ${\bf N}\subset\mathcal R_{\field/\sfield}(\hat\NMM).$
Taking $\sfield$-points we get part (1).
	
We now show that part (1) implies (2). To see this we first note that $\NMM$ is connected by definition. 
Next recall that by~\eqref{eq:N0-M0}
we have {}{$E_x\subset{\bf M}(\sfield)\subset{\bf N}(\sfield)$ for $\mu$-a.e.\ $x$.} In view of the definition of $E_x$, see~\eqref{eq:def-Ex}, and the fact that $\NMM$ is finite index in $\hat\NM$ we get that $\NMM$ is noncommutative. 
Moreover, note that $\NM$ is Zariski open and closed in $\hat\NMM$.  
By the definition of $\hat\NMM$ in \eqref{eq:def-NM} we have that $\hat\NMM(\field)\cap\Gamma$ is Zariski
dense in $\hat\NMM$. Together it follows that $\NM(\field)\cap\Gamma$ 
is Zariski dense in $\NMM$.
Finally by~\eqref{eq:A-N0} we have $g_0^{-1}A^{{\rm sp}}_{\sfield}g_0\subset {\bf N}(\sfield)\subset\hat\NMM(\field)$. 
Since $A^{{\rm sp}}_{\sfield}$ is cocompact in $A$, we obtain the last assumption
namely that $A   \cap g _ 0  \NM (\field) g _ 0^{-1}$ cocompact in $A$. 
\end{proof}

Put 
\be\label{eq:def-M-H}
\hat{\bf H}:=\text{the Zariski closure of ${\bf M}(\sfield)\cap\Gamma_B$ 
in $\SL_d$.}
\ee
Note that $\hat{\bf H}$ is {}{a smooth group} defined over $\gfield$ and hence over $\field.$
Put ${\bf H}:=\hat{\bf H}^\circ$, the connected component of the identity in $\hat{\bf H}$; then 
$\bf H$ is also a smooth group defined over $\gfield$ and hence over $\field$.

\begin{lem}\label{lem:H-is-H}\leavevmode
\begin{enumerate}
\item ${\bf M}(\sfield)\subset {\bf H}(\field),$ and hence ${\bf M}(\sfield)$ is Zariski dense in ${\bf H},$ 
\item $[\NMM,\NMM]={\bf H},$
\item ${\bf H}$ is almost $K$-simple.
\item ${\bf H}(\field)\cong\prod\SL(d_0,\field)$ where $d=nd_0.$
\end{enumerate}
\end{lem}

\begin{proof}
Recall from (B-2) that ${\bf M}\subset \rcal_{\field/\sfield}(\SL_d)$ is connected and that ${\bf M}(\sfield)\cap \Gamma$ is Zariski dense in ${\bf M}$. Since $\bf M$ is connected and $[\Gamma:\Gamma_B]<\infty$, we get that 
\be\label{eq:M-Gamma-GammaB}
\text{${\bf M}(\sfield)\cap \Gamma_B$ is Zariski dense in $\bf M$.}
\ee 
Therefore, as in the proof of Lemma~\ref{lem:NM-is-NM}(1), we have
\[
{\bf M}(\sfield)\cap\Gamma_B\subset\hat{\bf H}(\field)=\mathcal R_{\field/\sfield}(\hat{\bf H})(\sfield)\subset{\bf G}'(\sfield).
\]
This in view of our preceding discussion implies that ${\bf M}\subset \mathcal R_{\field/\sfield}(\hat{\bf H})$.
Since $\bf M$ is connected and $\mathcal R_{\field/\sfield}({\bf H})$ is a finite index subgroup\footnote{Indeed in view of the smoothness of $\H$ it follows from~\cite[Prop.~A.5.9]{CGP-PseudoRed} that $\mathcal R_{\field/\sfield}({\bf H})$ is connected.} of $\mathcal R_{\field/\sfield}(\hat{\bf H})$, we get that ${\bf M}\subset \mathcal R_{\field/\sfield}({\bf H})$. Taking $\sfield$-points part~(1) follows.

By~\eqref{eq:N0-M0} we have $g{\bf M}(\sfield)g^{-1}={\bf M}(\sfield)$ for all
$g\in {\bf N}(\sfield)$. Hence by part (1) and Lemma~\ref{lem:NM-is-NM}(1)
we obtain that ${\bf H}\subset\NM$ is a normal subgroup of $\hat\NMM$ and hence of $\NM$.
Moreover, since {}{$E_x\subset{\bf M}(\sfield)$ for $\mu$-a.e.\ $x$,} we have ${\bf H}$ is non-commutative. 
As was mentioned above, ${\bf H}$ is a $\gfield$-subgroup of $\NMM$. Hence Lemma~\ref{lem:NM-is-NM}(2)
and Lemma~\ref{lem:reductive}(5) imply 
\be\label{eq:comm-inc}
[\NMM,\NMM]\subset{\bf H}.
\ee
We now show the other inclusion. In view of Lemma~\ref{lem:NM-is-NM} and Lemma~\ref{lem:reductive} we have
$g_0{\bf R}(\field)\Gamma/\Gamma$ is a closed orbit with finite $g_0{\bf R}(\field)g_0^{-1}$-invariant measure for 
${\bf R}=\NMM,[\NMM,\NMM].$ 
Moreover, by the choice of $g_0$ in Lemma~\ref{applyingAA2} and Lemma \ref{lem:reductive}(1) we have 
\be\label{eq:mu-NM}
\mbox{$\mu$ is supported on $\overline{Ag_0\Gamma/\Gamma}\subset g_0{\NMM}(\field)\Gamma/\Gamma.$}
\ee 
Since $E_x\subset[\NMM,\NMM]$ and any $E_x$-ergodic component of $\mu$ is supported on 
a homogeneous space $g_0g[\NMM,\NMM](\field)\Gamma/\Gamma,$ for some $g\in\NMM(\field)$ we get that ${\bf M}(\sfield)\subset[\NMM,\NMM](\field).$
This completes the proof of part (2) thanks to part (1) and~\eqref{eq:comm-inc}.

The fact that ${\bf H}$ satisfies parts (3) and (4)  now follow from part (2), Lemma~\ref{lem:NM-is-NM}, and Lemma~\ref{lem:reductive}.
\end{proof}

Let us put $A_{\bf H}=A\cap g_0{\bf H}(\field)g_0^{-1}.$
In view of Lemmas~\ref{lem:NM-is-NM} and~\ref{lem:H-is-H} we see that $g_0{\bf H}(\field)g_0^{-1}$ has a block structure.
Put $C_{{\bf H}}=g_0{}{Z({\bf F}(\field))}g_0^{-1}.$ Then $A'':=A_{\bf H}C_{\bf H}$ is a cocompact 
subgroup of $A.$ We have the following.

\begin{lem}\label{lem:center-gone}
We can decompose the measure as follows 
\[
\mu=\int_{A/{\rm Stab}(\eta)}a_{*}\eta\operatorname{d}\!a
\]
where $\operatorname{d}\!a$ is the Haar measure on the compact group $A/{\rm Stab}(\eta),$ and $\eta$ is an $A_{\bf H}$-ergodic component of $\mu$ which is supported on
$g_0{\bf H}(\field)\Gamma/\Gamma$. 
Moreover we have
\[
\eta=\int\nu_z\operatorname{d}\!\eta(z). 
\]
\end{lem}

\begin{proof}
Recall from~\eqref{eq:mu-NM} that $\mu$ is supported on the closed orbit 
$g_0{\bf F}(\field)\Gamma/\Gamma.$
Hence, $C_{\bf H}\cap g_0\Gamma g_0^{-1}$ acts trivially on ${\rm supp}(\mu).$
Moreover, by Lemma~\ref{lem:NM-is-NM} and Lemma~\ref{lem:reductive}(4) we have that
$Z({\bf F}(\field))\Gamma/\Gamma$ is compact. This and the fact that $A/A''$ is compact implies that
\begin{equation}\label{AA-compact}
	A/A_{\bf H}(C_{\bf H}\cap g_0\Gamma g_0^{-1})
\end{equation}
is a compact group.
Therefore the $A_{\bf H}(C_{\bf H}\cap g_0\Gamma g_0^{-1})$-ergodic decomposition of $\mu$ 
can be written as 
\[
\int_{A/A_{\bf H}(C_{\bf H}\cap g_0\Gamma g_0^{-1})}a_{*}\eta\operatorname{d}\!a
\]
where $\eta$ is an $A_{\bf H}(C_{\bf H}\cap g_0\Gamma g_0^{-1})$-invariant measure on 
$g_0{\bf H}(\field)\Gamma/\Gamma.$ This implies the decomposition of~$\mu$
as in the lemma.

For the final claim we note that the above discussion also shows 
that~$\mathcal{B}^{A_{\bf H}}\subset\mathcal{P}$, 
where~$\mathcal{B}^{A_{\bf H}}$
is the~$\sigma$-algebra of~${A_{\bf H}}$-invariant sets.
Hence the conditional measures~$\mu_x^{\mathcal{P}}$
for the Pinsker~$\sigma$-algebra can be obtained by double conditioning, i.e.
\[
\mu_y^{\mathcal{P}}=\bigl(\mu_x^{\mathcal{B}^{A_{\bf H}}}\bigr)_y^{\mathcal{P}}
\]
for~$\mu$-a.e.~$x$ and~$\mu_x^{\mathcal{B}^{A_{\bf H}}}$-a.e.~$y$. 
Again because of compactness of~\eqref{AA-compact}
and the equivariance properties of
the conditional measures it suffices to consider one
of the conditional measure~$\eta=\mu_x^{\mathcal{B}^{A_{\bf H}}}$.
For the Pinsker conditional measure~$\eta_y^{\pins}$ we have considered
in~\eqref{eq:pinsmain} a decomposition into ergodic components
for the group~$E_y$. These ergodic components have been completely described
in Lemma~\ref{applyingAA}. The lemma follows by integration over~$\eta$.
\end{proof}

The following proposition describes the algebraic structure of the group $L$ in Lemma~\ref{applyingAA}. 
It turns out to be more convenient for us to explicate the structure of the finite index subgroup
\[
L_B:=\overline{M^+(\la)({\bf M}(\sfield)\cap\Gamma_B)}
\]
of $L$. Note that $L\Gamma/\Gamma=L_B\Gamma/\Gamma$.

\begin{prop}\label{prop:L}
Let $n$ be as in Lemma~\ref{lem:H-is-H}(4). Then there exist 
\begin{enumerate}
\item A collection $(\sfield_i:1\leq i\leq n)$ of closed (not necessarily distinct) subfields of $\field,$ 
\item For every $1\leq i\leq n$ a connected, simply connected, absolutely almost simple $\sfield_i$-group ${\bf H}_i$ and an isomorphism~$\varphi_i:{\bf H}_i\times_{\sfield_i}\field\rightarrow\SL_{d_0}$ (where~$\SL_{d_0}$ is considered as the~$i$th block subgroup corresponding to the indices~$(i-1)d_0+1,\ldots,id_0$)
\end{enumerate} 
so that $L_B=\prod_{i=1}^n\varphi_i({\bf H}_i(\sfield_i))\subset {\bf H}(\field)$ 
\end{prop}

\begin{proof}
In view of~\eqref{eq:M-Gamma-GammaB} and Lemma~\ref{lem:H-is-H}(3) and (4), the groups ${\bf M}$ and ${\bf H}$ satisfy the conditions in~\cite[\S7]{MS-SL2} for the lattice $\Gamma_B$. Therefore~\cite[Thm.\ 7.1]{MS-SL2}, which in turn relies heavily on~\cite{CGP-PseudoRed, Pink-Compact, LarPink},
implies the following. There exist 
\begin{enumerate}
\item[(a)] a collection $(\sfield_i:1\leq i\leq r)$ of closed subfields of $\field,$
\item[(b)] for every $1\leq j\leq n$ some $1\leq i(j)\leq r$ and a continuous field embedding $\tau_{j}:\sfield_{i(j)}\to\field,$ 
\item[(c)] for every $1\leq i\leq r$ a connected, simply connected, absolutely almost simple $\sfield_i$-group ${\bf H}_i$ (which is a form of~$\SL_{d_0}$),
\item[(d)] for every~$i\in\{1,\ldots,r\}$ there exists some~$j\in\{1,\ldots,n\}$
with~$i(j)=i$,
\item[(e)] an isomorphism $\varphi:\coprod_{i=1}^r{\bf H}_i\times_{\tau(\oplus_{i=1}^r\sfield_i)}\oplus_{i=1}^n\field\to \coprod\SL_{d_0},$ with $\tau=(\tau_{1},\ldots,\tau_n)$
\end{enumerate} 
so that $L_B=\varphi\Bigl(\prod_{i=1}^r{\bf H}_i(\sfield_i)\Bigr)\subset {\bf H}(\field)$.

We now claim
\be\label{eq:r=n}
r=n.
\ee

Assuming~\eqref{eq:r=n}, and after possibly renumbering 
and replacing $\sfield_i$ by $\tau_i(\sfield_i)$ for $1\leq i\leq r=n,$ we get the proposition.

We now turn to the proof of~\eqref{eq:r=n}.
Put $\Delta:={\bf H}(\field)\cap\Gamma$ and recall the notation $A_{{\bf H}}=A\cap g_0{\bf H}(\field)g_0^{-1}.$
In view of Lemma~\ref{lem:center-gone} we can reduce the study of the measure~$\mu$ to the study 
of the measure $\eta$, which is an $A_{\bf H}$-ergodic invariant measure on $g_0{\bf H}(\field)/\Delta.$ 
Put 
\[
\RNM:=\mathcal R_{\oplus_{j=1}^n\field/\tau\bigl(\oplus_{i=1}^r\sfield_{i}\bigr)}\biggl(\coprod_{j=1}^n\SL_{d_0}\biggr).
\]
Then $\RNM$ is a smooth $\oplus_{i=1}^r\sfield_i$-group and $\RNM(\oplus_{i=1}^r\sfield_i)={\bf H}(\field),$ see~\cite[Prop.~{A.5.2}]{CGP-PseudoRed}.
Moreover, $L_B=\varphi(\prod_{i=1}^r{\bf H}_{i}(\sfield_i))$ is the group
of $\oplus_{i=1}^r\sfield_i$-points of a $\oplus_{i=1}^r\sfield_i$-subgroup of $\RNM,$ see~\cite[Prop.~{A.5.7}]{CGP-PseudoRed}.
Define
\be\label{eq:def-norm-L}
{\bf R}:= \text{the Zariski closure of ${N_{{\bf H}(\field)}\bigl(L_B)\cap\Delta}$ in 
$\RNM$.}
\ee 
Put 
\[
R={\bf R}(\oplus_{i=1}^r\sfield_i)\subset{\bf H}(\field).
\]
Then $R\subset N_{{\bf H}(\field)}(L_B).$ 

In view of~\eqref{eq:supp-nu} and Lemma~\ref{lem:center-gone} we have the following. 
For $\eta$-a.e.\ $x\in{\bf H}(\field)/\Delta$ and $\eta_x^\pins$-a.e.\ $z$ we have  
\[
{\rm supp}(\nu_z)=g_0gL\Delta/\Delta=g_0gL_B\Delta/\Delta
\]
for some $g\in{\bf H}(\field).$

Therefore, arguing for each $i$ separately, as in the proof of Lemma~\ref{applyingAA2}
we get the following. There is a cocompact subgroup $A_{\bf H}'\subset A_{\bf H}$
and some $g_1\in{\bf H}(\field)$ so that 
\[
g_1^{-1}g_0^{-1}A'_{\bf H}g_0g_1\subset R,
\]
moreover, $\overline{A_{\bf H}g_0g_1\Gamma}={\rm supp}(\eta).$ 

In particular we have $A'_{\bf H}$ normalizes the group $g_0g_1L_Bg_1^{-1}g_0^{-1}.$
Recall now that $A_{\bf H}$ is a maximal torus in the block diagonal group $g_0{\bf H}(\field)g_0^{-1}.$
These and the fact that $A'_{\bf H}$ is cocompact in $A_{\bf H}$ imply that the block structure 
of $L_B$ and ${\bf H}$ agree with each other, i.e.\ $r=n.$ To see this, assume $i(j)=1$ for $j=1,2.$
Let $a$ be an element in $A'_{\bf H}$ which equals the identity in all the blocks $j=2,\ldots,n$
and in the first block it is a diagonal elements which generates an unbounded group.  
Then since $a$ normalizes $g_0g_1L_Bg_1^{-1}g_0^{-1}$ we get a contradiction.
\end{proof}

\begin{cor}\label{cor:norm-L}
$N_{{\bf H}(\field)}(L_B)/{}{Z({\bf H}(\field))}L_B$ is a torsion abelian group.
\end{cor}

\begin{proof}
In view of Proposition~\ref{prop:L} it suffices to argue in each $\SL_{d_0}$-block separately.
Hence we fix some $i\in\{1,\ldots,n\}.$ First note that ${\bf H}_i$ is an $\sfield_i$-form of $\SL_{d_0}$.
Suppose now that $g\in\SL(d_0,\field)$ normalises ${\bf H}_i(\sfield_i)$. 
Since ${\bf H}_{i}(\sfield_i)$ is Zariski dense in the $\sfield_i$-group ${\bf H}_{i}$, 
see e.g.~\cite[Ch.~1, Prop.\ 2.5.3]{Margulis-Book}, we thus get that $g$ induces an $\sfield_i$-automorphism of ${\bf H}_{i}.$ 
Extending the scalars from~$\sfield_i$ to~$\field$ we see that the automorphism is inner,
i.e.~this automorphism $\sigma_i(g)$
belongs to~${\bf H}_i^{\operatorname{ad}}(\field)$.
Together it follows that~$\sigma_i(g)\in {\bf H}_i^{\operatorname{ad}}(\sfield_i)$.
This automorphism is, moreover, nontrivial if and only if $g$ is not central in $\SL_{d_0}.$
Hence, we get a monomorphism $g\mapsto\sigma(g)$ from
\[
N_{\SL(d_0,\field)}(L_B\cap \SL(d_0,\field))/{}{Z(\SL(d_0,\field))}
\]
into ${\bf H}^{\rm ad}_{i}(\sfield_i).$ This map sends ${\bf H}_i(\sfield_i)$ to $[{\bf H}^{\rm ad}_{i}(\sfield_i),{\bf H}^{\rm ad}_{i}(\sfield_i)]$ by~\cite[Ch.~1, Thm.\ 2.3.1]{Margulis-Book} and the claims hold true by~\cite[Ch.~1, Thm.\ 2.3.1]{Margulis-Book}.
\end{proof}

Let us now complete the proof of Theorem~\ref{thm:measure-class}.
\begin{proof}[Proof of Theorem~\ref{thm:measure-class}]
In view of Lemma~\ref{lem:center-gone} we may and will restrict our attention to the measure $\eta$
appearing in the statement of that lemma. 
Similar to the proof of~\eqref{eq:r=n}, put $\Delta:={\bf H}(\field)\cap\Gamma.$
Define
\[
\RNM:=\mathcal R_{\oplus_{j=1}^n\field/\oplus_{i=1}^n\sfield_{i}}\Bigl(\coprod_{i=1}^n\SL_{d_0}\Bigr).
\]
Then $\RNM$ is a smooth $\oplus_{i=1}^n\sfield_i$-group and $\RNM(\oplus_{i=1}^n\sfield_i)={\bf H}(\field),$ see~\cite[Prop.~{A.5.2}]{CGP-PseudoRed}.
Moreover, $L_B=\prod_{i=1}^n{\bf H}_{i}(\sfield_i)$ is the group
of $\oplus_{i=1}^n\sfield_i$-points of a $\oplus_{i=1}^n\sfield_i$-subgroup of $\RNM,$ see~\cite[Prop.~{A.5.7}]{CGP-PseudoRed}.
Since $\RNM(\oplus_{i=1}^r\sfield_i)={\bf H}(\field)$, we may view $Z({\bf H}(\field))$ as a finite subgroup of $\RNM(\oplus_{i=1}^r\sfield_i)$. Define
\[
{\bf R}:= {}{Z({\bf H}(\field))}\Bigl(\text{the Zariski closure of ${N_{{\bf H}(\field)}\bigl(L_B)\cap \Delta}$ in 
$\RNM$}\Bigr).
\] 
Put 
$
R={\bf R}(\oplus_{i=1}^n\sfield_i)\subset{\bf H}(\field),
$
since ${\bf H}(\field)={\bf H}'(\oplus(\sfield_i))$ we have ${}{Z({\bf H}(\field))}\subset R.$
Moreover, $R\subset N_{{\bf H}(\field)}(L_B),$ and by Corollary~\ref{cor:norm-L} we have
\be\label{eq:L-comm-R}
[R,R]\subset {}{Z({\bf H}(\field))}L_B.
\ee

In view of~\eqref{eq:supp-nu} and Lemma~\ref{lem:center-gone}, for $\eta$-a.e.\ $x\in g_0{\bf H}(\field)/\Delta$ we have  
\be\label{eq:supp-nu-1}
{\rm supp}(\nu_x)=g_0gL\Delta/\Delta=g_0gL_B\Delta/\Delta
\ee
for some $g\in{\bf H}(\field)$ depending on $x$.

Therefore, arguing as in the proof of Lemma~\ref{applyingAA2}
we get the following. There is a cocompact subgroup $A'_{\bf H}\subset A_{\bf H}$
containing ${Z({\bf H}(\field))}$
and there is some $g_1\in{\bf H}(\field)$ so that $g_1^{-1}g_0^{-1}A_{\bf H}'g_0g_1\subset R$.
We may furthermore require that $\overline{A_{\bf H}g_0g_1\Gamma/\Gamma}={\rm supp}(\eta).$ 
This gives the following decomposition. 
\be\label{eq:mu-decom-2}
\eta=\int_{A_{\bf H}/A'_{\bf H}}a_*\eta' \operatorname{d}\!a,
\ee
where 
\begin{itemize}
\item $\operatorname{d}\!a$ is the Haar measure on the compact group $A_{\bf H}/A'_{\bf H},$
\item $\eta'$ is an $A'_{\bf H}$-invariant and ergodic probability measure on $g'_0R/\Delta'$ where $\Delta':=R\cap\Delta$ and $g_0'=g_0g_1$. 
\end{itemize}

\noindent
Note that we have implicitly identified here $g'_0R/\Delta'$ with $g'_0R\Delta/\Delta$ (which in turn itself has already been implicitly identified with $g'_0R\Gamma/\Gamma$).

We now further investigate the measure $\eta'.$  
In view of~\eqref{eq:supp-nu-1} we can write  
\be\label{eq:eta-pins}
\eta'=\int_{g'_0R/\Delta'}\nu_{x}\operatorname{d}\!\eta'(x),
\ee 
where $\nu_x$ is the $g'_0 gL_Bg^{-1}  {g'_0}^{-1}$-invariant measure on $ g'_0 g L_B\Delta'/\Delta'$ where we write $x$ as $x=g'_0 g\Delta'/\Delta'$ for $g \in R$.

Since $L_B$ is normal in $R$, we get that $\eta'$ is $g_0'L_B{g_0'}^{-1}$-invariant.
Moreover, since ${}{Z({\bf H}(\field))}\subset A'_{\bf H},$ we also have $\eta'$ is ${}{Z({\bf H}(\field))}$-invariant.
Finally since $L_B\Delta/\Delta$ is closed in ${\bf H}(\field)/\Delta$ we have that ${}{Z({\bf H}(\field))}L_B\Delta'$ is a closed subgroup of $R$. Let $L'_B=Z({\bf H}(\field))L_B$.
We define $\eta'_1$ as the push forward of $\eta'$ under the canonical quotient map
from $g_0'R/\Delta'$ into $g_0'R/L_B\Delta'$, and similarly $\eta'_2$
as the push forward to $g_0'R/L_B'\Delta'$. 
With this we obtain from~\eqref{eq:eta-pins} that for $\nu_{L_B} = \nu_{L_B\Delta'/\Delta'}$
\begin{align}\label{eq:mu-decom-3}
\nonumber\eta'&=\int_{g_0'R/L_B\Delta'} g_*\nu_{L_B}\operatorname{d}\!\eta'_1(gL_B\Delta')\\
\nonumber&=\int_{g_0'R/L_B'\Delta'} g_*\left(\int_{Z({\bf H}(\field))}h_*\nu_{L_B}\operatorname{d}\!h\right)\operatorname{d}\!\eta'_2(gL_B'\Delta')\\
&=(g_0')_*\int_{R/L_B'\Delta'} g_*\left(\int_{Z({\bf H}(\field))}h_*\nu_{L_B}\operatorname{d}\!h\right)\operatorname{d}\!\eta_P(gL_B'\Delta'),
\end{align}
for a~$(g_0')^{-1}A'_{\bf H}g_0'$-invariant and ergodic probability probability measure $\eta_P$ on $P=R/L_B'\Delta'$.
We note that the measure defined by the inner integral in \eqref{eq:mu-decom-3} is actually homogeneous.
Furthermore, by Corollary~\ref{cor:norm-L} we know that $P=R/L_B'\Delta'$
is a torsion abelian group.

We claim that  
\be\label{eq:image-of-AH}
\text{the image of $(g_0')^{-1}A'_{\bf H}g_0'$ in $P$ is compact and in particular closed.}
\ee 
Assuming~\eqref{eq:image-of-AH}, let us finish the proof. 
Indeed,~\eqref{eq:image-of-AH} implies that $\eta_P$ equals the Haar measure
on a coset of
\[
\biggl((g_0')^{-1}A'_{\bf H}g_0'\biggr)L_B'\Delta'/L_B'\Delta'
\]
This together with~\eqref{eq:mu-decom-2}
finish the proof.

\medskip

We now prove~\eqref{eq:image-of-AH}. Let $\{s_1,\ldots,s_r\}\subset (g_0')^{-1}A'_{\bf H}g_0'$ be a subset which generates a cocompact subgroup of $(g_0')^{-1}A'_{\bf H}g_0'$. By Corollary~\ref{cor:norm-L}, there exists some 
$m\in\bbn$ so that $s_i^m\in Z({\bf H}(\field))L_B=L_B'$ for all $1\leq i\leq r$.
Let $D$ be the group generated by $\{ s_1^m,\ldots, s_r^m\}$. Then $D$ is cocompact in $(g_0')^{-1}A'_{\bf H}g_0'$ and the natural orbit map from $(g_0')^{-1}A'_{\bf H}g_0'$ to $P$ factors through the natural map from $(g_0')^{-1}A'_{\bf H}g_0'/D$ to $P$. These maps are continuous 
and $(g_0')^{-1}A'_{\bf H}g_0'/D$ is compact, thus~\eqref{eq:image-of-AH} follows.
\end{proof}

\section{Joining classification}\label{sec:joining}

\subsection{On the group generated by certain commutators}
A key to the classification of joinings is the following simple general fact about a rank two $k$-torus. 
Let $\Gbf$ denotes a connected, simply connected,
absolutely almost simple group defined over a local field $\field$ with ${\rm char}(\field)>3.$
Let $\la:{\bf G}_m^2\to\Gbf$ be an
algebraic monomorphisms  defined over $\field;$ 
let $\mathbf A=\la(\Gbf_m^2)$.
Fix a maximal $\field$-split, $\field$-torus
${\bf S}\subset{\bf G}$ so that ${\bf A} \subset{\bf S}$. 
Further, let ${\bf T}\supset{\bf S}$ be a maximal torus of $\Gbf$ which is defined over $\field.$
Put $\Phi:=\Phi(\mathbf T,\G),$ ${}_\field\Phi:={}_\field\Phi(\mathbf S,\G),$ and $\APhi:={}_\field\Phi(\mathbf A,\G)$. 
For $\Psi \subset \APhi$ set
\be\label{eq:vartheta-psi}
\vartheta(\Psi):=\{\alpha\in\Phi({\bf T},\Gbf): \alpha|_{{\bf A}}\in\Psi\}.
\ee

\begin{prop}\label{prop:commutators} The group $\mathbf G $ is generated by the commutators $[\mathbf V _ {[\alpha]}, \mathbf V _ {[\beta]}]$ where $\alpha, \beta$ run over all \emph{linearly independent} pairs in $\overline \Phi $.
\end{prop}

We need the following lemma from~\cite[Lemma 4.2]{EL-Joining}, see also~\cite[Lemma 9.6]{EK-2}.

\begin{lem}\label{lem:LI-NP}
Let $\delta\in\APhi$ and $\delta'\in\vartheta([\delta]).$ 
Then there exist some $\beta\in{\APhi}$ and some $\beta'\in\vartheta([\beta])$ with the following properties.
\begin{enumerate}
\item $\{\beta,\delta\}$ is a linearly independent subset of $\APhi$.
\item $\delta'-\beta'\in\Phi$.
\end{enumerate}
\end{lem}

\begin{proof}
Let $\bar\field$ be the algebraic closure of $\field$.
Let
\[
\Upsilon=\{\alpha\in\mathbb{R}\otimes X^*({\bf T}): \alpha|_{{\bf A}}\in\mathbb{R}\delta\},
\]
where $X^*({\bf T})$ denotes the group of characters of ${\bf T}$.  

Let $\mathfrak g'$ be the $\bar\field$-span of $\{\gfrak_{\alpha'}, [\gfrak_{\alpha'},\gfrak_{\beta'}]:\alpha',\beta'\in\Phi\setminus\Upsilon\}$.
It follows easily from the Jacobi identity (cf.\ the proof of~\cite[Lemma 4.2]{EL-Joining} for details) that $\mathfrak g'$
is an ideal of $\gfrak$. 
Recall that $\mathbf A=\la(\Gbf_m^2)$. Therefore, 
$\APhi$ has at least two linearly independent roots, and $\gfrak'$ is not central. 
Since $\mathfrak g$ has no proper non-central ideals, 
we have $\gfrak'=\gfrak$. 

In particular, we get that
\[
\gfrak_{\delta'}\subset\sum_{\alpha'\in\Phi_1\setminus\Upsilon} \gfrak_{\alpha'}+\sum_{\alpha',\beta'\in\Phi\setminus\Upsilon}[\gfrak_{\alpha'},\gfrak_{\beta'}].
\]
Since $\delta'\in\Upsilon$, the above implies that $\gfrak_{\delta'}\subset\sum_{\alpha',\beta'\in\Phi\setminus\Upsilon}[\gfrak_{\alpha'},\gfrak_{\beta'}].$ 
But for every $\alpha ', \beta '$, we have that $[\mathfrak g _ {\alpha '}, \mathfrak g _ {\beta '}] \subseteq \mathfrak g _ {\alpha ' + \beta '}$ hence $\delta'=\alpha'+\beta'$ for some $\alpha',\beta'\in\Phi\setminus\Upsilon$. In particular, since $\beta'\not\in\Upsilon$, it holds that 
$\beta := \beta'|_{{\mathbf A}}$ is linearly independent from $\delta$.
\end{proof}

\begin{proof} [Proof of Proposition~\ref{prop:commutators}]
Since the statement of the proposition is on the level of algebraic groups, the validity of the statement over the algebraic closure $\bar\field$ 
of $k$ implies that of the statement when the groups are considered as algebraic groups over $k$. 
Over $\bar\field$, we can write for every $\alpha \in \overline \Phi $
\begin{equation*}
\mathbf V _ {[\alpha]} = \prod_  {\delta' \in \vartheta ([\alpha])} \mathbf U _ {\delta'}
\end{equation*}
with each $\mathbf U _ {\delta'}$ a one parameter unipotent group over $\bar\field$.

Since the group $\mathbf G $ is absolutely almost simple, in particular semisimple, the root groups $\mathbf U _ {\delta'}$ for 
${\delta'} \in \Phi$ generate. Therefore to prove the proposition it is enough to show that for every $\delta' \in \Phi$, one can find $\alpha$ and $\beta$ in $\overline \Phi$, linearly independent, so that 
\begin{equation}\label{eq:delta is good}
\mathbf U _ {\delta'} \subset [\mathbf V _ {[\alpha]}, \mathbf V _ {[\beta]}].
\end{equation}
Let $\beta, \beta '$ be as in Lemma~\ref{lem:LI-NP} applied to $\delta:= \delta '|_{\mathbf A _ 1}$ and $\delta '$, and let $\alpha'= \delta'-\beta'$ and
\(\alpha = \alpha ' |_{\mathbf A }\). In particular, $\alpha$ and $\beta$ are linearly independent.

Recall that ${\rm char}(\field)\neq 2,3$, hence by~\cite[\S4.3]{BoTi-Abstract}
irregular commutation relations do not occur. This means in particular that
\[
[\mathbf U_{\alpha'},\mathbf U_{\beta'}]=\mathbf U_{\alpha'+\beta'}\]
(cf.\ also~\cite[\S2.5]{BoTi-RedGr}).
But $\mathbf U _ {\alpha '} \subset\mathbf V_{[\alpha] }$, $\mathbf U _ {\beta '} \subset\mathbf V_{[\beta] }$, and by definition $\alpha ' + \beta ' = \delta'$. Equation \eqref{eq:delta is good} and hence the proposition follows.

\end{proof}

\subsection{The main entropy inequality and the invariance group of the leafwise measures}\label{sec:entropy-inequal}

From now on, we use the notation from Theorem~\ref{thm:joining}. 
In particular,  for $i=1,2$, $\Gbf_i$ denotes a connected, simply connected,
absolutely almost simple group defined over $\field$. We put $G_i=\Gbf_i(\field)$ and $G=G_1\times G_2$.
Recall also that ${\rm char}(\field)>3.$

Suppose fixed two algebraic monomorphisms $\la_i:{\bf G}_m^2\to\Gbf_i$ defined over $\field;$ 
let $\mathbf A_i=\la_i(\Gbf_m^2)$ and $A_i=\mathbf A_i(\field).$ 
For $i=1,2$ we fix a maximal $\field$-split, $\field$-torus
${\bf S}_i\subset{\bf G}_i$ so that ${\bf A}_i\subset{\bf S}_i$, and set  
 ${}_\field\Phi_i:={}_\field\Phi(\mathbf S_i,\G_i),$ and $\APhi_i:={}_\field\Phi(\mathbf A_i,\G_i)$. 
 
{}{Define $\bf A$ to be the smooth $\field$-group so that
\[
{\bf A}(R):=\{(\la_1(\fele),\la_2(\fele)):\fele\in\Gbf_m(R)^2\}
\]
for any algebra $R/\field$; let 
$
A:=\mathbf A(\field).
$}

Let
\begin{equation*}
\APhi={}_\field \Phi(\mathbf A,\G_1 \times\G_2).
\end{equation*}
Using the natural homomorphisms from $\mathbf A$ to $\mathbf A_i$, for $i=1,2$ we can view ${}_\field \Phi(\mathbf A_i,\G_i)$ as subsets of $\APhi$ and moreover we have that 
\begin{equation*}
\APhi={}_\field \Phi(\mathbf A_1,\G_1) \cup {}_\field \Phi(\mathbf A_2,\G_2).
\end{equation*}
For each $\alpha \in \APhi$, we can write the coarse Lyapunov group $V_{[\alpha]} \subset G _ 1 \times G _ 2$ as a product $V ^ 1_{[\alpha]} \times V ^ 2_{[\alpha]}$ with $V ^ i _ {[\alpha]} \subset G _ i$; by convention if $\alpha \not\in \APhi$ then $V ^ i _ {[\alpha]} = \left\{ 1 \right\}$.
For $i=1,2$ we fix a maximal, compact, open subgroup $\mathfrak G_i\subset G_i$ 
and put $\mathfrak G:=\mathfrak G_1\times\mathfrak G_2.$

Recall that $\mu$ denotes an ergodic joining for the action of $A_i$ on $(X_i,m_i)$
for $i=1,2.$

\begin{propc*}[Cf.~\cite{EL-Joining}, \S3]\label{prop:el-inequality}
Let $a=(a_1,a_2)\in A$ and let $\Psi\subset\APhi$ be a positively closed subset. Put
\[
W=V_\Psi\subset W_{G_1\times G_2}^-(a).
\] 
Then $W=W_1\times W_2$ where $W_i\subset G_i$ for $i=1,2$ and
\be\label{eq:el-inequality}
\entropy_\mu(a,W)\leq\entropy_{m_1}(a_1,W_1)+\entropy_{\mu}(a,\{\identity\}\times W_2).
\ee
Furthermore, the following hold.
\begin{enumerate}
\item If the equality holds in~\eqref{eq:el-inequality}, then 
$W_1$ is the smallest algebraic subgroup of 
$W_1$ which contains $\pi_1\Bigl(\supp(\mu_x^W)\cap\mathfrak G\Bigr)$. 
\item The equality holds for $W=W^-_{G_1\times G_2}(a)$.
\item For every $\alpha\in\APhi,$ the equality holds for $W=V_{[\alpha]}$. 
\end{enumerate}
\end{propc*}

\begin{proof}
The main inequality follows\footnote{The arguments in~\cite{EL-Joining} generalize to the setting at hand without a change.} from~\cite[Prop.\ 3.1]{EL-Joining}. 

Parts (2) and~(3) follow from~\cite[Prop.\ 3.3 and Cor.\ 3.4]{EL-Joining}.

To see part~(1), first note that by~\cite[Prop.\ 6.2]{EL-GenLow} we have 
\[
\pi_1\biggl(\overline{\supp(\mu_x^W)\cap\mathfrak G}^z\biggr)
\] 
is a (Zariski closed) subgroup which is normalized by $a$ and contains $\pi_1\Bigl(\supp(\mu_x^W)\Bigr).$
Part~(1) now follows from~\cite[Prop.\ 3.2]{EL-Joining}. 
\end{proof}

\begin{cor}\label{cor:zariski dense support} For any $\alpha \in \APhi$, we have that $\pi_i\Bigl(\lsupp_x^{[\alpha]} \cap \mathfrak G\Bigr)$ is Zariski dense in $\pi_i(V _ {[\alpha]})$ for $i=1,2$.
\end{cor}

\begin{proof}
In view of Proposition~\ref{prop:el-inequality}.(3), this is a direct consequence of Proposition~\ref{prop:el-inequality}.(1) and the definition of $ \lsupp_x^{[\alpha]}$.
\end{proof}

Fix an element $a=(a_1,a_2)\in A$ that is {\em regular} with respect to $\APhi$,
that is: $\alpha(a)\neq1$ for any $\alpha\in\APhi.$
We denote the Pinsker $\sigma$-algebra, $\pins_{a}$, simply by $\pins.$ 

Disintegrate $\mu$ as follows.
\be\label{eq:joining-pinsker-cond}
\mu=\int_X\mu_x^{\pins}\operatorname{d}\!\mu(x),
\ee
where $\mu_x^{\pins}$ denotes the $\pins$-conditional measure for 
$\mu$-almost every $x\in X.$

Similar to~\eqref{eq:def-Hcal}, define 
\[
\Hcal_x:=\{g\in G_1\times G_2: g\mu_x^\pins=\mu_x^\pins\}.
\]
We have $a\Hcal_xa^{-1}=\Hcal_{ax}$ 
for all $a\in A$ and $\mu$-a.e.\ $x$, see~\eqref{eq:Hcal-A-equiv}.

\begin{lem}\label{lem:H-x is big}
For $\mu$-a.e. $x$, and any linearly independent $\alpha, \beta \in \Phi$ the measure $\mu ^ {\mathcal{P}} _ x$ is a.s. invariant under $\Bigl[\lsupp_x^{[\alpha]},\lsupp_x^{[\beta]}\Bigr]$, i.e. $\Bigl[\lsupp_x^{[\alpha]},\lsupp_x^{[\beta]}\Bigr] \subset \Hcal_{x}$.
\end{lem}

\begin{proof}
By Lemma~\ref{lem:pinsker-measurable}, for every $\alpha \in \overline \Phi$ and $\mu$-a.e.\ $x$, we have that $\mu ^ {\mathcal{P}} _ x$ is invariant under $I ^ {[\alpha]} _ x$, and hence by Lemma~\ref{lem:Inv-prod-struc} is invariant under $I ^ {\Psi} _ x$ for any positively closed $\Psi \subset \overline \Phi$. By Lemma~\ref{lem:commutator} we have therefore that for any linearly independent $\alpha, \beta \in \Phi$ the measure $\mu ^ {\mathcal{P}} _ x$ is a.s. invariant under $\Bigl[\lsupp_x^{[\alpha]},\lsupp_x^{[\beta]}\Bigr]$.
\end{proof}

Recall that $\mathfrak G=\mathfrak G_1\times\mathfrak G_2$ is a compact, open subgroup of 
$G=G_1\times G_2$.
Define
\[
\Qcal_x:=\Bigl\langle \{g\in \Hcal_x\cap\mathfrak G:\text{ $g$ is unipotent}\}\Bigr\rangle.
\]

\begin{cor}\label{cor:joining-pinsker-cond-inv1}\label{cor:proj-Zd}
For $\mu$-a.e.\ $x$,
$\pi_i(\Qcal_x)$ is Zariski dense in $\G_i$ and $\pi_i(\Hcal_x)$ is unbounded for $i=1,2$. 
\end{cor}

\begin{proof}
For any $x$, let ${\bf L}_{i,x}$ denote the Zariski closure of $\pi_i(\Qcal_x)$ in $\G_i.$
Let $\alpha, \beta \in\APhi$ be two linearly independent roots.
By Corollary~\ref{cor:zariski dense support}, a.s.\ $\pi_i\Bigl(\lsupp_x^{[\alpha]}\cap \mathfrak G\Bigr)$ is Zariski dense in $\pi_i(\mathbf V_{[\alpha]})$ and similarly for $\beta$, for $i=1,2$. By Lemma~\ref{lem:H-x is big}, $[\lsupp_x^{[\alpha]}\cap \mathfrak G,\lsupp_x^{[\beta]}\cap \mathfrak G] \subset \Qcal_x$. It follows that
\begin{equation*}
\pi_i\Bigl([\mathbf V_{[\alpha]},\mathbf V_{[\beta]}]\Bigr)\subset  {\bf L}_{i,x}
\end{equation*}
for any two linearly independent $\alpha,\beta \in\APhi.$
The first part of the claim follows using Proposition~\ref{prop:commutators}. 

For the second, by Lemma~\ref{lem:H-x is big} and Lemma~\ref{lem:Inv-prod-struc} there is an $\alpha \in \APhi$ such that $\linv_x^{[\alpha] }$ is non trivial. If $\linv_x^{[\alpha] }$ would be bounded on a set of positive measure its diameter would be a monotone increasing measurable function under an appropriate subsemigroup of $A$, in contradiction to Poincare recurrence.
\end{proof}

\subsection{Proof of Theorem~\ref{thm:joining}}\label{sec:joining-proof}
Let $X'\subset X$ be a conull subset so that the conclusions of Lemma~\ref{lem:pinsker-comp-joinning} and
Corollary~\ref{cor:proj-Zd} hold true on $X'$. 

By Corollary~\ref{cor:proj-Zd}~, for all $x\in X'$ the group $\Qcal_x$ 
satisfies the conditions in Theorem~A.2 in Section~\ref{sec:hilbert-LarPin1}. 
Therefore, there are two possibilities to consider.

\medskip
\noindent
{\bf Case 1:} There is a subset $X''\subset X'$ with $\mu(X'')>0$ 
so that for all $x\in X''$ and $i=1,2,$ the following holds. There are
\begin{itemize}
\item subfields $\field_{i,x}\subset \field,$ 
\item $\field_{i,x}$-groups ${\bf H}_{i,x},$  
\item $\field$-isomorphism $\varphi_{i,x}:{\bf H}_{i,x}\times_{\field_{i,x}}\field\to\G_i,$ and  
\item open, compact subgroups $\Qcal_{i,x}\subset \varphi_{i,x}\Bigl({\bf H}_{i,x}(\field_{i,x})\Bigr)$,
\end{itemize}
so that $\Qcal_{1,x}\times\Qcal_{2,x}\subset \Qcal_x.$

\begin{lem}\label{lem:F-unbounded}
For every $x\in X''$ and every $h\in\mathcal Q_{1,x}$ define
\[
F_x(h):=\{v(h,1)v^{-1}: v\in \Hcal_x\}.
\]
\begin{enumerate}
\item For every $h\in\mathcal Q_{1,x}$ we have $F_x(h)\subset\mathcal H_x.$
\item There exists an element $h\in \Qcal_{1,x}$ such that $F_x^\alpha(h)$ is unbounded.
\end{enumerate}
\end{lem}

\begin{proof}
Part (1) is immediate since 
$\mathcal Q_{x,1}\times\{1\}\subset \mathcal Q_x.$

We now prove part~(2). 
Let $\{v_n\}\subset \Hcal_x$ be a sequence so that $\pi_1(v_n)\to\infty,$ 
see Corollary~\ref{cor:proj-Zd}.
Let 
\[
v_n=(v_{n,1},v_{n,2})=\Bigl(r'_{n,1}s_{n,1}r_{n,1},r'_{n,2}s_{n,2}r_{n,2}\Bigr)
\]
be the Cartan decomposition of $v_n.$ Then $s_{n,1}\to\infty.$

Passing to a subsequence, if necessary, we assume that
\begin{itemize}
\item $\{r_{n,i}\}$ and $\{r'_{n,i}\}$ converge for $i=1,2,$ moreover, 
\item ${\bf P}:=\Bigl\{g\in {\bf G}_1:\text{$\{s_{n,1}^{-1}gs_{n,1}\Bigr\}$ is bounded}\}$ is a proper parabolic $\field$-subgroup of ${\bf G}_1.$   
\end{itemize}

Since $\Qcal_{1,x}$ is Zariski density in the $\field$-group ${\bf G}_1,$ there exists 
some $h\in\Qcal_{1,x}$ which does not lie in $r^{-1}{\bf P}r$ where $r_{n,1}\to r.$ 
The claim in part~(2) holds for this $h.$ 
\end{proof}

\begin{proof}[Proof of Theorem~\ref{thm:joining}: Case 1]
Let $x\in X''$, and
let $h$ and $F_x(h)$ be as in part~(2) of Lemma~\ref{lem:F-unbounded}. Suppose
 $\{(g_n,1)\}\subset F_x(h)$ is an unbounded sequence.
By part~(1) of that Lemma we have
\be\label{eq:gn-pins-inv}
(g_n,1)\in\Hcal_x\;\text{ for all $n.$}
\ee

Recall from Lemma~\ref{lem:pinsker-comp-joinning} that
\be\label{eq:case-1-pins}
\pi_{i*}(\mu_x^{\pins})=m_i\text{ for $i=1,2$.}
\ee

Since ${\bf G}_1$ is connected, simply connected, and absolutely almost simple, 
 it follows from the generalized Mautner phenomenon,
~\cite[Ch.~1, Thm.\ 2.3.1, Ch.~2, Thm.\ 7.2]{Margulis-Book}, 
that $(X_1,m_1)$ is ergodic for the action of the unbounded group $\langle \{g_n\}\rangle.$
  
This together with~\eqref{eq:gn-pins-inv} and~\eqref{eq:case-1-pins} implies that 
$\mu_x^{\pins}=m_1\times m_2,$ see e.g.\ the argument in 
Case~1 of the proof of~\cite[Prop.~4.3]{EinMoh-Joining}. 

Since $\mu(X'')>0$ and $\mu$ is $A$-ergodic, we get that $\mu=m_1\times m_2.$
\end{proof}

The rest of this section is devoted to the analysis of the following case.

\medskip
\noindent
{\bf Case 2:} Replacing $X'$ by a conull subset, which we continue to denote by $X',$ we have the following. For every $x\in X'$ there are
\begin{itemize}
\item a subfield $\field_{x}\subset \field$ and a continuous embedding $\tau_x:\field_x\to\field,$   
\item a $\field_{x}$-group ${\bf H}_{x},$ and 
\item a $\field\oplus\field$-isomorphism $\varphi_{x}:{\bf H}_{x}\times_{\Delta_{\tau_x}(\field_x)}(\field\oplus\field)\to\G_1\coprod\G_2$ where as  in~\eqref{eq:delta-tau}, $\Delta_{\tau_x}(\field_x)=\{(c,\tau_x( c)):c\in\field_x\}$
\end{itemize}
so that $\Qcal_x$ is an open subset of the image under $\varphi_x$ of $\mathbf{H}_x(\field_x)$ with the latter considered as a subset of the $\field\oplus\field$-points of $\mathbf {H}_{x}\times_{\Delta_{\tau_x}(\field_x)}(\field\oplus\field)$
using the injection of rings $\Delta_{\tau_x}:\field_x \to \field\oplus\field$.
Moreover, $\Delta_{\tau_x}(\field_x)$ is unique, and ${\bf H}_x$ and $\varphi_x$ are 
unique up to unique isomorphisms.

Let us further recall that
\be\label{eq:field-join}
\field_x=\text{the field of quotients of the ring generated by $\{{\rm tr}(\rho(g)): g\in\mathcal Q_x\}$},
\ee
where $\rho$ denotes the non-constant irreducible representation occurring
as subquotient of the adjoint representation of $\G_1^{\rm ad}.$ 

Put $E_x:=\varphi_x\Bigl({\bf H}_x(\field_x)\Bigr)\subset G_1\times G_2.$ 

\begin{prop}\label{prop:Ex-joining}\leavevmode
\begin{enumerate}
\item There is a subfield $\field'\subset\field$
and an embedding $\tau:\field'\to\field$ so that $\Delta_{\tau_x}(\field_x)=\Delta_{\tau}(\field')$ 
on a conull subset of $X$.
\item The map $x\mapsto E_x$ is an $A$-equivariant Borel map on a conull subset of $X.$
\end{enumerate}
\end{prop}

\begin{proof}
In view of~\eqref{eq:field-join} and the fact that 
$x\mapsto\Qcal_x$ is a Borel map, we get that $x\mapsto \Delta_{\tau_x}(\field_x)$ is a Borel map, 
see the proof of Lemma~\ref{lem:kx-Qx-Borel}(1).

To see the other claims in part (1), first recall that $a\Hcal_xa^{-1}=\Hcal_{ax}$ for all $a\in A$ and 
$\mu$-a.e.\ $x\in X$. 
Hence, for any $a\in A$ there exists some finite index subgroup 
$\mathcal Q_x(a)\subset\mathcal Q_x$ so that  
\be\label{eq:deep-levels-join}
a\mathcal Q_x(a)a^{-1}\subset\mathcal Q_{ax}.
\ee
Therefore, the same arguments as in the proof of Lemma~\ref{lem:kx-constant}(1) 
applies here and finishes the proof of part~(1), see~\eqref{eq:deep-levels} and~\eqref{eq:trace-field}.

We now turn to the proof of part (2). Put
\[
\G':=\mathcal R_{\field\oplus\field/\Delta_\tau(\field')}\biggl(\G_1\coprod\G_2\biggr);
\] 
This is a $\Delta_\tau(\field')$-group. 

Now, part (1), the fact that $\varphi_x$ is an isomorphism, and the universal property of the restriction of scalars functor, see~\cite[\S A.5]{CGP-PseudoRed}, imply that
\[
E_x=\biggl(\mathcal R_{\field\oplus\field/\Delta_\tau(\field')}(\varphi_x)\Bigl({\bf H}_x\Bigr)\biggr)(\Delta_\tau(\field')).
\] 
Hence, using~\cite[Ch.~1, Prop.\ 2.5.3]{Margulis-Book}, we get that $E_x$ is identified with 
the $\Delta_\tau(\field')$-points of the Zariski closure of $\Qcal_x$ in the $\Delta_\tau(\field')$-group $\G'$.

Since the map $x\mapsto\Qcal_x$ is Borel, we thus get that $x\mapsto E_x$ is a Borel map.

To see the $A$-equivariance, first recall from~\eqref{eq:deep-levels-join} 
that $a\mathcal Q_x(a)a^{-1}$ is an open 
subgroup of $\mathcal Q_{ax}$. Using~\cite[Ch.~1, Prop.\ 2.5.3]{Margulis-Book}, thus, we get that $E_{ax}$ 
is the Zariski closure of $a\mathcal Q_x(a)a^{-1}$ in $\G'(\Delta_\tau(\field')).$ 
On the other hand, this Zariski closure equals $aE_xa^{-1}$; the claim follows.    
\end{proof}

\begin{lem}\label{lem:E-inv-join}
For $\mu$-a.e.\ $x\in X$ we have $E_x\subset\Hcal_x$ and $E_x$ is not compact.
\end{lem}

\begin{proof}
We first recall from~\cite[Thm.~T]{Pr-ThmTits} that since $\H_x$ is connected, simply connected, 
and absolutely almost simple, any open and unbounded subgroup of $E_x$ equals $E_x.$
Since $\Qcal_x\subset\Hcal_x$ is an open subgroup of $E_x,$ thus, both assertions in the lemma
will follow if we show that $\Hcal_x\cap E_x$ is unbounded for $\mu$-a.e.\ $x\in X$. 

However, the proof of Corollary~\ref{cor:proj-Zd} shows that for some $\alpha\in\APhi$, we have that $\Qcal_x \cap \linv_x^{[\alpha] }$ is non-trivial. Since $x \mapsto E_x$ is an $A$ equivariant map, using Poincare recurrence as in Corollary~\ref{cor:proj-Zd} it follows that  $\Hcal_x\cap E_x$ is unbounded. 
\end{proof}

\begin{proof}[Proof of Theorem~\ref{thm:joining}, Case 2] 
The argument is similar to the  proof of Theorem~\ref{thm:measure-class}. 

\medskip

{\em Step 1.} Let
\be\label{eq:joining-pins-erg-decom}
\mu_{x}^\pins=\int_X \nu_z\operatorname{d}\!\mu_x^\pins(z)
\ee
be the ergodic decomposition of $\mu_x^\pins$ with respect to $E_x.$ 

As before, $\field\oplus\field$ is a $\Delta_\tau(\field')$-algebra. 
Put 
\[
\Gbf':=\mathcal R_{\field\oplus\field/\Delta_\tau(\field')}\biggl(\Gbf_1\coprod\Gbf_2\biggr).
\] 
This is a connected group defined over $\Delta_\tau(\field')$,~\cite[\S{A5}]{CGP-PseudoRed}.
Moreover, $\Gamma_1\times\Gamma_2$ is a lattice in $\Gbf'(\Delta_\tau(\field'))=\Gbf_1(\field)\times\Gbf_2(\field)=G_1\times G_2=G.$

Applying Theorem~B in Section~\ref{sec:ratner-sl2} we conclude 
that for $\mu_x^{\pins}$-a.e.\ $z$ the measure $\nu_z$ is described 
as follows. There exist
\begin{enumerate}
\item $\sfield_z=(\field')^{q_z}$ where $q_z=p^{n_z},$ $p={\rm char}(\field)$, and $n_z\geq 1$,
\item a connected $\Delta_\tau(\sfield_z)$-subgroup ${\bf M}_z$ of $\rcal_{\Delta_\tau(\field')/\Delta_\tau(\sfield_z)}(\Gbf')$ 
so that 
\[
{\bf M}_z(\Delta_\tau(\sfield_z))\cap (\Gamma_1\times\Gamma_2)
\] 
is Zariski dense in ${\bf M}_z,$ and
\item an element $g_z\in G_1\times G_2,$ 
\end{enumerate} 
such that $\nu_z$ is the $g_zL_zg_z^{-1}$-invariant probability Haar measure 
on the closed orbit $g_zL_z(\Gamma_1\times\Gamma_2)/(\Gamma_1\times\Gamma_2)$ with 
\[
L_z=\overline{M_z^+(\la_z)\Bigl({\bf M}_z(\Delta_\tau(\sfield_z))\cap(\Gamma_1\times\Gamma_2)\Bigr)} 
\]
where  
\begin{itemize}
\item the closure is with respect to the Hausdorff topology, and 
\item $\la_z:{\bf G}_m\to{\bf M}_z$ is a noncentral $\Delta_\tau(\sfield_z)$-homomorphism, $M_z^+(\la_z)$ is defined in~\eqref{eq:M-+-la}, and $E_x\subset M_z^+(\la_z).$   
\end{itemize} 

Arguing as in the proof of Lemma~\ref{applyingAA}, there exists a triple
$(\sfield_0,[{\bf M}_0],[M^+_{0}(\la_0)])$ so that 
\[
\Bigl(\sfield_z,[{\bf M}_z],[M^+_{z}(\la_z)]\Bigr)=\Bigl(\sfield_0,[{\bf M}_0],[M^+_{0}(\la_0)]\Bigr)\text{ for $\mu$-a.e.\ $x$ and $\mu_x^\pins$-a.e.\ $z.$}
\]
Put $L_0:=\overline{M_0^+(\la_0)\Bigl({\bf M}_0(\Delta_\tau(\sfield_0))\cap(\Gamma_1\times\Gamma_2)\Bigr)}.$

\medskip

{\em Step 2.}
One of the following holds 
\begin{enumerate}
\item[(a)] $L_0=G_1\times G_2,$ or
\item[(b)] $\pi_i(L_0)=G_i$ and $\ker(\pi_i|_{L_0})\subset C(G_1\times G_2)$ for $i=1,2.$ 
\end{enumerate} 

To see this, first note that by Lemma~\ref{lem:pinsker-comp-joinning} 
we have $\pi_{i*}\mu_x^\pins=m_i$ for $\mu$-a.e.\ $x\in X$ and $i=1,2$.
This, together with~\eqref{eq:joining-pins-erg-decom}, implies that
\[
m_i=\pi_{i*}\mu_x^\pins=\int_X \pi_{i*}\nu_z\operatorname{d}\!\mu_x^\pins(z)\text{ for $\mu$-a.e.\ $x$.}
\]

Since $\nu_z$ is invariant under $E_x,$ the projection $\pi_{i*}(\nu_z)$ is invariant under $\pi_i(E_x).$
By Lemma~\ref{lem:E-inv-join}, the group $\pi_i(E_x)$ is an unbounded subgroup of $G_i$ for $i=1,2$.
Since ${\bf G}_i$ is simply connected, $m_i$ is $\pi_i(E_x)$ 
ergodic, see~\cite[Ch.~1, Thm. 2.3.1, Ch.~2, Thm.~7.2]{Margulis-Book}. 
Therefore, 
\[
\mbox{$\pi_{i*}\nu_z=m_i$ for $\mu_x^\pins$-a.e.\ $z.$} 
\]
In particular, we get that $\pi_{i}(g_z L_0g_z^{-1})=G_i$ for $\mu_x^\pins$-a.e.\ $z$ and $i=1,2$.

Since $\G_i$ is absolutely almost simple, any proper normal 
subgroup of $G_i,$ as an abstract group, is central~\cite[Ch.~1, Thm.~1.5.6]{Margulis-Book}.
This implies that one of the following holds.
\begin{itemize}
\item $L_0=G_1\times G_2,$ or
\item $\pi_i(L_0)=G_i$ and $\ker(\pi_i|_{L_0})\subset C(G_1\times G_2)$ for $i=1,2.$ 
\end{itemize}  
as we claimed.

If $L_0=G\times G,$ we are done with the proof. 
Hence, our standing assumption for the rest of the argument is that (b) above holds.

\medskip
{\em Step 3.}
The assertion in~(b) also holds for $M_0$ and $M_0^+(\la_0)$ in place of $L_0.$

Let us first show this for $M_0.$ Since $L_0\subset M_0$ 
we have 
\[
\pi_i(M_0)=G_i\text{ for $i=1,2.$}
\]
Therefore, as above, either $M_0=G_1\times G_2$ or (b) holds for $M_0$. 
Assume to the contrary that $M_0=G_1\times G_2.$
Recall that $\la_0:\Gbf_m\to\mathbf M_0$ is a noncentral homomorphism.
Since $\G_i$ is connected, simply connected, and absolutely almost simple for $i=1,2,$
using~\cite[Ch.~1, Prop. 1.5.4 and Thm.\ 2.3.1]{Margulis-Book}, we have that either 
\begin{itemize}
\item $M_0^+(\la_0)=G_1\times G_2$, or 
\item $M_0^+(\la_0)\subset G_i$ for some $i=1,2$.
\end{itemize} 
However, since $M^+_0(\la_0)\subset L_0$, the above contradict our assumption that (b) holds.

We now turn to the proof of the claim for $M_0^+(\la_0).$ Since 
$M_0\neq G_1\times G_2$ and $M_0^+(\la_0)\subset M_0,$ 
the claim follows if we show that 
\be\label{eq:M+-proj}
\pi_{i}(M_0^+(\la_0))=G_i\text{ for $i=1,2$}.
\ee 
To see this note that $\la_0(\sfield_0^\times)\subset M_0(\la_0).$ 
Since (b) holds for $M_0,$ we have $\pi_i(\la_0(\sfield_0^\times))$ is unbounded
for $i=1,2.$ 
Therefore,~\eqref{eq:M+-proj} follows from~\cite[Ch.~1, Prop. 1.5.4 and Thm.\ 2.3.1]{Margulis-Book}. 

Let us record the following corollaries of the above discussion for later use.
Since (b) holds for $M_0^+(\la_0)$, $L_0,$ and $M_0,$ we have
\be\label{eq:norm-M0-joining}
N_{G_1\times G_2}(M_0)\subset CM_0
\ee 
where $C:= {}{Z(G_1\times G_2)}$

We also have
\be\label{eq:M+-M0-L0}
\text{$M_0^+(\la_0)$ is a finite index subgroup of $L_0$ and of $M_0.$}
\ee

\medskip
 
{\em Step 4.} 
Both 
\[
\text{$M^+_0(\la_0)(\Gamma_1\times\Gamma_2)/(\Gamma_1\times\Gamma_2)$
and $M_0(\Gamma_1\times\Gamma_2)/(\Gamma_1\times\Gamma_2)$}
\] 
are closed orbits
with probability, invariant, Haar measures.
In particular, $\nu_x$ is the Haar measure on the closed orbit 
\[
g_xM^+_0(\la_0)(\Gamma_1\times\Gamma_2)/(\Gamma_1\times\Gamma_2).
\] 

Indeed, let $\Lambda:=M_0\cap (\Gamma_1\times \Gamma_2)$.
Then by~\eqref{eq:M+-M0-L0} and {\em Step 1.}, $\Lambda$ is a lattice in $M_0$, as was claimed for $M_0.$

Using~\eqref{eq:M+-M0-L0}, once more, we have $\Lambda\cap M^+_0(\la_0)$ 
has finite index in $\Lambda.$
This implies that $\Lambda\cap M^+_0(\la_0)$ is a lattice in $M^+_0(\la_0)$, hence, the claim for 
$M^+(\la_0).$  
\medskip

{\em Step 5.} We are now in a position to finish the proof.
In view of~\eqref{eq:norm-M0-joining},~\eqref{eq:M+-M0-L0} and {\em Step 4.}, we can argue 
as in the proof of Lemma~\ref{applyingAA2}, see in particular~\eqref{eq:Borel-map}, and get the following. 
Let $C':=C\cap\Bigl(\Gamma_1\times\Gamma_2\Bigr).$
The decomposition
\[
\mu=\int\nu_x\operatorname{d}\!\mu
\]
yields the Borel map $f(x)=g_xC'M_0$ from a conull subset of $X$ to 
$G_1\times G_2/C'M_0$.
Moreover, $f$ is an $A$-equivariant map.

Hence, it follows from Lemma~\ref{lem:zd-measure1} that there exists some 
\[
g_0\in{\rm Fix}_{A^{{\rm sp}}_{\sfield_0}}\Bigl(G_1\times G_2/C'M_0\Bigr)
\] 
so that $f_*\mu$ is the $A$-invariant measure on the compact orbit $Ag_0.$

By Lemma~\ref{lem:w-pm-normal} and~\eqref{eq:M+-M0-L0} 
we have $M^+_0(\la_0)$ is a normal and finite index 
subgroup of $M_0;$ furthermore, $C'$ is a finite group.
Therefore, arguing as we did to complete the proof Theorem~\ref{thm:measure-class} after~\eqref{eq:eta-pins}, we get that there is some $g_1\in M_0$ so that
\[
\mu=\int_{A/A\cap g_0g_1M^+_0(\la_0)g_1^{-1}g_0^{-1}}a_*\nu\operatorname{d}\!a
\]  
where $\operatorname{d}\!a$ is the probability Haar measure on the compact group 
\[
A/A\cap g_0g_1M^+_0(\la_0)g_1^{-1}g_0^{-1},
\]
and $\nu$ is the is the probability Haar measure on the closed orbit 
\[
g_0g_1M^+_0(\la_0)(\Gamma_1\times\Gamma_2)/(\Gamma_1\times\Gamma_2).
\]
Hence, Theorem~\ref{thm:joining}(2) holds with $\Sigma=g_0g_1M^+_0(\la_0)g_1^{-1}g_0^{-1}$.
\end{proof}


\bibliographystyle{amsplain}
\bibliography{papers}

\end{document}